\documentclass[11pt,reqno]{amsart}
\usepackage{graphicx,enumerate,amssymb,bbold,amsmath}
\usepackage[notrig]{physics}
\usepackage{nicefrac}
\usepackage{mathtools}
\usepackage[font=footnotesize,labelfont=small]{subcaption}
\usepackage[ruled,vlined,linesnumbered]{algorithm2e}

\SetKwInOut{Parameter}{parameter}

\usepackage[left=2.5cm,right=2.5cm,top=2.5cm,bottom=2.5cm,includeheadfoot]{geometry} 

\usepackage[colorlinks=true, pdfstartview=FitV, linkcolor=blue, 
            citecolor=blue, urlcolor=blue]{hyperref}
 \usepackage{lipsum,cases}
\usepackage[htt]{hyphenat}

\numberwithin{equation}{section}
\numberwithin{figure}{section}
\theoremstyle{plain}
\newtheorem{theorem}{Theorem}[section]
\newtheorem{lemma}[theorem]{Lemma}
\newtheorem{proposition}[theorem]{Proposition}

\theoremstyle{definition}
\newtheorem{definition}[theorem]{Definition}

\newtheorem{remark}[theorem]{Remark}

\newcommand{\bitem}{\begin{itemize}}
\newcommand{\eitem}{\end{itemize}}
\newcommand{\mc}[1]{\mathcal{#1}}

\newcommand{\N}{\mathbb{N}}
\newcommand{\R}{\mathbb{R}}

\newcommand{\bpm}{\begin{pmatrix}}
\newcommand{\epm}{\end{pmatrix}}
\newcommand{\bvm}{\begin{vmatrix}}
\newcommand{\evm}{\end{vmatrix}}
\newcommand{\bsm}{\left(\begin{smallmatrix}}
\newcommand{\esm}{\end{smallmatrix}\right)}
\newcommand{\T}{\top}

\newcommand{\ol}[1]{\overline{#1}}
\newcommand{\la}{\langle}
\newcommand{\ra}{\rangle}

\newcommand{\mrm}[1]{\mathrm{#1}}

\newcommand{\veps}{\varepsilon}

\newcommand{\gdw}{\Leftrightarrow}

\DeclareMathSymbol{\mydiv}{\mathbin}{symbols}{"04}

\DeclareMathOperator{\dom}{dom}

\DeclareMathOperator{\argmin}{arg min}

\DeclareMathOperator{\dgp}{dgp}

\title[Superiorization vs. Accelerated Convex Optimization]{Superiorization vs.~Accelerated Convex Optimization: The Superiorized / Regularized Least-Squares Case}
\author[Y.~Censor, S.~Petra, C.~Schn\"{o}rr]{Yair Censor, Stefania Petra, Christoph Schn\"{o}rr}
\address[Y.~Censor]{Dept.~Mathematics, University of Haifa,  Israel} 
\email{yair@math.haifa.ac.il}
\urladdr{\url{http://math.haifa.ac.il/yair}}
\address[S.~Petra]{Mathematical Imaging Group, Heidelberg University, Germany} 
\email{petra.uni-heidelberg.de}
\urladdr{\url{https://www.stpetra.com}}
\address[C.~Schn\"{o}rr]{Image and Pattern Analysis Group, Heidelberg University, Germany} 
\email{schnoerr@math.uni-heidelberg.de}
\urladdr{\url{https://ipa.math.uni-heidelberg.de}}
\date{Submitted: November 5, 2019. Revised: March 31, 2020.} 


\begin{document}

\begin{abstract}
We conduct a study and comparison of superiorization and optimization
approaches for the reconstruction problem of superiorized / regularized
least-squares solutions of underdetermined linear equations with nonnegativity
variable bounds. Regarding superiorization, the state of the art is
examined for this problem class, and a novel approach is proposed
that employs proximal mappings and is structurally similar to the
established forward-backward optimization approach. Regarding convex
optimization, accelerated forward-backward splitting with inexact
proximal maps is worked out and applied to both the natural splitting
least-squares term / regularizer and to the reverse splitting regularizer
/ least-squares term. Our numerical findings suggest that superiorization
can approach the solution of the optimization problem and leads to
comparable results at significantly lower costs, after appropriate
parameter tuning. On the other hand, applying accelerated forward-backward
optimization to the reverse splitting slightly outperforms superiorization,
which suggests that convex optimization can approach superiorization
too, using a suitable problem splitting.
\end{abstract}

\maketitle
\tableofcontents

\section{Introduction}\label{sec:Introduction}
\subsection{Overview and Motivation}
The purpose of this work is to present a comparative study of superiorization
and accelerated inexact convex optimization.

The motto of the \textit{superiorization methodology (SM)} is to take
an iterative algorithm -- called the \textit{basic algorithm} --
which is known to converge to a point in a specified set and perturb
its iterates such that the perturbed algorithm -- called \textit{the
superiorized version of the basic algorithm} -- will still converge
to some point of the same set. Hence, these perturbations can be used
to lower the value of a given exogenous \textit{target function} without
compromising convergence of the generated sequence of iterates to
the specified set, which explains the adjective \textit{superiorized}.
In comparison to numerical iterative optimization methods, the SM
strives to return a \textit{superior} feasible point rather than an
\textit{optimal} feasible point. On the other hand, the SM only requires
minor modifications of existing code in order to perturb a basic algorithm.
In addition, the SM does not depend on what target function is chosen
for the application domain at hand. As a consequence, superiorization
provides a flexible framework for applications that has attracted
interest recently (see Subsection \ref{sec:SM-related}).

\textit{Convex optimization (CO)} is an established mature field of
research. For large-scale problems that we are mainly interested in,
iterative proximal minimization are the methods of choice based on
various problems splittings \cite[Chapter 28]{Bauschke:2010aa}, \cite{Combettes-prox-splitt}.
The most fundamental problem splitting is based on the decomposition
\begin{equation}
h\colon\R^{n}\to(-\infty,+\infty],\qquad h(x)=f(x)+g(x)\label{eq:f-f0-g}
\end{equation}
of a given convex objective function $h$ into a \textit{smooth} convex
function $f$ and a, possibly \textit{nonsmooth,} convex function
$g$, respectively, where smooth means continuously differentiable
with $L_{f}$-Lipschitz continuous gradient $\nabla f$, whereas $g$
is only required to be lower-semicontinuous. \textit{Forward-backward
splitting (FBS)} -- also called \textit{proximal gradient iteration}
-- iteratively computes for $k\in\{0,1,\dotsc\}$ with arbitrary
initial point $x_{0}\in\R^{n}$
\begin{subequations}\label{eq:intro-FBS}
\begin{align}
x_{k+1} &= P_{\alpha} g\big(x_{k}-\alpha\nabla f(x_{k})\big),
\label{eq:intro-FBS-a} \\ \label{eq:intro-FBS-b}
&= x_{k}-\alpha\psi_{\alpha}(x_{k}),
\intertext{with $\psi_{\alpha}$ given by} \label{eq:intro-FBS-c}
\psi_{\alpha}(x) 
&= \nabla f(x) + \nabla e_{\alpha}g(x-\alpha\nabla f(x)\big),\qquad \alpha \in \Big(0,\frac{2}{L_{f}}\Big),
\end{align}
\end{subequations}
where $P_{\alpha}g(\cdot)$ and $e_{\alpha}g$
denote the proximal mapping and the Moreau envelope with respect to
$g$, respectively, as defined below by \eqref{eq:def-P-lambda} and
\eqref{eq:def-Moreau-envelope} \cite[Definition 1.22]{Rockafellar:2009aa}.
The sequence $(x_{k})_{k\geq0}$ generated by \eqref{eq:intro-FBS-a}
is known to converge to an optimal point $x^{\ast}$ minimizing the
objective function $h$, provided the step-size is chosen in the interval
$\alpha\in(0,\frac{2}{L_{f}})$ \cite[Corollary 28.9]{Bauschke:2010aa}.

Algorithm \eqref{eq:intro-FBS} bears a structural resemblance to
the SM. The evaluation of the proximal map $P_{\alpha}g$ may be considered
as a `basic algorithm'\footnote{Throughout this paper, we indicate by `$\dots$' when the 
notion `basic algorithm' is used to highlight the similarity of the SM and CO. Otherwise, 
\textit{basic algorithm} without `$\dots$' refers to the technical term of the SM as 
defined in Section \ref{sec:SUPvsCO}.} 
that is perturbed in \eqref{eq:intro-FBS-a}
by the gradient $\nabla f$ steps without losing convergence. For
example, if $g=\delta_{C}(x)$ is the indicator function (see Subsection
\ref{sec:Notation} for the basic notation) of a closed
convex feasibility set $C\subset\R^{n}$, then $P_{\alpha}g=\Pi_{C}$
is just the orthogonal projection onto $C$, i.e.,~evaluating the
proximal map returns a feasible point. Regarding $f$ as the target
function then enables to interpret \eqref{eq:intro-FBS-a} as a \textit{perturbed}
version of the basic algorithm that returns a feasible point that,
not only is superior, but is rather optimal. CO provides theory that
explains how this desirable asymptotic behavior of an iterative algorithm
is achieved. In the case of the FBS \eqref{eq:intro-FBS}, the perturbed
proximal iteration \eqref{eq:intro-FBS-a} can be interpreted as stepping
along \textit{descent directions} provided the \textit{generalized
gradient mapping} $\psi_{\alpha}$ (cf.~\eqref{eq:intro-FBS-b}),
where the gradient is taken with respect to the composite objective
function $h$ \eqref{eq:f-f0-g}, and `generalized' indicates the
evaluation of the gradient $\nabla e_{\alpha}g$ of the Moreau envelope
of the nonsmooth function $g$ (cf.~\eqref{eq:intro-FBS-c}).

This structural similarity between the SM and CO that is further discussed
in Subsection \ref{sec:SM-CO-Preliminary}, motivated us to conduct a
comparative study based on a particular but representative problem
instance of the form \eqref{eq:f-f0-g}.

\subsection{Contribution and Objectives}\label{sec:Contribution}
We consider the problem
\begin{equation}\label{eq:f-LS-R}
\min_{x \in \R^{n}} h(x),\qquad
h(x) = \frac{1}{2}\|Ax-b\|^{2} + \lambda R(x)
\end{equation}
in order to conduct a comparative study of the SM and CO. Problem
\eqref{eq:f-LS-R} is representative in that it comprises a data fitting
term and a regularizing term, similar to countless variational formulations
of inverse problems for data restoration and recovery, etc. The matrix
$A\in\R^{m\times n}$ is assumed to have dimensions $m<n$ such that
an affine subspace of points $x\in\R^{n}$ exists that minimize the
first least-squares term on the right-hand side of \eqref{eq:f-LS-R}.
We refer to Subsection \ref{sec:Problem-Formulation} for further details
and assumptions regarding \eqref{eq:f-LS-R}.

For the data and functions of \eqref{eq:f-LS-R}, a superiorization
approach will not try to solve the unconstrained penalized objective
function $h$. Instead, it will amount to choosing a basic algorithm
for solving the well-known least-squares problem, e.g., \cite{Boerck1996},
and perturbing it by subgradients of $R$, which is regarded as a
target function. By contrast, regarding CO through FBS, one naturally
identifies in \eqref{eq:f-f0-g} the smooth component $f=\frac{1}{2}\|A\cdot-b\|^{2}$
and the nonsmooth component $g=R$, in order to apply iteration \eqref{eq:intro-FBS}.
Thus, from the perspective of the SM, the roles of the functions $f$
and $g$ are \textit{interchanged}: the proximal map $P_{\alpha}g$
defines the basic algorithm of the CO approach, which is perturbed by the gradient $\nabla f$ of the least-squares term $f$.

Due to the opposing roles of the two terms of \eqref{eq:f-LS-R} in
the superiorization and optimization approaches, respectively, we
also consider the alternative case in which $f:=R$ and $g:=\frac{1}{2}\|A\cdot-b\|^{2}$,
in order to make CO closer to the SM and our study more comprehensive.
However, since FBS requires $f$ to be continuously differentiable,
this requires to approximate $R$ by a $C^{1}$-function $R_{\tau}$,
which can be done in a standard way \cite[Sect.~2.8]{Auslender:2003aa},
\cite{Combettes:2018aa}, whenever a regularizing function $R$ is
used that takes the form of a support function in the sense of convex
analysis \cite[Section 8.E]{Rockafellar:2009aa} (see Section \ref{sec:Problem-Formulation}
for details). We point out that although the function $R_{\tau}$
is smooth in the mathematical sense, it is still a highly nonlinear
function from the viewpoint of numerical optimization.

Thus, we actually study the following two variants of \eqref{eq:f-LS-R},
\begin{subequations}\label{eq:hu-hc}
\begin{align}
\min_{x \in \R^{n}} h_{u}(x),\qquad
h_{u}(x) &= \frac{1}{2}\|Ax-b\|^{2} + \lambda R_{\tau}(x),
\label{eq:fu} \\ \label{eq:fc}
\min_{x \in \R^{n}} h_{c}(x),\qquad
h_{c}(x) &= \frac{1}{2}\|Ax-b\|^{2} + \lambda R_{\tau}(x) + \delta_{\R_{+}^{n}}(x),
\end{align}
\end{subequations}
where \eqref{eq:fc} additionally takes into account nonnegativity constraints $x \geq 0$ 
that are often important in applications and turn the objective function $h_{c}$ (`c' for: constrained) 
into a truly nonsmooth version of the objective function $h_{u}$ (`u' for: unconstrained).

\vspace{0.2cm}
\noindent
The \textbf{goals and contributions} of this paper are the following.
\begin{enumerate}
\item \textit{Superiorization methodology (SM)}.
\begin{enumerate}
\item Examine the state of the art regarding \textit{basic algorithms} that are adequate for applying 
the SM to \eqref{eq:hu-hc}. This concerns, in particular, the \textit{resilience} of possible basic 
algorithms with respect to perturbations that are supposed to lower the value of the target function. 
See Subsection \ref{sec:Intro-SM} and Subsection \ref{sec:pert-resilience}  for specific definitions of 
\textit{perturbation resilience} and Table \ref{tab:list-of-algorithms} on page \pageref{tab:list-of-algorithms} for a list of basic algorithms that are examined in this paper.
\item Study various strategies for generating suitable perturbations in view of the given target function (see again Table \ref{tab:list-of-algorithms}). Here we also contribute a novel strategy utilizing the proximal mapping, which further underlines the structural similarity between the SM and CO.
\end{enumerate}
\item \textit{Convex optimization (CO).} \\
Apply state-of-the-art forward-backward splitting to \eqref{eq:hu-hc}, and take into account that the summands $f, g$ of the general form \eqref{eq:f-f0-g} can be identified in two alternative ways with the terms of the specific objective functions \eqref{eq:hu-hc}, as discussed above. By ``state of the art'' we mean 
\begin{enumerate}[(a)]
\item
the application of \textit{accelerated} variants of FBS that achieve the currently best known convergences rates $h(x_{k})-h(x^{\ast}) = \mc{O}(\frac{1}{k^{2}})$ for any convex objective function $h$ in a class encompassing both objective functions $h_{u}, h_{c}$ of \eqref{eq:hu-hc};
\item
to work out the \textit{inexact} evaluation of the proximal map of the FBS-step \eqref{eq:intro-FBS-a}. This amounts to adopt an \textit{inner} iterative loop for solving the convex optimization problem \eqref{eq:def-P-lambda} that defines $P_{\alpha} g$, along with a termination criterion that allows for stopping early this inner iteration \textit{without} compromising convergence of the outer FBS iteration.
\end{enumerate}
Clearly, both measures (a) and (b) aim at maximizing computational efficiency.
\item
\textit{SM vs.~CO}. \\
The studies of (1) and (2) above enable us to compare SM and CO and to discuss their common and different aspects. The variety of aspects seems comprehensive enough to draw conclusions that should more generally hold for other convex problems of the form \eqref{eq:f-f0-g}. We also indicate further promising directions of research in view of the gap remaining between the SM and CO.
\end{enumerate}

\subsection{Related Work}
\subsubsection{Superiorization}\label{sec:SM-related}
The \textit{superiorization method (SM)} was born when the terms and notions ``superiorization''\ and
``perturbation resilience'',\ in the present context, first appeared
in the 2009 paper \cite{dhc-itor-09} which followed its 2007 forerunner
\cite{butnariu}. The ideas have some of their roots in the 2006 and
2008 papers \cite{brz06,brz08}. All these culminated in Ran Davidi's
2010 PhD dissertation \cite{davidi-phd} and the many papers since
then cited in \cite{Censor-bib}.

Introductory and advanced materials about the SM, accompanied by relevant
references, can be found in the recent papers \cite{Censor-Levi:2019,CensorDerivativeFree2019arXiv},
in particular, \cite[Section 2]{CensorDerivativeFree2019arXiv} contains the basics
of the superiorization methodology in condensed form. A comprehensive
overview of the state-of-the-art and current research on superiorization
appears in our continuously updated  online bibliography that
currently contains 110 items \cite{Censor-bib}. Research works in
this bibliography include a variety of reports ranging from new applications
to new mathematical results on the foundations of superiorization.
A special issue entitled: ``Superiorization: Theory and Applications''
of the journal Inverse Problems \cite{Sup-Special-Issue-2017} contains
several interesting papers on the theory and practice of SM, such
as \cite{Cegielski-2017}, \cite{He2017}, \cite{Reich2017}, to name
but a few. Later papers continue research on perturbation resilience,
which lies at the heart of the SM, see, e.g., \cite{Bargetz2018}. 
Another recent work attempts at analysing the behavior of the SM via the concept of concentration of measure \cite{Censor-Levi:2019}.
Efforts to understand the very good performance of the SM in
practice recently motivated Byrne \cite{Byrne:2019aa} to notice and study similarities between some SM algorithms and
optimization methods.

There are two research directions in the general area of the SM.
``Weak superiorization'' assumes that the solution set is nonempty
and  uses ``bounded perturbation resilience''. In contrast, ``strong superiorization'' 
replaces asymptotic convergence to a solution point by
$\varepsilon$-compatibility   with a specified  set and uses the notion of 
``strong perturbation resilience''. The terms weak and strong superiorization, respectively, were proposed in
in \cite[Section 6]{cz3-2015} and \cite{constanta-weak-strong}.



The recent work \cite{Zibetti:2018aa,Helou:2018aa}
applied the SM to the unconstrained linear least-squares problem. Specifically, the preconditioned conjugate gradient iteration (PCG) was used as the basic algorithm for solving the least-squares problem in the algebraic (fully-discretized) model of computed tomography (CT), and superiorized using the discretized total variation
(TV)  as a target function. The authors report that superiorized PCG compares favorably to a state of the art optimization algorithm FISTA \cite{Beck:2009aa}, that minimizes \eqref{eq:f-LS-R} by accelerated FBS based on the choice $f=\frac{1}{2}\|A \cdot -b\|^{2}$ and $g=R$ in \eqref{eq:f-f0-g}. 
The main contribution of \cite{Zibetti:2018aa} that we exploit in this paper, was to show perturbation resilience of the conjugate gradient (CG) method and hence to show that CG can be superiorized by \emph{any} target function utilized as the criterion for superiorization.
\subsubsection{Convex Optimization}
Proximal iterations based on various operator splittings
\cite{Combettes-prox-splitt}, while tolerating inexactness in terms of summable error sequences \cite{Tossings:1994aa,Combettes:2004aa}, are well-known. More recent work has focused on inexactness criteria that can be checked computationally at each iteration. In particular, research has focused
on extending the acceleration technique introduced by Nesterov \cite{Nesterov:1983aa} to
such inexact proximal schemes. We refer to \cite{Reem:2017aa} and
references therein for a recent comprehensive study and survey, including
a novel accelerated inexact forward-backward scheme for minimizing
finite convex objective functions. This scheme does not apply to \eqref{eq:fc},
however, due to the nonnegativity constraints.

We point out that problems of the form \eqref{eq:hu-hc} 
have been extensively studied from the viewpoint of convex programming; see, e.g.,~\cite{Vogel:1998aa,Beck:2009ab,Goldstein:2009aa,Chen:2012ab}. For example, FISTA \cite{Beck:2009aa} is well-known to be particularly
efficient for convex problems with sparsity-enforcing $\ell_{1}$
or TV regularizers, like $R$ in \eqref{eq:f-LS-R}, because proximal maps with respect to these regularizers,
when iterated, can be carried out efficiently by shrinkage and TV-based
denoising, respectively \cite{Goldstein:2009aa,Beck:2009ab}. 

Selecting the most efficient optimization approach is \textit{not}
the main objective of the present paper, however. Rather, by reversing the role of $f$ and $g$ in \eqref{eq:f-f0-g}, we also consider FBS iterations that better 
mimic the structure of superiorization using $R_{\tau}$
as target function, even if this approach is \textit{not}
the most efficient one for optimizing \eqref{eq:hu-hc}. 

\subsection{Organization} 
Section \ref{sec:SUPvsCO} details our assumptions about problem \eqref{eq:f-LS-R} and the smoothing operation that turns $R$ into $R_{\tau}$ so as to enable forward-backward splitting in two alternative ways (Subsection \ref{sec:Problem-Formulation}). 
Subsections \ref{sec:Intro-SM} and \ref{sec:Intro-FBS} detail the superiorization approach and the convex optimization approach adopted in this paper, both from a top-level point of view, followed by a preliminary short discussion of common aspects and conceptual differences in Subsection \ref{sec:SM-CO-Preliminary}. Subsections \ref{sec:Intro-SM} and \ref{sec:Intro-FBS} provide templates for the concrete superiorization and optimization approaches worked out in Sections \ref{sec:Superiorization} and \ref{sec:Optimization}, respectively. Extensive numerical results are reported and discussed in Section \ref{sec:NumericalResults}. We conclude in Section \ref{sec:Conclusion}.

\subsection{Preliminary Notions and Notation}\label{sec:Notation}

We set $[n]=\{1,2,\dotsc,n\}$ for any $n\in\N$. The Euclidean space is denoted by $\R^{n}$
and its nonnegative orthant by $\R_{+}^{n}$. The strictly positive real numbers are denoted by $\R_{++}=\{x\mid x>0\}$.
$\ol{\R}=[-\infty,+\infty]$ denotes the extended real line.
For a function $f \colon \R^{n} \to \ol{\R}$ the set
$\dom f := \{x \in \mathbb R^n \colon f(x) < \infty\}$ denotes
its effective domain.

We denote  by  $K^{\ast}$
the polar cone of a  cone $K\subseteq \R^n$, where $K$ is either  $\R^n$ or $\R^n_+$ in the following. 
For a closed convex set $C\subset\R^{n}$, $N_{C}(x)$ denotes the
normal cone at $x\in C$ and $\delta_{C}$ is the indicator function
of $C$, i.e.,~$\delta_{C}(x)=0$ if $x\in C$, and $\delta_{C}(x)=+\infty,$
otherwise. The orthogonal projection onto $C$ is denoted by $\Pi_{C}$.
In the specific case $C:=K:=\R^n_+$, we simply write $x_{+}=\Pi_{K}(x)$ and,
similarly, $x_{-}=\Pi_{K^{\ast}}(x)=x-x_{+}$. The identity matrix in $\R^{n\times n}$ is
denoted by $I_n$ or by $I$ if the dimension is clear from the context. The Euclidean vector norm and inner product are denoted
by $\|\cdot\|$ and $\la\cdot,\cdot\ra$, respectively. For two vectors $x,y\in\R^n$ we write $x\perp y$ whenever they are orthogonal.
For a matrix
$A\in\R^{m\times n}$, $\|A\|_{2}$ denotes its spectral norm, $A^{\top}$ the transpose of $A$ and $\|A\|$ the Frobenius norm induced by the inner product $\la A, B \ra = \tr(A^{\T} B)$, where $\tr(\cdot)$ returns the trace of a square matrix as argument. The least-squares solution of minimal Euclidean norm is denoted by $x_{\rm{LS}}$.

We assume that images are discretized on $n$ grid points in a two dimensional domain in $\R^{2}$. 
Using the one-dimensional discrete derivative operator
 \begin{equation}
 \partial_d : \R^d \to \R^{d}, \quad  \partial_d = \begin{cases} -1, & i=j<d,\\
 +1, & j=i+1,\\
 0, & \text{otherwise},
 \end{cases}
 \end{equation}
 along each spatial direction,  with $d\in\{M,N\}$,  we define the discrete gradient matrix of an $M\times N$ discrete image 
 \begin{equation}\label{eq:discrete-grad}
D =\bpm D_{1} \\D_{2}\epm = \begin{pmatrix}
 \partial_{M} \otimes I_{N}\\
I_{M} \otimes \partial_{N}
 \end{pmatrix}\in\R^{2n\times n},
 \end{equation}
 where $\otimes$ stands for the Kronecker product and $I_{M}, I_{N}$ are identity matrices of appropriate dimensions.

The following classes of convex functions are relevant to our investigation.
\begin{subequations} 
\begin{align}
\mc{F}_{c} & :=\{f\colon\R^{n}\to\ol{\R}\mid f\;\text{is convex, proper and lower semicontinuous}\},\\
\mc{F}_{c}^{1}(L) & :=\{f\in\mc{F}_{c}\mid f\in C^{1}\,\text{and \ensuremath{\nabla f} is \ensuremath{L}-Lipschitz-continuous}\},\\
\mc{F}_{c}^{1}(L,\mu) & :=\{f\in\mc{F}_{c}^{1}(L)\mid f\,\text{is \ensuremath{\mu}-strongly monotone convex}\}.
\end{align}
\end{subequations} We use subscripts $L_{f},\mu_{f}$ to specify
the corresponding concrete function $f$. $f^{\ast}$ denotes the
Legendre-Fenchel conjugate of $f\in\mc{F}_{c}$. It is well-known that $f^{\ast} \in \mc{F}_{c}$ if and only if $f\in\mc{F}_{c}$.

Given $f\in\mc{F}_{c}$ and $\alpha>0$, the \textit{proximal
mapping} of the point $x\in\R^{n}$ is defined by 
\begin{equation}\label{eq:def-P-lambda}
P_{\alpha}f(x):=\arg\min_{y\in\R^n}\Big\{ f(y)+\frac{1}{2\alpha}\|y-x\|^{2}\Big\}\in \dom f,
\end{equation}
whereas the \textit{Moreau envelope} is the function defined by 
\begin{equation}\label{eq:def-Moreau-envelope}
e_{\alpha}f(x):=\inf_{y\in\R^n}\Big\{ f(y)+\frac{1}{2\alpha}\|y-x\|^{2}\Big\}.
\end{equation}
This function is continuously differentiable with gradient \cite[Theorem 2.26]{Rockafellar:2009aa}
\begin{equation}\label{eq:nabla-Moreau}
\nabla e_{\alpha}f(x)=\frac{1}{\alpha}\big(x-P_{\alpha}f(x)\big).
\end{equation}

\section{Superiorization Versus Convex Optimization}\label{sec:SUPvsCO}

\subsection{Problem Formulation}\label{sec:Problem-Formulation}
We further specify problems \eqref{eq:hu-hc}. Throughout this paper, we assume
\begin{equation}\label{eq:ass-A}
A\in\R_{+}^{m\times n},\qquad\rank(A)=m<n.
\end{equation}
In the case of discrete tomography considered in Section \ref{sec:NumericalResults}, where $A$ represents the incidence relation of projection rays and cells covering a Euclidean domain, the full-rank condition can be ensured by, e.g., selecting a specific ray geometry, as illustrated in Section \ref{sec:NumericalResults}, or by slightly perturbing the nonzero entries in $A$ \cite{Petra2014}.

Regarding data errors, we adopt the basic Gaussian noise model 
\begin{equation}\label{eq:def-noise}
(Ax-b)_{i}\sim\mc{N}(0,\sigma^{2}),\;i\in[m].
\end{equation}

Regarding the definition
of $R_{\tau}\in\mc{F}^{1}_{c}(L_{R_{\tau}})$, we use the discrete gradient  matrix \eqref{eq:discrete-grad} and
index by $i\in[n]$ the vertices of the regular image grid.
Then $\bsm(D_{1}x)_{i}\\ (D_{2}x)_{i}\esm\in\R^{2}$ represents
the gradient at location $i$ in terms of the samples $x^{i},\,i\in[n]$
of a corresponding image function. We set 
\begin{equation}\label{eq:def-R}
R\colon\R^{n}\to\R_{+},\qquad R(x)=\sum_{i\in[n]}\big(|(D_{1}x)_{i}|+|(D_{2}x)_{i}|\big).
\end{equation}
To obtain a $C^{1}$-approximation, we note that $R$ is a support
function,  
\begin{subequations}
\begin{align}
R(x)&=\sup_{p\in C}\sum_{i\in[n]}\Big\la p_{i},\bsm(D_{1}x)_{i}\\ (D_{2}x)_{i}\esm\Big\ra,
\\
&\qquad
C=\big\{(p_1,\dots,p_i,\dots,p_n)\in\R^{2n}\mid p_i\in\R^2, \|p_i\|_\infty\le 1, i\in{[n]}\big\},
\end{align}
\end{subequations}
which enables the smooth approximation of $R$ in a standard way
\cite[Sect.~2.8]{Auslender:2003aa}, \cite{Combettes:2018aa}, by
\begin{subequations}\label{eq:def-R-tau}
\begin{align}
R_{\tau}(x) & =\tau\sum_{i\in[n]}\Big(\sqrt{1+(D_{1}x/\tau)_{i}^{2}}+\sqrt{1+(D_{2}x/\tau)_{i}^{2}}\Big)\\
 & =\sum_{i\in[n]}\Big(\sqrt{\tau^{2}+(D_{1}x)_{i}^{2}}+\sqrt{\tau^{2}+(D_{2}x)_{i}^{2}}\Big),\qquad0<\tau\ll1.\label{eq:R-tau}
\end{align}
\end{subequations} 
\begin{lemma}\label{lem:L-R-tau}
The gradient $\nabla R_{\tau}$ is Lipschitz continuous with constant
\begin{equation}\label{eq:LRtau-bound}
L_{R_{\tau}}\leq\frac{1}{\tau}\|D\|_{2}^{2}\leq\frac{8}{\tau}.
\end{equation}
\end{lemma}
\begin{proof}
We show that $\sup_{x \in \R^{n}}\|\nabla^{2}R_{\tau}(x)\|_{2}^{2} \leq \frac{8}{\tau}$ which implies \eqref{eq:LRtau-bound}. 
Denoting by $D_{1;i}, D_{2;i}$ the row vectors of the discrete gradient matrix \eqref{eq:discrete-grad}, we have 
\begin{equation}\label{eq:grad-R_tau}
\nabla R_{\tau}(x)
= \sum_{i \in [n]}\left(
\frac{\la D_{1;i}, x\ra}{\sqrt{\tau^{2}+\la D_{1;i}, x\ra^{2}}}D_{1;i}
+ 
\frac{\la D_{2;i}, x\ra}{\sqrt{\tau^{2}+\la D_{2;i}, x\ra^{2}}}D_{2;i}
\right)
\end{equation}
Consider any summand on the right-hand side of \eqref{eq:grad-R_tau} that has the form
\begin{equation}
s_{1;i}(x) D_{1;i} = \frac{\la D_{1;i},x\ra}{\sqrt{\tau^{2} + \la D_{1;i},x\ra^{2}}} D_{1;i}
\end{equation}
and $s_{2;i} D_{2;i}$ defined similarly. 
Using
\begin{equation}\label{eq:def-s-tilde-s}
\nabla s_{1;i}(x)
= \widetilde s_{1;i}(x) D_{1;i}
= \frac{1}{\sqrt{\tau^{2}+\la D_{1;i}, x\ra^{2}}}\Big(
1-\frac{\la D_{1;i}, x\ra^{2}}{\tau^{2}+\la D_{1;i}, x\ra^{2}}\Big) D_{1;i}
\end{equation}
and a similar expression for $\nabla s_{2;i}(x)=\widetilde s_{2;i}(x) D_{2;i}$, we obtain
\begin{equation}\label{eq:tilde-s-bound}
0 \leq \widetilde s_{1;i}(x) \leq \frac{1}{\tau},\qquad
0 \leq \widetilde s_{2;i}(x) \leq \frac{1}{\tau}
\end{equation}
and, consequently, with
\begin{equation}
\nabla^{2} R_{\tau}(x)
= \sum_{i \in [n]}\Big(D_{1;i}\nabla s_{1;i}(x)^{\T} + D_{2;i}\nabla s_{2;i}(x)^{\T}\Big)
\end{equation}
the expression
\begin{subequations}
\begin{align}
\|\nabla^{2} R_{\tau}(x)\|_{2}
&= \sup_{\substack{y \in \R^{n}\\ \|y\|=1}}\sum_{i \in [n]}\Big(\big\la y,D_{1;i}\nabla s_{1;i}(x)^{\T} y\big\ra
+ \big\la y, D_{2;i}\nabla s_{2;i}(x)^{\T} y \big\ra \Big)
\\
&\overset{\eqref{eq:def-s-tilde-s}}{=}
\sup_{\substack{y \in \R^{n}\\ \|y\|=1}}\sum_{i \in [n]}
\Big(\widetilde s_{1;i}(x)\la y, D_{1;i}D_{1;i}^{\T} y\ra + 
\widetilde s_{2;i}(x)\la y, D_{2;i}D_{2;i}^{\T} y\ra\Big).
\end{align}
\end{subequations}
Since all summands are nonnegative
\begin{subequations}
\begin{align}
\|\nabla^{2} R_{\tau}(x)\|_{2}
&\overset{\eqref{eq:tilde-s-bound}}{\leq}
\frac{1}{\tau}\sup_{\substack{y \in \R^{n}\\ \|y\|=1}}\Big\la 
y,\sum_{i \in [n]}(D_{1;i}D_{1;i}^{\T} + D_{2;i}D_{2;i}^{\T}) y \Big\ra
= \frac{1}{\tau}\sup_{\substack{y \in \R^{n}\\ \|y\|=1}}
\big\la y,(D_{1}^{\T} D_{1} + D_{2}^{\T} D_{2}) y \big\ra
\\
&\overset{\eqref{eq:discrete-grad}}{=} 
\frac{1}{\tau}\sup_{\substack{y \in \R^{n}\\ \|y\|=1}}
\la y, D^{\T} D y\ra
\leq \frac{1}{\tau}\|D\|_{2}^{2}
\leq \frac{8}{\tau},
\end{align}
\end{subequations}
where the last equality follows from applying the Gerschgorin circle theorem to the matrix $D^{\T} D$.
\end{proof}
\begin{lemma}\label{lem:gl-Lip-R-tau}
The function $R_{\tau}$ is globally Lipschitz continuous.
\end{lemma}
\begin{proof}
We show that $\sup_{x \in \R^{n}}\|\nabla R_{\tau}(x)\| \leq M$ for some $M>0$, which implies the claim. By \eqref{eq:grad-R_tau} and the proof of 
Lemma \ref{lem:L-R-tau} we have
\begin{equation}
\nabla R_{\tau}(x)
= \sum_{i \in [n]}\left( s_{1;i}(x) D_{1;i} + s_{2;i}(x) D_{2;i} \right),
\end{equation}
where $ |s_{1;i}(x)|\le 1$ and $ |s_{2;i}(x)|\le 1$. This gives $\|\nabla R_{\tau}(x)\|\le  \sum_{i \in [n]}\left( |D_{1;i} |+ | D_{2;i}| \right)=:M$.
\end{proof}
\subsection{The Superiorization Methodology (SM)}\label{sec:Intro-SM}
In this subsection, we briefly describe the SM in general terms and relate
it to the problems studied in this paper. Algorithm \ref{alg:SA}
will serve as a template for the concrete basic algorithms and their
superiorized versions worked out later in Section \ref{sec:Superiorization}.

Consider some given mathematically-formulated problem $\mathcal{T}$
and its solution set $\Omega_{\mc{T}}$. Typical instances of $\mathcal{T}$
are feasibility or optimization problems. Let $\mc{A}$ denote an
algorithmic operator such that, for any given $x_{0}\in\R^{n}$, it
generates sequences $\mc{X}_{\mc{A}}:=(x_{k}),$ by the iterative
process
\begin{equation}
x_{k+1}=\mc{A}(x_{k}),\qquad \forall k\geq0.\label{eq:basic-alg-1}
\end{equation}

The following definition, providing a criterion for iterative algorithms
that can be superiorized by bounded perturbations, is an adaptation
of \cite[Definition 1]{CDH_IP2010} to our notation used here.

\begin{definition}[bounded perturbation resilience] \label{definition-BPR-1}
Let $\Omega_{\mc{T}}$ be the solution set of some given problem $\mathcal{T}$
and let $\mc{A}$ be an algorithmic operator. The algorithm \eqref{eq:basic-alg-1}
is said to be bounded perturbation resilient with respect to $\Omega_{\mc{T}}$
if the following holds: If the algorithm \eqref{eq:basic-alg-1} generates
sequences $\mc{X}_{\mc{A}}$ that converge to points in $\Omega_{\mc{T}}$
for all $x_{0}\in\R^{n}$ then any sequence $\mc{Y}_{\mc{A}}=(y_{k})$,
generated by $y_{k+1}=\mc{A}(y_{k}+\beta_{k}v_{k}),$ also converges
to a point in $\Omega_{\mc{T}}$ for any $y_{0}\in\R^{n}$ provided
all terms $\beta_{k}v_{k}$ are ``bounded perturbations'' meaning
that the vector sequence $(v_{k})$ is bounded, $\beta_{k}\geq0$
for all $k\geq0$, and ${\sum_{k=0}^{\infty}}\beta_{k}<+\infty$.
\end{definition}

The objective of the SM is to transform a \textit{basic algorithm}
defined by an algorithmic opertor $\mc{A}$ into a \textit{superiorized
version of that algorithm} with respect to a \textit{target function}
$\phi$. If the basic algorithm \eqref{eq:basic-alg-1} generates
a sequence $(x_{k})$ that converges to some $x^{*}$ and is bounded
perturbation resilient, than the bounded perturbations, that typically
employ nonascending directions for the target function $\phi$, can
be used to find another point $y^{\ast}\in\Omega_{\mc{T}}$ that satisfies
$\phi(y^{\ast})\leq\phi(x^{\ast})$.

\begin{definition}[nonascending direction]\label{def:nonascend-1}
Given a function $\phi\colon\R^{n}\to\ol{\R}$ and a point $x\in\dom\phi$,
we say that $v\in\R^{n}$ is a \textit{nonascending vector for $\phi$
at} $x$, if $\|v\|\leq1$ and there is a $\ol{t}>0$ such that 
\begin{equation}
\phi(x+tv)\leq\phi(x),\quad\text{for all}\;t\in(0,\ol{t}].\label{eq:nonascend-1}
\end{equation}
\end{definition} A mathematical guarantee has not been found to date
that the overall process of the superiorized version of the basic
algorithm will not only retain its feasibility-seeking nature but
also preserve globally the target function reductions. This fundamental
question of the SM, called \textquotedblleft the guarantee problem
of the SM'', received recently partial answers in \cite{Censor-Levi:2019,cz3-2015}.
However, numerous works cited in \cite{Censor-bib} show that this
global function reduction of the SM occurs in practice in many real-world
applications.

Algorithm \ref{alg:SA} precisely defines the superiorized version
of a basic algorithm \eqref{eq:basic-alg-1} generated by $\mc{A}$. 
\begin{algorithm}
\DontPrintSemicolon
\textbf{initialization:} Set $N \in \mathbb{N}, k=0$ and pick an arbitrary initial point $y_{0} \in \mathbb{R}^{n}$.\;
\Repeat{a stopping rule is met.}{
  Given a current iterate $y_{k}$, set 
  $y_{k,0}\leftarrow y_{k}$ and pick 
  $N_{k} \in \{0, 1,2,\dotsc,N\}$.\;
  \For{$i \leftarrow 0$ \KwTo $N_{k}-1$ \label{SM-general-line-4}}{
    Pick $\beta_{k,i} \in (0,1]$ so that 
    $\sum_{k=0}^{\infty}  \sum_{i=0}^{N_{k}-1} \beta_{k,i} 
    < \infty$.\;
    Pick a nonascending direction $v_{k,i}$ with respect 
    to the target function $\phi$ at $y_{k,i}$ such that $\phi(y_{k,i}
    + \beta_{k,i} v_{k,i})\le \phi(y_{k,i})$. \;
    Compute the perturbation $y_{k,i} \leftarrow y_{k,i}
    + \beta_{k,i} v_{k,i}$.\; \label{SM-general-line-7}
  }
  Apply the basic algorithm and update
  $y_{k+1} \leftarrow \mc{A}(y_{k,N_{k}-1})$.\; \label{SM-general-line-8}
  Increment $k \leftarrow k+1$.\;}
\caption{Superiorized Version of a Basic Algorithm $\mc{A}$\label{alg:SA}}
\end{algorithm}
It interleaves the iterations of the
basic algorithmic operator $\mc{A}$ (line \ref{SM-general-line-8}
of Algorithm \ref{alg:SA}) with perturbations by a target function
reduction procedure $S$, defined by lines \ref{SM-general-line-4}-\ref{SM-general-line-7}
of Algorithm \ref{alg:SA}, and can be summarized as the iteration
\begin{equation}
y_{0}\in\R^{n},\quad y_{k+1}={\mc{A}}\big(S(y_{k})\big),\quad\forall k\geq0.\label{eq:superiorization-of-A}
\end{equation}
We observe the following.
\begin{enumerate}
\item The procedure $S$ works by applying $N_{k}$-times target function
reduction steps in an additive manner (line \ref{SM-general-line-7}).
\item The procedure $S$ also requires a summable sequence $(\beta_{k,i})$
that is, in practice, realized by choosing it to be a subsequence
of $(\gamma_{0}a^{k+i})$ where $\gamma_{0}$ and $a$ are positive
and $a<1$.
\item The perturbations by the procedure $S$ need not be applied after
\textit{each} application of the basic algorithm. This corresponds to the choice $N_k=0$. The question how
to coordinate the number of target function reduction steps and the
number of basic algorithmic operations in a superiorization algorithm
is of prime importance in practical applications. We call it the \emph{``balancing
problem of SM''.}
\end{enumerate}

In this paper we study the SM by applying Algorithm \ref{alg:SA}
with several basic algorithms for $\mc{A}$ and several target function
reduction procedures for $S$. Details are given in Section \ref{sec:Superiorization}.
In view of \eqref{eq:hu-hc}, specific basic algorithms $\mc{A}$
are used, based on the conjugate gradient (CG) iteration \cite[Sect. 8.3]{Saad:2003aa}
and the (projected) Landweber method (LW) \cite{Bauschke:1996aa},
respectively, to reduce each of the three functions \begin{subequations}\label{eq:def-gu-gc}
\begin{align}
g_{u}(x) & :=\frac{1}{2}\|Ax-b\|^{2},\label{eq:def-gu}\\
g_{u}^{\mu}(x): & =\frac{1}{2}\|Ax-b\|^{2}+\frac{\mu}{2}\|x\|^{2},\label{eq:def-gu-mu}\\
g_{c}(x) & :=\frac{1}{2}\|Ax-b\|^{2}+\delta_{\R_{+}^{n}}(x),\label{eq:def-gc}
\end{align}
\end{subequations} where the unconstrained and nonnegativity-constrained
cases are indicated by subscripts `\textit{u}' and `\textit{c}', respectively. The smooth
regularization with parameter $\mu$ in \eqref{eq:def-gu-mu} is chosen
to ensure applicability of the CG iteration as a basic algorithm to
the unconstrained least-squares problem, which achieves state of the
art performance.

In each case, Algorithm \ref{alg:SA} uses as stopping rule the notion
of $\varepsilon$-compatibility. For a given problem $\mathcal{T}$,
a \textit{proximity function} $\mathcal{P}r_{\mathcal{T}}:\mathbb{R}^{n}\rightarrow\mathbb{R}_{+}$
measures how incompatible $x$ is with it. Given an $\varepsilon>0,$
we say that $x$ is $\varepsilon$-\emph{compatible} with $\mathcal{T}$
if $\mathcal{P}r_{\mathcal{T}}\left(x\right)\leq\varepsilon$. We
look, for an $\veps>0$, a given problem $\mathcal{T}$ and a chosen
proximity function $\mathcal{P}r_{\mathcal{T}}$ at the set $\Gamma_{\veps}\subset\R^{n}$
of the form 
\begin{equation}
\Gamma_{\veps}:=\Big\{ x\in\R^{n}\mid\mathcal{P}r_{\mathcal{T}}\left(x\right)\leq\veps\Big\}.\label{eq:prox-1}
\end{equation}
We refer to such sets in the sequel as ``proximity sets''.

For the problems of reducing the three functions in \eqref{eq:def-gu-gc},
via the corresponding basic algorithms for each of them, we run the
superiorized version of each basic algorithm until a point $x^{\ast}$
in the corresponding proximity set \begin{subequations}\label{def:Gamma-sets}
\begin{align}
\Gamma_{u,\veps}:=\{x\in\R^{n}\mid g_{u}(x)\leq\veps\},\label{def:Gamma_u}\\
\Gamma_{u,\veps}^{\mu}:=\{x\in\R^{n}\mid g_{u}^{\mu}(x)\leq\veps\},\label{def:Gamma_mu_u}\\
\Gamma_{c,\veps}:=\{x\in\R^{n}\mid g_{c}(x)\leq\veps\}\label{def:Gamma_c}
\end{align}
\end{subequations} is obtained, respectively, where $\veps$ is fixed,
depending on the noise level \eqref{eq:def-noise}.

The target functions that we use, correponding to \eqref{eq:def-gu-gc},
are defined as \begin{subequations}\label{eq:def-phi-concrete} 
\begin{align}
\phi_{u}(x) & :=R_{\tau}(x),\\
\phi_{c}(x) & :=R_{\tau}(x)+\delta_{\R_{+}^{n}}(x),
\end{align}
\end{subequations} where $R_{\tau}$ is given by \eqref{eq:def-R-tau}.

\subsection{Convex Optimization (CO)}\label{sec:Intro-FBS}
Analogous to Section \ref{sec:Intro-SM}, we elaborate here on the  optimization scheme based on forward-backward splitting (FBS), to fix notation and to provide 
in Algorithm \ref{alg:FBS} a template for the concrete algorithms worked out in Section \ref{sec:Optimization}.

This algorithm is based on a decomposition \eqref{eq:f-f0-g} of a given objective function $h = f+g \in \mc{F}_{c}$, where $f$ is continuously differentiable with $L_{f}$-Lipschitz continuous gradient $\nabla f$.

\begin{algorithm}
\DontPrintSemicolon
\textbf{initialization:} Set $k=0$, $x_{0},y_{0} \in \dom g$, $t_{0}>1$, $a \in (0,2)$. Pick sequences $(\alpha_{k}) \subset \big(0,\frac{2-a}{L_{f}}\big]$, $(a_{k})\subset [a,2-\alpha_{k}L_{f}]$, $(\veps_{k})\subset \R_{++}$.\;
\Repeat{a stopping rule is met.}{
Set $z_{0} \in \R^{n}$ and $l=0$.\;
\Repeat{$z_{l}$ has error $\leq \veps_{k}$; then set $\ol{z}_{k}=z_{l}$.}{
Generate the next iterate $z_{l+1}$ of a sequence $(z_{l})$ converging to $P_{\alpha_{k}}g\big(y_{k}-\alpha_{k}\nabla f(y_{k})\big)$.\;
Increment $l \leftarrow l+1$.\;
}
Update $x_{k+1} = \ol{z}_{k}$.\;
Update $t_{k+1} = \frac{1}{2}\Big(1 + \Big(1+4\frac{a_{k}\alpha_{k}}{a_{k+1}\alpha_{k+1}} t_{k}^{2}\Big)^{1/2}\Big)$.\;
Update $y_{k+1} = x_{k+1} + \frac{t_{k}-1}{t_{k+1}}(x_{k+1}-x_{k}) + (1-a_{k})\frac{t_{k}}{t_{k+1}}(y_{k}-x_{k+1})$.\;
Increment $k \leftarrow k+1$.\;
}
\caption{Accelerated Forward-Backward Iteration with Inexact Proximal Points \label{alg:FBS}}
\end{algorithm}

Algorithm \ref{alg:FBS}, adopted from \cite{Villa:2013aa},  generalizes the basic FBS iteration \eqref{eq:intro-FBS-a} as follows.
\begin{itemize}
\item Besides the primal sequence $(x_{k})$ with step-sizes $(\alpha_{k})$, a sequence of auxiliary point $(y_{k})$ and step sizes $(t_{k})$ are used in order to accelerate FBS.
\item Acceleration is achieved 
\begin{itemize}
\item by computing proximal points of $y_{k}-\alpha_{k}\nabla f(y_{k})$ instead of $x_{k}-\alpha_{k}\nabla f(x_{k})$;
\item by inexact evaluation of the proximal mapping $P_{\alpha_{k}} g$.
\end{itemize}
\end{itemize}
Note that Algorithm \ref{alg:FBS} without acceleration reduces to the basic forward-backward iteration \eqref{eq:intro-FBS}.

Corresponding to the discussion preceding the two problems \eqref{eq:hu-hc}, the following four variants of the decomposition \eqref{eq:f-f0-g} are 
applied to \eqref{eq:hu-hc} in Section \ref{sec:Optimization}. 
The first two splittings handle the unconstrained case indicated by the subscript `u'.
\begin{subequations}\label{eq:splitting-u}
\begin{align}
h_{u}(x) &= f_{u}(x) + g_{u}(x), 
\label{eq:objective-u} \\ \label{eq:splitting-u-SM}
f_{u}(x) &= \lambda R_{\tau}(x), \qquad
g_{u}(x) = \frac{1}{2}\|A x-b\|^{2},
\\ \label{eq:splitting-u-reversed}
f_{u}(x) &= \frac{1}{2}\|A x-b\|^{2}, \qquad
g_{u}(x) = \lambda R_{\tau}(x).
\end{align}
\end{subequations}
The last two splittings handle the constrained case indicated by the subscript `c'.
\begin{subequations}\label{eq:splitting-c}
\begin{align}
h_{c}(x) &= f_{ u}(x) + g_{c}(x), 
\label{eq:objective-c} \\ \label{eq:splitting-c-SM}
f_{ u}(x) &= \lambda R_{\tau}(x), \qquad
g_{c}(x) = \frac{1}{2}\|A x-b\|^{2} + \delta_{\R_{+}^{n}},
\\ \label{eq:splitting-c-reversed}
f_{ u}(x) &= \frac{1}{2}\|A x-b\|^{2}, \qquad
g_{c}(x) = \lambda R_{\tau}(x) + \delta_{\R_{+}^{n}}.
\end{align}
\end{subequations}
The reason for considering the splittings \eqref{eq:splitting-u-reversed} and \eqref{eq:splitting-c-reversed}, in addition to \eqref{eq:splitting-u-SM}, \eqref{eq:splitting-c-SM}, is discussed in the following subsection.

\subsection{SM vs.~CO: Common Aspects and Conceptual Differences}

\label{sec:SM-CO-Preliminary} The two splittings \eqref{eq:splitting-u-SM}
and \eqref{eq:splitting-c-SM} resemble structurally the SM in that
the least-squares term can be viewed as the task to which the `basic
algorithm' is applied, by evaluating the proximal map $P_{\alpha_{k}}g$
(this does not change the set of minima of $g$). Furthermore, this
`basic algorithm' is perturbed by function reduction steps of the target
function $f$, performed by nonascending directions via negative gradients.

However, a key structural difference is that the roles of inner and
outer iterations in Algorithms \ref{alg:SA} and \ref{alg:FBS} are
interchanged. The SM uses multiple inner iterations (lines \ref{SM-general-line-4}-\ref{SM-general-line-7}
of Algorithm \ref{alg:SA}) in order to accumulate bounded perturbations,
whereas the basic algorithm $\mc{A}$ is updated in a single step
of the outer loop. By contrast, CO performs several iterations of
its `basic algorithm' of evaluating the proximal map $P_{\alpha_{k}}g$
in an inner loop, whereas the perturbation is updated in a single
step of the outer loop. Another difference is that the number $N_{k}$
of inner iterations in Algorithm \ref{alg:SA} is a tuning parameter
of the SM, whereas the corresponding number in the CO scheme follows
from a mathematical termination criterion, that ensures a sufficiently
descreasing error level so as to guarantee convergence of the outer
loop to a global minimizer of the objective function $h$.

These differences motivated us to study also the above splitting after
interchanging the roles of the functions $f$ and $g$. This defines
the two splittings \eqref{eq:splitting-u-reversed} and \eqref{eq:splitting-c-reversed}.
Adopting the viewpoint of the SM, this entails many iterative reduction
steps with respect to the target function $g$ through evaluating
the proximal mapping $P_{\alpha_{k}}g$.


\section{Superiorization}

\label{sec:Superiorization}

In this section, we present our SM algorithms, outlined in Subsection
\ref{sec:Intro-SM}, by specifying  -- cf. \eqref{eq:superiorization-of-A} --
different basic algorithmic operators $\mc{A}(\cdot)$, and
various target function reduction procedures $S(\cdot)$.
Instances of $\mc{A}$ solve problems
\eqref{eq:def-gu-gc}, whereas instances of $S$ take into account
the target functions \eqref{eq:def-phi-concrete} in order to superiorize
the basic algorithm implemented by $\mc{A}$. Algorithm \ref{alg:SA}
(see page \pageref{alg:SA}) details the interplay between $\mc{A}$
(line \ref{SM-general-line-8}) and its superiorization through $S$
(lines \ref{SM-general-line-4}-\ref{SM-general-line-7}).

Specific instances $\mc{A}_{\mrm{CG}}$, $\mc{A}_{\mrm{LW}}$ and
$\mc{A}_{\mrm{LW+}}$ of the basic algorithm $\mc{A}$ are worked
out in Subsection \ref{sec:basic-alg}, based on the conjugate gradient
(CG) iteration \cite[Sect. 8.3]{Saad:2003aa} and the (projected)
Landweber method (LW) \cite{Bauschke:1996aa}, respectively. Specific
instances $S_{\nabla}$ and $S_{\mrm{prox}}$ or $S_{\mrm{prox+}}$
of the target function reduction procedures $S$ are worked out in
Subsection \ref{sec:perturbations}, based on negative gradient steps
or on generalized gradient steps, respectively. The latter are defined
via the proximal mapping in order to handle nonsmooth convex constraints,
using a strategy which is common practice in convex optimization --
cf.~\eqref{eq:intro-FBS}.

Subsection \ref{sec:pert-resilience} is devoted to the notion of strong
perturbation resilience. Table \ref{tab:list-of-algorithms} lists
all combinations of $\mc{A}$ and $\mc{S}$ that we examine as the
SM approaches for solving the unconstrained and constrained least-squares
problems \eqref{eq:def-gu-gc} with target functions \eqref{eq:def-phi-concrete}.
These include the novel approaches that interleave either of the basic
algorithmic operators $\mc{A}_{\mrm{CG}}$, $\mc{A}_{\mrm{LW}}$ or
$\mc{A}_{\mrm{LW+}}$ with the target function reduction procedures
$S_{\mrm{prox}}$ or $S_{\mrm{prox+}}$ based on proximal mappings.

\begin{table}[h!]
\begin{center}
 \begin{tabular}{||c | c  | c | c ||} 
 \hline
Method & Basic &  Target function & Perturbation resilience\\
             & algorithm & reduction procedure &\\ [0.5ex] 
 \hline\hline
 \texttt{GradSupCG} &  ${\mc{A}}_{\mrm{CG}}$ & $S_\nabla$ from Algorithm \ref{alg:sup-classic}&\checkmark, see \cite[Thm. A.1]{Zibetti:2018aa}\\ 
 \hline
 \texttt{GradSupLW} &  ${\mc{A}}_{\mrm{LW}}$ & $S_\nabla$ from Algorithm \ref{alg:sup-classic} & \checkmark, see \cite{Guo2018}\\
 \hline
\texttt{ProxSupCG} & ${\mc{A}}_{\mrm{CG}}$ & $S_{\mrm{prox}}(x,\beta_k)$ from  \eqref{alg:sup-prox} & follows from Prop \ref{prop:prox-additive-form} and  \cite[Thm. A.1]{Zibetti:2018aa}\\
 \hline
\texttt{ProxSupLW} &  ${\mc{A}}_{\mrm{LW}}$  & $S_{\mrm{prox}}(x,\beta_k)$ from  \eqref{alg:sup-prox} & follows from \cite{Jin2016,Luo2019} \\
 \hline
\texttt{ProxCSupCG} &  ${\mc{A}}_{\mrm{CG}}$ & $S_{\mrm{prox+}}(x,\beta_k)$ from  \eqref{alg:sup-prox-plus} & open \\ 
\hline
\texttt{ProxCSupLW} &  ${\mc{A}}_{\mrm{LW}}$  & $S_{\mrm{prox+}}(x,\beta_k)$ from  \eqref{alg:sup-prox-plus} & open \\
\hline
\texttt{GradSupProjLW}  & ${\mc{A}}_{\mrm{LW+}}$  & $S_\nabla$ from Algorithm \ref{alg:sup-classic}&  \checkmark, see \cite{Jin2016,Guo2018} \\
\hline
\texttt{ProxSupProjLW} &  ${\mc{A}}_{\mrm{LW+}}$  & $S_{\mrm{prox}}(x,\beta_k)$ from \eqref{alg:sup-prox}& follows from \cite{Luo2019} \\[1ex] 
 \hline
\end{tabular}
\end{center}
\caption{List of superiorized algorithms. The left column shows the name tags used for the superiorized algorithms that are numerically evaluated in Section
\ref{sec:NumericalResults}. The second column from the left shows the basic algorithm used while the third column indicates our strategy for the target function reduction. The last
column indicates if the perturbation resilience property holds or is still open. Checkmark
means that the property has been proven.
}
\label{tab:list-of-algorithms}
\end{table}

\subsection{Basic Algorithms}\label{sec:basic-alg} 

We first consider the basic algorithms for the task of minimizing
the unconstrained least-squares objective $g_{u}$ \eqref{eq:def-gu}.
This problem always has a solution and $\min_{x\in\R^{n}}g_{u}(x)=0$
holds due to the \emph{underdetermined} full-rank matrix $A$. The
set of minimizers is an affine subspace that has the dimension of
the nullspace of $A$. In view of noisy measurements however, we merely
wish to find an approximate minimizer of $g_{u}$, one that is $\varepsilon$-compatible
with the set of minimizers, i.e., an element of $\Gamma_{u,\veps}$
\eqref{def:Gamma_u}, where $\veps$ is fixed, depending on the noise
level \eqref{eq:def-noise}. To this end we consider a basic algorithm
that monotonically decreases the function $g_{u}$.

A straightforward choice is the Landweber iteration, presented in
Algorithm \ref{eq:Landweber-basic}, which is a damped gradient descent
method with a constant step-size parameter $\gamma\in(0,2/\|A\|_{2}^{2})$,
see line \ref{LW_eval_2} of Algorithm \ref{alg:Landweber}, which
describes the basic algorithmic operator $\mc{A}_{LW}$. It makes
use of the spectral norm of $A$, which has to be computed or estimated
beforehand. We terminate the iteration when $g_{u}$ has reached an
$\varepsilon$-compatiblity dictated by the noise level.

\begin{algorithm}
\DontPrintSemicolon \textbf{initialization:} Set $k=0$, choose an
arbitrary initial point $x_{0}\in\mathbb{R}^{n}$, $\veps>0$ and
$\gamma\in(0,2/\|A\|_{2}^{2})$.\; \While{ $g(x_{k})>\veps$}{
Given a current iterate $x_{k}$ calculate $x_{k+1}={\mc{A}}_{\mrm{LW}}(x_{k},\gamma)$
by Algorithm \ref{alg:Landweber}.\; Increment $k\leftarrow k+1$.\;
} \caption{The Landweber Algorithm \label{eq:Landweber-basic}}
\end{algorithm}

\begin{algorithm}
\DontPrintSemicolon \SetKwInOut{Input}{input} \SetKwInOut{Output}{output}
\Input{Current iterate $x$.} \Output{ Updated iterate $x_{\mrm{new}}$.}
\Parameter{A step-size parameter $\gamma\in(0,2/\|A\|_{2}^{2})$.}
\Begin{ $g=A^{\top}(Ax-b)$,\; \label{LW_matrix_eval_1} $x_{\mrm{new}}=x-\gamma g$.
\label{LW_eval_2} } \caption{$x_{\mrm{new}}\leftarrow{\mc{A}}_{\mrm{LW}}(x,\gamma)$ \label{alg:Landweber}}
\end{algorithm}

The Landweber iteration can be easily extended to handle nonnegativity
constraints and be used as a basic algorithm for the task of minimizing
$g_{c}$ \eqref{eq:def-gc} by, simple to execute, projections onto
the nonnegative orthant. The resulting basic algorithm for finding
an approximate, i.e., $\varepsilon$-compatible, solution of $g_{c}$
is Algorithm \ref{eq:projLandweber-basic}. We note that the nonnegative
least-squares system might have an empty solution set as opposed
to the plain least-squares problem. Therefore, we assume that for
an appropriate choice of $\veps$ the proximity set $\Gamma_{c,\veps}$
\eqref{def:Gamma_c} is nonempty. 

\begin{algorithm}
\DontPrintSemicolon \textbf{initialization:} Set $k=0$, choose an
arbitrary initial point $x_{0}\in\mathbb{R}^{n}$, $\veps>0$ and
$\gamma\in(0,2/\|A\|_{2}^{2})$.\; \While{ $g_{c}(x_{k})>\veps$}{
Given a current iterate $x_{k}$ calculate $x_{k+1}={\mc{A}}_{\mrm{LW_{+}}}(x_{k},\gamma)$
by Algorithm \ref{alg:projLandweber}.\; Increment $k\leftarrow k+1$.\;
} \caption{The Projected Landweber Algorithm \label{eq:projLandweber-basic}}
\end{algorithm}

\begin{algorithm}
\DontPrintSemicolon \SetKwInOut{Input}{input} \SetKwInOut{Output}{output}
\Input{Current iterate $x$.} \Output{ Updated iterate $x_{\mrm{new}}$.}
\Parameter{A step-size parameter $\gamma\in(0,2/\|A\|_{2}^{2})$.}
\Begin{ $g=A^{\top}(Ax-b)$,\; \label{LW_matrix_eval_1} $x_{\mrm{new}}=\Pi_{\R_{+}^{n}}(x-\gamma g)=\max(x-\gamma g,0)$.
} \caption{$x_{\mrm{new}}\leftarrow{\mc{A}}_{\mrm{LW+}}(x,\gamma)$ \label{alg:projLandweber}}
\end{algorithm}

It is well-known that the (projected) Landweber iteration exhibits
slow convergence for ill-conditioned problems, becomes unstable in
further iterations and needs to be stopped early because of semi-convergence.
This motivates us to consider an additional basic algorithm that copes
better with ill-conditioned problems. For this purpose we adopt the
following CG algorithm, listed as Algorithm \ref{eq:CG-pert-res-basic},
which is a slight modification of \cite[Algorithm 8]{Zibetti:2018aa}.
Our choice will be justified by the good numerical results obtained
by a superiorized CG algorithm for an undersampled tomographic problem,
reported and discussed in Section \ref{sec:NumericalResults}.

\begin{algorithm}
\DontPrintSemicolon \textbf{initialization:} Set $k=0$, pick an
arbitrary initial point $x_{0}\in\mathbb{R}^{n}$ and choose $\mu>0$
small and $\veps>0$. Set $p_{0}=A^{\top}(b-Ax_{0})+\mu x_{0}$ and
$h_{0}=A^{\top}Ap_{0}+\mu p_{0}$.\; \While{ $g_{\mu}(x_{k})>\veps$}{
Given a current iterate $x_{k}$ and auxilary vectors $p_{k}$ and
$h_{k}$, \; calculate $(x_{k+1},p_{k+1},h_{k+1})={\mc{A}}_{\mrm{CG}}(x_{k},p_{k},h_{k})$
by Algorithm \ref{alg:CG-pert-res}.\; Increment $k\leftarrow k+1$.\;
} \caption{The Conjugate Gradient Algorithm \label{eq:CG-pert-res-basic}}
\end{algorithm}

\begin{algorithm}
\DontPrintSemicolon \SetKwInOut{Input}{input} \SetKwInOut{Output}{output}
\Input{Current iterates $x,p,h$.} \Output{ Updated iterates
$x_{\mrm{new}},p_{\mrm{new}},h_{\mrm{new}}$.} \Begin{ $g=A^{\top}(Ax-b)+\mu x$,\;
\label{CG_grad_update} $\beta=\la g,h\ra/\la p,h\ra$,\; $p_{\mrm{new}}=-g+\beta p$,\;
$h_{\mrm{new}}=A^{\top}Ap_{\mrm{new}}+\mu p_{\mrm{new}}$,\; \label{CG_matrix_eval_2}
$\gamma=-\la g,p_{\mrm{new}}\ra/\la p_{\mrm{new}},h_{\mrm{new}}\ra$,\;
\label{CG_dir_update} $x_{\mrm{new}}=x+\gamma p_{\mrm{new}}$.\;}
\caption{$(x_{\mrm{new}},p_{\mrm{new}},h_{\mrm{new}})\leftarrow{\mc{A}}_{\mrm{CG}}(x,p,h)$
\label{alg:CG-pert-res}}
\end{algorithm}

The updates of the current iterates in lines \ref{CG_grad_update}
and \ref{CG_dir_update} of Algorithm \ref{alg:CG-pert-res} differ
from the corresponding updates in \cite[Algorithm 8]{Zibetti:2018aa},
because in our case the CG iteration minimizes $g_{u}^{\mu}$ \eqref{eq:def-gu-mu}
and is applied to the regularized least-squares problem 
\begin{equation}\label{eq:reg-least-squares}
\min \frac{1}{2}\Big \| \bpm A\\ \sqrt{\mu} I \epm x -  \bpm b\\ 0 \epm\Big\|^2,
\end{equation}
with a small parameter $\mu>0$, 
in order to obtain a well-defined algorithm for the underdetermined
scenario \eqref{eq:ass-A} considered here. 
We show however that for an appropriate choice of parameter $\mu$ minimizing $g_{u}^{\mu}$ is equivalent to finding the minimal Euclidean norm solution of
$g_u$ \eqref{eq:def-gu}.
\begin{lemma}\label{lem:choice_of_mu} Let $A\in\R^{m\times n}$ with $\rank(A)=m<n$ and define $g_u$ and $g_{u}^{\mu}$ as in \eqref{eq:def-gu} and \eqref{eq:def-gu-mu}
respectively. If $x_{\rm{LS}}$ is the minimizer of
\begin{equation}\label{LASSO-2}
\min \|x\| \quad \text{subject to}\quad x\in\argmin_{y\in\R^n} g_u(y)
\end{equation}
then there exists a parameter $\mu=\mu(x_{\rm{LS}})\ge 0$
such that $x_{\rm{LS}}$ is also a minimizer of $g_{u}^{\mu}$ in \eqref{eq:def-gu-mu}.
\end{lemma}
\begin{proof}
Due to the assumption $\rank(A)=m$ we  always have $b\in\mc{R}(A)$. Consequently, the set of minimizers of $g_u$
is equal to the solution set of the consistent linear system $Ax=b$. Thus, problem \eqref{LASSO-2}
 is equivalent to
\begin{equation}\label{BPDN-1}
\min \|x\| \quad \text{subject to}\quad Ax=b.
\end{equation}
Now problem \eqref{BPDN-1} can be equivalently rewritten as
\begin{equation}\label{BPDN-2}
\min \frac{1}{2}\|x\|^2 \quad \text{subject to}\quad \frac{1}{2}\|Ax-b\|^2\le 0,
\end{equation}
in view of the monotonicity of $t\mapsto \frac{1}{2}t^2$ and the positivitivity of the Euclidean norm.
Clearly, $x_{\rm{LS}}$ also solves \eqref{BPDN-1} and \eqref{BPDN-2}. 
The full rank  assumption implies that the relative interior of the feasible set in  \eqref{BPDN-2} is nonemepty. Thus,
strong duality holds for \eqref{BPDN-2} and, therefore, there exists a dual solution $y^\ast$ of \eqref{BPDN-2}.
We consider  the Lagrangian of \eqref{BPDN-2} that reads
$$
L(x,y)=\frac{1}{2}\|x\|^2  + y\Big(\frac{1}{2}\|Ax-b\|^2\Big).
$$
The primal-dual optimal pair $(x_{\rm{LS}},y^\ast)$ is a saddle-point of $L$ and
 $L(x_{\rm{LS}},y^\ast)\le L(x,y^\ast)$ for all $x\in\R^n$. Thus $x\mapsto  \mu L(x,y^\ast) $ is minimized by $x_{\rm{LS}}$, which differs from the objective function 
$g_{u}^{\mu}$ in \eqref{eq:def-gu-mu} only by the scalar $\mu y^\ast$ in front of the least-squares term.
In conclusion, we can chose $\mu = 1/y^\ast$. Since $y^\ast$ is connected to $x_{\rm{LS}}$ through the optimality conditions, the parameter $\mu$ depends on $x_{\rm{LS}}$ too.
\end{proof}
Note that the CG Algorithm \ref{eq:CG-pert-res-basic}, as well as
\cite[Algorithm 8]{Zibetti:2018aa}, differ from the classic CG algorithm
for least-squares in that the gradient of the least-squares term is
evaluated at every individual iterate, and not by the, commonly used,
computationally more efficient update 
\begin{equation}
g_{k+1}=g_{k}+A^{\top}Ap_{k}+\mu p_{k},\label{classic-CG-grad-update}
\end{equation}
see \cite[Sect. 8.3]{Saad:2003aa}. As discussed in \cite{Zibetti:2018aa},
by doing so, one obtains an algorithm that performs exactly as CG,
but is resilient to bounded perturbations.

The recent work \cite{Zibetti:2018aa,Helou:2018aa} also considers
the \textit{preconditioned} CG iteration as a basic algorithm for
the SM. While in case of overdetermined systems preconditioning is
mandatory, CG without preconditioning works well in our underdetermined
scenario. However, the CG method cannot be directly applied to a consistent
linear system with nonnegativity constraints. There exist CG versions
for nonnegativity and box-constraints that use an active set strategy
but are hampered by frequent restarts of the CG iteration. A more
efficient box-constrained CG method for nonnegative matrices can be
found in \cite{Vollebregt2014}, but even its convergence theory it
still open, let alone its perturbation resilience.

As a viable alternative, we include constraints in the CG method via
superiorization, as detailed below in Subsection \ref{sec:perturbations}.

\subsection{Perturbation Resilience}

\label{sec:pert-resilience} In the sequel, let $\mc{A}$ denote either
of the operators ${\mc{A}}_{\mrm{LW}}$, ${\mc{A}}_{\mrm{LW+}}$,
${\mc{A}}_{\mrm{CG}}$ that define the basic algorithms in Algorithms
\ref{eq:Landweber-basic}, \ref{eq:projLandweber-basic} or \ref{eq:CG-pert-res-basic},
respectively. Likewise, let $\Gamma_{\veps}$ denote either of the
proximity sets $\Gamma_{u,\veps}$ of \eqref{def:Gamma_u}, $\Gamma_{c,\veps}$
of \eqref{def:Gamma_c} and $\Gamma_{u,\veps}^{\mu}$ of \eqref{def:Gamma_mu_u},
defined as the sub-level sets of the functions $g_{u}$, $g_{c}$
and $g_{u}^{\mu}$ in \eqref{eq:def-gu-gc}, respectively. We assume
that $\veps>0$ has been chosen large enough such that $\Gamma_{\veps}\neq\emptyset$.
As a consequence, for any $x_{0}\in\R^{n}$, the sequence $\mc{X}_{\mc{A}}:=(x_{k}),$
generated by by the iterative process \eqref{eq:basic-alg-1} will
terminate at a point in $x^{\ast}\in\Gamma_{\veps}$.

 Bounded perturbation resilience with respect to a nonempty solution
set, recall Definition \ref{definition-BPR-1}, is known to hold for the basic Landweber and projected Landweber
iteration, see \cite{Jin2016,Guo2018} and Table \ref{tab:list-of-algorithms}.
This implies bounded perturbation resilience of $\mc{A}_{LW}$ and
$\mc{A}_{LW+}$ .

A notion of perturbation resilience that is more relevant to concrete
applications considers algorithmic \textit{termination} instead of
asymptotic convergence of the basic algorithm and its superiorized
version. Termination is defined in terms of the $\veps$-output of
a sequence. \begin{definition}[$\veps$-output of a sequence with
respect to $\Gamma_{\veps}$]\label{def:eps-output} For some $\varepsilon>0$,
a nonempty proximity set $\Gamma_{\veps}$ and a sequence $\mc{X}:=(x_{k})$
of points, the $\veps$-output $O\left(\Gamma_{\veps},\mc{X}\right)$
of the sequence $\mc{X}$ with respect to $\Gamma_{\veps}$ is defined
to be the element $x_{k}$ with smallest $k\in\N$ such that $x_{k}\in\Gamma_{\veps}$.
\end{definition}

\begin{remark} Due to our assumption that $\Gamma_{\veps}\neq\emptyset$,
the $\veps$-output $O\left(\Gamma_{\veps},\mc{X_{\mc{A}}}\right)$
of the sequences $\mc{X}_{\mc{A}}$ exists, for either instance ${\mc{A}}_{\mrm{LW}}$,
${\mc{A}}_{\mrm{LW+}}$ or ${\mc{A}}_{\mrm{CG}}$ of ${\mc{A}}$,
with respect to either proximity set in \eqref{def:Gamma-sets}. If
$\mc{X}_{\mc{A}}$ is an \textit{infinite} sequence generated by the
iterative process \eqref{eq:basic-alg-1}, then $O\left(\Gamma_{\veps},\mc{X_{\mc{A}}}\right)$
is the \textit{output} produced by that algorithm when we add to it
a stopping rule based on $\Gamma_{\veps}$, as done in Algorithms
\ref{eq:Landweber-basic}, \ref{eq:projLandweber-basic} and \ref{eq:CG-pert-res-basic}.
\end{remark} We define the notion of strong perturbation resilience
from \cite[Subsection II.C]{super-med-phys12}, adapted to our notation.
\begin{definition}[strong perturbation resilience] Let $\Gamma_{\veps}$
denote the proximity set and let $\mc{A}$ be an algorithmic operator
as in \eqref{eq:basic-alg-1}. The algorithm \eqref{eq:basic-alg-1}
is said to be strongly perturbation resilient if the following conditions
hold:
\begin{enumerate}
\item there is an $\veps>0$ such that the $\veps$-output $O\left(\Gamma_{\veps},\mc{X_{\mc{A}}}\right)$
of the sequence $\mc{X}_{\mc{A}}$ exists, for every $x_{0}\in\R^{n}$;
\item for all $\veps\geq0$ such that $O\left(\Gamma_{\veps},\mc{X}_{\mc{A}}\right)$
is defined for every $x_{0}\in\R^{n}$, we also have that $O\left(\Gamma_{\veps'},\mc{Y}_{\mc{A}}\right)$
is defined, for every $\veps'\geq\veps$, and for every sequence $\mc{Y}_{\mc{A}}=(y_{k})$
generated by 
\begin{equation}
y_{k+1}=\mc{A}(y_{k}+\beta_{k}v_{k}),\qquad\forall k\geq0,
\end{equation}
where the terms $\beta_{k}v_{k}$ are bounded perturbations as specified
by Definition \ref{definition-BPR-1}. 
\end{enumerate}
\end{definition} \textbf{}

Note that if a problem $\mathcal{T}$ has a nonempty solution set
$\Omega_{\mc{T}}\ne\emptyset$ which is contained in a proximity set
$\Omega_{\mc{T}}\subset\Gamma_{\veps}$, then bounded perturbation
resilience implies strong perturbation resilience. Sufficient conditions
for strong perturbation resilience appeared in \cite[Theorem 1]{super-med-phys12}.
%
%
We note that $A_{CG}$ has been shown to be strongly perturbation
resilient in \cite[Thm. A.1]{Zibetti:2018aa}.


\subsection{Superiorization by Bounded Perturbations}

\label{sec:perturbations}

In this subsection we specify in detail the superiorized versions
of the basic algorithms \ref{eq:Landweber-basic}, \ref{eq:projLandweber-basic}
and \ref{eq:CG-pert-res-basic} discussed above that fit into the
general framework of Algorithm \ref{alg:SA}, see also \cite[page 5537]{super-med-phys12}.
Since we already fixed the algorithmic operators to be either $\mc{A}_{\mrm{LW}}$,
$\mc{A}_{\mrm{LW+}}$ or $\mc{A}_{\mrm{CG}}$ we need to specify the
target function reduction procedures $S$, that we will use to compute
the perturbations of the basic algorithms.

\subsubsection{Target Function Reduction Procedures}

The general target function reduction procedure described in Algorithm
\ref{alg:SA}, lines \ref{SM-general-line-4}-\ref{SM-general-line-7},
can be easily adapted for reducing the differentiable target function
$\phi_{u}$ \eqref{eq:def-phi-concrete} using nonascent directions
based on normalized negative gradients. This leads to Algorithm \ref{alg:sup-classic},
that is similar to \cite[Algorithm 1]{Helou:2018aa}. 
\begin{algorithm}
\DontPrintSemicolon
\SetKwInOut{Input}{input}
\SetKwInOut{Output}{output}
\Input{Current iterate $x$, current exponent $\ell$. }
\Output{Superior iterate $x_{\mrm{new}}$, updated exponent $\ell_{\mrm{new}}$.}
\Parameter{The parameters  $\ell$,  $a$, $\gamma_0$, $\kappa$.}
\Begin{
$y \leftarrow x$.\;
\For{$i=1,\dots , \kappa$}{
        \eIf{$\nabla R_\tau(y) \ne 0$}{
        $v \leftarrow -\nabla R_\tau(y)/\|\nabla R_\tau(y)\|$, see \eqref{eq:grad-R_tau},\;
         }
         {  $v \leftarrow 0$;}
	\Repeat{
	$R_\tau(y_{\mrm{new}})\le R_\tau(y)$.
	}{
	$\gamma \leftarrow \gamma_0 a^\ell$\;
	$y_{\mrm{new}} \leftarrow y + \gamma v$ \label{eq:sug-grad-additive-step}\;
	$\ell\leftarrow \ell+1$
	}
	$ y \leftarrow y_{\mrm{new}}. $
	}
        $ \ell_{\mrm{new}} \leftarrow \ell $,\;
        $ x_{\mrm{new}} \leftarrow y $.\;
}
\caption{$(x_{\mrm{new}},\ell_{\mrm{new}}) \leftarrow S_\nabla (x,\ell, a,\gamma_0,\kappa)$
\label{alg:sup-classic}
}
\end{algorithm}
Note that $S_{\nabla}$ repeats $\kappa$-times steps along normalized
negative gradients (line \ref{eq:sug-grad-additive-step}) to reduce
the target function $\phi_{u}$. Each normalized negative gradient
direction is scaled by parameter $\gamma$ of the form $\gamma_{0}a^{\ell}$
for some $\ell\in\N$ that as specified by procedure $S_{\nabla}$.
As a consequence, perturbations can become very small since $a\in(0,1)$.

We consider a second variant of the target function reduction procedure
$S$ that allows better control of the perturbation parameters $\beta_{k,i}$
from Algorithm \ref{alg:SA} line \ref{SM-general-line-7}. For the
target function $\phi_{u}$ \eqref{eq:def-phi-concrete}, we define
\begin{equation}\label{alg:sup-prox}
S_{\mrm{prox}}(x,\beta):=P_{\beta}\phi_{u}(x)=\arg\min_{z}\Big\{ R_{\tau}(z)+\frac{1}{2\beta}\|z-x\|^{2}\Big\}.
\end{equation}

For the reduction of the target function $\phi_{c}$ \eqref{eq:def-phi-concrete},
we define 
\begin{equation}\label{alg:sup-prox-plus}
S_{\mrm{prox+}}(x,\beta):=P_{\beta}\phi_{c}(x)=\arg\min_{z}\Big\{ R_{\tau}(z)+\delta_{\R_{+}^{n}}(z)+\frac{1}{2\beta}\|z-x\|^{2}\Big\},
\end{equation}
which allows us to incorporate the constraints in the target function
reduction procedure.

Perturbation by proximal points has been done in \cite{Luo2019},
as well as in \cite{Helou:2017aa}, as we became aware while preparing
this manuscript. However, the authors in \cite{Luo2019} do not take
into account constraints when calculating the perturbations.

One might also argue that computing such nonascent directions is expensive
and contradicts the spirit of the SM. However, the proximal point
in \eqref{alg:sup-prox} or \eqref{alg:sup-prox-plus} can be computed
efficiently in our case, e.g., by the box-constrained L-BFGS method
from \cite{LBFGS-B} that computes a highly accurate solution within
few iterations, as we demonstrate in Section \ref{sec:NumericalResults}.

Clearly, if $y_{k+\frac{1}{2}}:=S_{\mrm{prox}}(y_{k},\beta_{k})$,
then \begin{subequations} 
\begin{align}
\phi_{u}(y_{k+\frac{1}{2}}) & \le\phi_{u}(y_{k+\frac{1}{2}})+\frac{1}{2\beta_{k}}\|y_{k+\frac{1}{2}}-y_{k}\|^{2}\\
 & \le\phi_{u}(y_{k})+\frac{1}{2\beta_{k}}\|y_{k}-y_{k}\|^{2}=\phi_{u}(y_{k}),
\end{align}
\end{subequations} and equality holds if and only if $y_{k+\frac{1}{2}}=y_{k}$,
in view of the definition of the proximal mapping.

Furthermore, if $y_{k+\frac{1}{2}}:=S_{\mrm{prox+}}(y_{k},\beta_{k})$,
then \begin{subequations} 
\begin{align}
\phi_{c}(y_{k+\frac{1}{2}}) & \le\phi_{u}(y_{k+\frac{1}{2}})+\delta_{\R_{+}^{n}}(y_{k+\frac{1}{2}})+\frac{1}{2\beta_{k}}\|y_{k+\frac{1}{2}}-y_{k}\|^{2}\\
 & \le\phi_{u}(y_{k})+\delta_{\R_{+}^{n}}(y_{k})+\frac{1}{2\beta_{k}}\|y_{k}-y_{k}\|^{2}=\phi_{c}(y_{k}).
\end{align}
\end{subequations} Therefore, the \emph{decrease} of the target function
by applying $S_{\mrm{prox}}$ or $S_{\mrm{prox+}}$ is guaranteed
in either case.
\begin{proposition}\label{prop:prox-additive-form} Consider any basic operator $\mc{A}$  and a summable nonnegative sequence $(\beta_k)$, i.e., 
$\beta_k\ge 0$ and $\sum_{k=0}^\infty \beta_k< +\infty$. Let $\phi:=\phi_c$ from \eqref{eq:def-phi-concrete} and define the target function reduction procedure by $S(x,\beta_{k}) :=P_{\beta_k}\phi (x)$. Then the perturbations 
$y_{k+\frac{1}{2}} = S(y_k,\beta_{k}) $ of the iterates $y_k$ generated by the superiorized version of the basic algorithm \eqref{eq:superiorization-of-A}
are bounded perturbations, i.e.,
\begin{equation}
y_{k+\frac{1}{2}} = y_{k} + \beta_{k} v_k,
\end{equation}
where the vector sequence $(v_k)$ is bounded.
\end{proposition}
\begin{proof} By definition we have $y_{k+\frac{1}{2}} = P_{\beta_k}\phi (y_k)$. Thus, by \eqref{eq:nabla-Moreau} we have
\begin{equation}\label{eq:prox-BP-1}
- \beta_{k} v_{k} := \beta_{k}\nabla e_{\beta_{k}}\phi(y_{k})
= y_{k}-P_{\beta_{k}}\phi(y_{k})
= y_{k} -  y_{k+\frac{1}{2}}.
\end{equation}
Using the variational inequality that characterizes $y_{k+\frac{1}{2}}$, gives
\begin{equation}
\frac{1}{\beta_{k}}\la y_{k}-y_{k+\frac{1}{2}},y-y_{k+\frac{1}{2}}\ra + \phi(y_{k+\frac{1}{2}})-\phi(y) \leq 0,\qquad \forall y \in\dom\phi.
\end{equation}
Setting $y=y_{k}$, we get
\begin{equation}\label{eq:prox-BP-2}
\|y_{k}-y_{k+\frac{1}{2}}\|^{2}
\leq \beta_{k}|\phi(y_{k})-\phi(y_{k+\frac{1}{2}})| \leq M \beta_{k}\|y_{k}-y_{k+\frac{1}{2}}\|,
\end{equation}
for some $M>0$. The second inequality above follows from the Lipschitz continuity of $\phi$ in view of Lemma \ref{lem:gl-Lip-R-tau}.
Now \eqref{eq:prox-BP-1} and \eqref{eq:prox-BP-2} imply
$
\|v_k\|\le M.
$
\end{proof}

\subsubsection{The Superiorized Versions of the Basic Algorithms}

Taking these considerations into account, we investigate the following
superiorized versions of the basic algorithms that interleave either
${\mc{A}}_{\mrm{LW}}$, ${\mc{A}}_{\mrm{LW+}}$ or ${\mc{A}}_{\mrm{CG}}$
with $S_{{\nabla}}$, $S_{\mrm{prox}}$ or $S_{\mrm{prox+}}$, as
introduced above.

We start with the Landweber and projected Landweber algorithms and
present in Algorithm \ref{eq:superiorized-grad-LW} the superiorized
version of the Landweber algorithm, called \texttt{GradSupLW}, that
combines ${\mc{A}}_{\mrm{LW}}$ from Algorithm \ref{alg:Landweber}
with the target function procedure $S_{\nabla}$ introduced in Algorithm
\ref{alg:sup-classic}. 
\begin{algorithm}
\DontPrintSemicolon
\textbf{initialization:} Set $k=0$ and pick an arbitrary initial point $y_{0} \in \mathbb{R}^{n}$ and parameters $\gamma_0>0$, $a\in(0,1)$ and $\veps>0$. Choose
 $\gamma \in(0,2/\|A\|^2)$.\;
\While{ $g (y_{k}) > \veps$}{
  Given a current iterate $y_{k}$, calculate the superiorized sequence\;
  by $(y_{k+\frac{1}{2}},\ell_{k+1}) = S_\nabla (y_k,\ell_k,a,\gamma_0,\kappa)$ \text{using Algorithm} ~\ref{alg:sup-classic}. \label{eq:CG-nonsascent-step}\; 
  Apply the basic algorithm \ref{alg:Landweber} and update
  $y_{k+1} = {\mc{A}}_{\mrm{LW}} (y_{k+\frac{1}{2}},\gamma)$\;
  and $\beta_{k+1}= \gamma_0 a^{k+1}$.\;
  Increment $k \leftarrow k+1$.\;
}
\caption{\texttt{GradSupLW}
\label{eq:superiorized-grad-LW}
}
\end{algorithm}

Furthermore, we present in Algorithm \ref{eq:superiorized-prox-plus-LW}
a second superiorized version of the Landweber algorithm, called \texttt{ProxCSupLW},
by combining the basic algorithmic operator ${\mc{A}}_{\mrm{LW}}$
from Algorithm \ref{alg:Landweber} with the target function reduction
procedure $S_{\mrm{prox+}}$ \eqref{alg:sup-prox-plus}. 
\begin{algorithm}
\DontPrintSemicolon
\textbf{initialization:} Set $k=0$ and pick an arbitrary initial point $y_{0} \in \mathbb{R}^{n}$ and parameters $\gamma_0>0$, $a\in(0,1)$ and $\veps>0$. Set $\beta_0=\gamma_0$
 and $\gamma \in(0,2/\|A\|^2)$.\;
\While{ $g (y_{k}) > \veps$}{
  Given a current iterate $y_{k}$, calculate the superiorized sequence\;
  by $y_{k+\frac{1}{2}} = 
  S_{\mrm{prox+}} (y_k,\beta_k)$, see ~\eqref{alg:sup-prox-plus}.\; 
  Apply the basic algorithm \ref{alg:Landweber} and update
  $y_{k+1} = {\mc{A}}_{\mrm{LW}} (y_{k+\frac{1}{2}},\gamma)$\;
  and $\beta_{k+1}= \gamma_0 a^{k+1}$.\;
  Increment $k \leftarrow k+1$.\;
}
\caption{\texttt{ProxCSupLW}
\label{eq:superiorized-prox-plus-LW}
}
\end{algorithm}
In view of the perturbation resilience of the Landweber Algorithm
\ref{eq:Landweber-basic}, the iterates of \texttt{GradSupLW} and
\texttt{ProxCSupLW} are guaranteed to converge to least-squares solutions.
On the other hand, the analogue combination of ${\mc{A}}_{\mrm{LW+}}$
with the target function reduction procedure $S_{\mrm{prox}}$ produces
a sequence that converges to a nonnegative least-squares solutions.
Despite their different theoretical behavior, these last two superiorized
versions of the Landweber algorithm generate sequences that show a
nearly identical numerically behavior, as demonstrated in Section
\ref{sec:NumericalResults}.

Proposition \ref{prop:rel-LW-PG} below shows the relation of algorithm
\texttt{ProxCSupLW} with an FBS iteration, recall \eqref{eq:intro-FBS}.
We point out that the parameter choice assumed below no longer qualifies
\texttt{ProxCSupLW} as a \textit{superiorized} Landweber algorithm,
since the sequence $(\beta_{k})$ then fails to be summable. However,
the modified algorithm converges and is guaranteed to return an optimal
solution. \begin{proposition}\label{prop:rel-LW-PG} Let $\gamma\in(0,2/L_{g_{u}})$
and $\lambda>0$. The algorithm obtained by setting $\veps=0$, $\gamma_{0}=\lambda\gamma$
and $a=1$ in the iteration procedure described in Algorithm \ref{eq:superiorized-prox-plus-LW}
generates a sequence $(y_{k+\frac{1}{2}})$ that converges to 
\begin{equation}
\underset{x\in\R^{n}}{\arg\min}\Big\{\lambda R_{\tau}(x)+\frac{1}{2}\|Ax-b\|^{2}+\delta_{\R_{+}^{n}}(x)\Big\}.\label{def-obj-1-2}
\end{equation}
\end{proposition} \begin{proof} We show that the iteration described
in Algorithm \ref{eq:superiorized-prox-plus-LW} and the FBS iteration,
recall \eqref{eq:intro-FBS}, 
\begin{equation}
x_{0}\in\R^{n}\quad x_{k+1}=P_{\gamma}\phi_{c}(x_{k}-\gamma\nabla g_{u}(x_{k})),\label{standard-PG}
\end{equation}
with $\phi_{c}=\lambda R_{\tau}+\delta_{\R_{+}^{n}}\in\mc{F}_{c}$,
$g_{u}=\frac{1}{2}\|A\cdot-b\|^{2}\in\mc{F}_{c}^{1}(L)$, are equivalent
for an appropriate choice of the parameter $\beta_{k}$. The iterates
$(x_{k})$ generated by \eqref{standard-PG} are converging to the
minimizer of $\phi_{c}+g_{c}$, the objective in \eqref{eq:fc}, see
\cite{Combettes-prox-splitt}, provided that
\begin{equation}
0<\gamma<\frac{2}{L_{g_{u}}},\label{g0-Lipschitz-1}
\end{equation}
where $L_{g_{u}}$ denotes the global Lipschitz constant of the gradient
of the least-squares term $g_{u}$ and that it can be estimated by
\begin{equation}
L_{g_{u}}\le\|A\|_{2}^{2}.\label{g0-Lipschitz-2}
\end{equation}
Furthermore, Fermat's optimality condition for problem \eqref{def-obj-1-2}
reads 
\[
0\in\lambda\nabla R_{\tau}(\bar{x})+A^{\top}(A\bar{x}-b)+\partial\delta_{\R_{+}^{n}}(\bar{x}).
\]
As done in \cite[Thm. 3.8]{Luo2019} in a different context, we introduce
an auxiliary variable and write \begin{subequations}\label{eq:rel-FB}
\begin{align}
0 & \in\lambda\nabla R_{\tau}(\bar{x})+\frac{1}{\gamma}(\bar{x}-x)+\partial\delta_{\R_{+}^{n}}(\bar{x}),\label{eq:rel-FB-a}\\
x & =\bar{x}-\gamma A^{\top}(A\bar{x}-b).
\end{align}
\end{subequations} Using the fact that \eqref{eq:rel-FB-a} is equivalent
to 
\[
\bar{x}=P_{\lambda\gamma}(R_{\tau}+\delta_{\R_{+}^{n}})(x)=\arg\min_{z}\Big\{ R_{\tau}(z)+\delta_{\R_{+}^{n}}(z)+\frac{1}{2\lambda\gamma}\|z-x\|^{2}\Big\},
\]
we recast \eqref{eq:rel-FB} as the iterative process \begin{subequations}\label{eq:rel-FB-2}
\begin{align}
\bar{x}_{k} & =P_{\lambda\gamma}(R_{\tau}+\delta_{\R_{+}^{n}})(x_{k}),\\
x_{k+1} & =\bar{x}_{k}-\gamma A^{\top}(A\bar{x}_{k}-b).
\end{align}
\end{subequations} Setting $y_{k+\frac{1}{2}}:=\bar{x}_{k}$, $y_{k}=x_{y}$
and $\beta_{k}:=\lambda\gamma$, the iterative process in \eqref{eq:rel-FB-2}
takes the form 
\begin{align}
y_{k+\frac{1}{2}} & =S_{\mrm{prox+}}(y_{k},\beta_{k}),\\
\beta_{k+1} & =\beta_{k}a,\\
y_{k+1} & =\mc{A}_{\mrm{LW}}(y_{k+\frac{1}{2}}),
\end{align}
and coincides with the iteration in Algorithm \ref{eq:superiorized-prox-plus-LW}
for $a=1$. On the other hand, the iterative process in \eqref{eq:rel-FB-2}
can be summarized by 
\[
\bar{x}_{k}=P_{\lambda\gamma}(R_{\tau}+\delta_{\R_{+}^{n}})(\bar{x}_{k-1}-\gamma A^{\top}(A\bar{x}_{k-1}-b)),\quad k\ge1,
\]
which is the FBS iteration \eqref{standard-PG}. Hence, $(\bar{x}_{k})$
and thus $(y_{k+\frac{1}{2}})$ converge to a minimizer of \eqref{def-obj-1-2}.
\end{proof}


Combinations between the Landweber and projected Landweber algorithms
and the suggested choices of target function reduction procedures
lead to different SM algorithms, that are summarized in Table \ref{tab:list-of-algorithms}.

Algorithm \ref{eq:superiorized-grad-CG} presents the superiorized
version of the CG algorithm, called \texttt{GradSupCG}, taken from
\cite[Algorithm 7, Algorithm 8]{Zibetti:2018aa}, but applied here
to the regularized least-squares problem corresponding to $g_{u}^{\mu}$
in \eqref{eq:def-gu-mu}. 
\begin{algorithm}
\DontPrintSemicolon
\textbf{initialization:} Set $k=0$ and pick an arbitrary initial point $y_{0} \in \mathbb{R}^{n}$ and parameters $\kappa\in\N$, $\gamma_0>0$, $a\in(0,1)$ and $\veps>0$. Set $\ell_0=0$,
 $p_{0}=A^\top (b-A x_{0}) + \mu x_{0}$ and
$h_{0} = A^\top A p_{0} +\mu p_{0}$.\;
\While{ $g _\mu(y_{k}) > \veps$}{
  Given a current iterate $y_{k}$, set $N_k:=\kappa$ and calculate the superiorized sequence\;
  by $(y_{k+\frac{1}{2}},\ell_{k+1}) = S_\nabla (y_k,\ell_k,a,\gamma_0,\kappa)$ \text{using Algorithm} ~\ref{alg:sup-classic}. \label{eq:CG-nonsascent-step}\; 
  Apply the basic algorithm and update
  $(y_{k+1},p_{k+1},h_{k+1})  = {\mc{A}}_{\mrm{CG}} (y_{k+\frac{1}{2}},p_k,h_k)$.\;
  Increment $k \leftarrow k+1$.\;
}
\caption{\texttt{GradSupCG}
\label{eq:superiorized-grad-CG}
}
\end{algorithm}

A second superiorized version of the CG algorithm, called \texttt{ProxCSupCG},
presented in Algorithm \ref{eq:superiorized-prox-CG}, employs the
target function reduction procedure $S_{\mrm{prox+}}$ by proximal
points.

\begin{algorithm}
\DontPrintSemicolon
\textbf{initialization:} Set $k=0$ and pick an arbitrary initial point $y_{0} \in \mathbb{R}^{n}$ and parameters $\kappa\in\N$, $\gamma_0>0$, $a\in(0,1)$  and $\veps>0$. Set $\beta_0=\gamma_0$,
 $p_{0}=A^\top (b-A x_{0}) + \mu x_{0}$ and
$h_{0} = A^\top A p_{0} +\mu p_{0}$.\;
\While{ $g (y_{k}) > \veps$}{
  Given a current iterate $y_{k}$, calculate the superiorized sequence\;
  by $y_{k+\frac{1}{2}} = 
  S_{\mrm{prox+}} (y_k,\beta_k)$, see ~\eqref{alg:sup-prox-plus}.\; \label{eq:superiorized-prox-CG-line}
  Apply the basic algorithm and update
  $(y_{k+1},p_{k+1},h_{k+1})  = {\mc{A}}_{\mrm{CG}} (y_{k+\frac{1}{2}},p_k,h_k)$\;
  and $\beta_{k+1}= \gamma_0 a^{k+1}$.\;
  Increment $k \leftarrow k+1$.\;
}
\caption{\texttt{ProxCSupCG}
\label{eq:superiorized-prox-CG}
}
\end{algorithm}

Omitting the nonnegativity constraints by replacing $S_{\mrm{prox+}}$
with $S_{\mrm{prox}}$ in line \ref{eq:superiorized-prox-CG-line}
of Algorithm \ref{eq:superiorized-prox-CG} leads to a slightly different
algorithm that we call \texttt{ProxSupCG}. Clearly the sequence $(\beta_{k})$
is summable for $a\in(0,1)$. 
It is not straightforward to show that \texttt{ProxSupCG} converges to a minimizer of $g^{\mu}_u$,
by following the lines of the proof from \cite{Luo2019}. However, it follows from Proposition \ref{prop:prox-additive-form} and  \cite[Theorem A.1]{Zibetti:2018aa}
that  \texttt{ProxSupCG} terminates at a point in $\Gamma^\mu_{u,\veps}$ \eqref{def:Gamma_mu_u}.
To show that Algorithm \ref{eq:superiorized-prox-CG} also terminates with an $\veps$-compatible point we need a counterpart of Proposition \ref{prop:prox-additive-form}
for the \emph{constrained} target function $\phi_c$ \eqref{eq:def-phi-concrete}. We leave
this open problem for future work.


\section{Optimization}\label{sec:Optimization}

In Subsections \ref{sec:basic-FB} and \ref{sec:Inner-Loop} we work out the 
details of the convex optimization algorithms for solving problems \eqref{eq:objective-u} and \eqref{eq:objective-c}, based on the splittings \eqref{eq:splitting-u-SM} and \eqref{eq:splitting-c-SM}, in Sections \ref{sec:basic-FB} and \ref{sec:Inner-Loop}. Plain, accelerated and inexact versions of the FB-iteration are considered -- cf.~\eqref{eq:intro-FBS} and Algorithm \ref{alg:FBS}. 
The reverse splittings \eqref{eq:splitting-u-reversed} and \eqref{eq:splitting-c-reversed} are briefly considered in Subsection \ref{sec:reverse-FB}. Here, we confine ourselves to exact evaluations of the corresponding proximal mapping.

Throughout this section, we use the symbols $\phi$ and $\phi_{0}$ as shorthand for the objective functions that define the proximal mapping \eqref{eq:Palpha-gc}, for the constrained and unconstrained cases, respectively. This should not be confused with the usage of $\phi$ for denoting any given target function in Algorithm \ref{alg:SA}.

\subsection{Proximal Map, Inexact Evaluation} \label{sec:basic-FB}
In this subsection, we consider the first splitting \eqref{eq:splitting-u-SM} and \eqref{eq:splitting-c-SM} with and without the nonnegativity constraint $x \in K = \R_{+}^{n}$ and examine the proximal map
\begin{equation}\label{eq:Palpha-gc}
P_{\alpha} g_{c}(x) = \arg\min_{y}\Big\{\frac{1}{2\alpha}\|y-x\|^{2} + \frac{1}{2}\|A y-b\|^{2} + \delta_{K}(y)\Big\},
\end{equation}
whose evaluation is required both in the FBS iteration \eqref{eq:intro-FBS-a} and in the accelerated FBS  defined in  Algorithm \ref{alg:FBS}. 
Note that the unconstrained case $P_{\alpha} g_{u}$ is the special case $K = \R^{n}$.
\subsubsection{Optimality Condition, Duality Gap}\label{sec:Proxmap-OC-DG}

We rewrite \eqref{eq:Palpha-gc} as
\begin{subequations}\label{eq:prox-map-FB-basic}
\begin{align}\label{eq:prox-point-oly}
P_{\alpha} g_{c}(x) &= \ol{y} = \arg\min_{y} \phi(y),\\
\shortintertext{where}
\label{eq:def-phi-0}
 \phi(y) &= \phi_{0}(y)+\delta_{K}(y),\qquad
\phi_{0}(y) = \frac{1}{2}\la y,B_{\alpha} y\ra - \big\la c_{\alpha}, y\big\ra,
\\ \label{eq:def-B-c-alpha}
B_{\alpha} &= A^{\T} A + \frac{1}{\alpha} I_{n},\qquad
c_{\alpha} = \frac{1}{\alpha} x + A^{\T} b.
\end{align}
\end{subequations}
Then $\phi_{0} \in \mc{F}_{c}^{1}\Big(\|A\|^{2}+\frac{1}{\alpha}\Big)$ and we have
\begin{subequations}\label{eq:def-nabla-phi0}
\begin{align}
\nabla\phi_{0}(y) 
&= B_{\alpha} y-c_{\alpha}
= \frac{1}{2}(y-x)+A^{\T}(A y-b) 
\\
&= \frac{1}{2}(y-x)+\nabla g_{u}(y),
\end{align}
\end{subequations}
with $g_{u}$ given by \eqref{eq:splitting-u-SM}. 
Applying Fermat's optimality condition to the proximal point $\ol{y}$, which is the unique minimizer of $\phi$, yields the variational inequality
\begin{subequations}\label{eq:P-alpha-normal-cone} 
\begin{align}
-\nabla\phi_{0}(\ol{y})
= c_{\alpha} - B_{\alpha} \ol{y} &\in N_{K}(\ol{y})
\label{eq:P-alpha-normal-cone-a} \\
\gdw\qquad
\big\la c_{\alpha} - B_{\alpha} \ol{y}, z-\ol{y}\big\ra
&\leq 0,\qquad \forall z \in K,
\end{align}
\end{subequations}
with $N_{K}(\ol{y})$ denoting the normal cone of $K$ at $\ol{y}$.

We consider any feasible point $z \in K$, $z \neq y$, and compute and error measure in terms of the duality gap induced by a dual feasible point associated with $z$. 
\begin{lemma}
Let $z \in K$ and define the dual feasible point 
\begin{equation}\label{eq:def-dual-p}
p = p(z) = (c_{\alpha}-B_{\alpha} z)_{-}.
\end{equation}
Then the duality gap induced by the inexactness $z \approx \ol{y} = P_{\alpha} g(x)$ is denoted by $\rm{dgp}$ and given by
\begin{equation}\label{eq:dgp-z}
\dgp(z)
= \frac{1}{2}\|(c_{\alpha}-B_{\alpha} z)_{+}\|_{B_{\alpha}^{-1}}^{2} - \la p,z\ra
\geq 0.
\end{equation}
\end{lemma}
\begin{proof} 
Denote by $p \in K^{\ast} = \R_{-}^{n}$ the multiplier vector corresponding to the constraint $z \in K$. Then the dual problem corresponding to the primal problem \eqref{eq:prox-point-oly} reads
\begin{equation}
\max_{p}\psi(p),\qquad
\psi(p) = -\frac{1}{2}\la c_{\alpha}-p, B_{\alpha}^{-1}(c_{\alpha}-p)\ra - \delta_{K^{\ast}}(p).
\end{equation}
The unique pair of optimal primal and dual points $(\ol{y},\ol{p}$) are connected by
\begin{equation}\label{eq:p-y-bar}
\ol{p} = p(\ol{y}) = c_{\alpha} - B_{\alpha} \ol{y}.
\end{equation}
Choosing any point $z \in K$ and the corresponding dual feasible point 
\begin{equation}\label{eq:def-p(z)}
p = p(z) = (c_{\alpha}-B_{\alpha} z)_{-} = c_{\alpha}-B_{\alpha} z - (c_{\alpha}-B_{\alpha} z)_{+}
\end{equation}
yields the duality gap
\begin{equation}\label{eq:duality-gap-phi}
\dgp(z) = \phi(z) - \psi\big(p(z)\big) \geq 0
\end{equation}
which, after rearranging, becomes \eqref{eq:dgp-z}.
\end{proof}
Clearly, choosing $z=\ol{y}$ yields a zero duality gap: comparing \eqref{eq:p-y-bar} and \eqref{eq:def-p(z)} shows that $(c_{\alpha}-B_{\alpha} \ol{y})_{+} = 0$. In addition, we have $\ol{p} \perp \ol{y}$ by \eqref{eq:P-alpha-normal-cone-a}. Both relations imply $\dgp(\ol{y})=0$ by \eqref{eq:dgp-z}.

\subsubsection{Inexactness Criteria}\label{sec:Inexactness-Criteria}
We describe two different ways for assessing the inexactness of evaluations of the proximal mapping $P_{\alpha} g$. These measures enable to specify criteria for terminating \textit{early} the corresponding inner iterative loops (cf.~Algorithm \ref{alg:FBS}, lines 4-6), without compromising convergence of the FBS iteration.
\begin{description}
\item[Summable error sequences]
Let
\begin{equation}\label{eq:z-Palpha-g-inexact}
\ol{z}_{k} \approx P_{\alpha_{k}}g\big(x_{k}-\alpha_{k}\nabla f(x_{k})\big)
\end{equation}
be points obtained by inexact evaluations of the proximal mapping $P_{\alpha_{k}} g$ \eqref{eq:Palpha-gc}, at each iteration $k \in \N$. Denote the corresponding error vectors by 
\begin{equation}\label{eq:def-ek}
e_{k} = x_{k+1}-\ol{z}_{k}.
\end{equation}
Then the FBS algorithm \eqref{eq:intro-FBS} converges \cite[Theorem~3.4]{Combettes:2005aa} if
\begin{equation}\label{eq:ek-summability}
\sum_{k \in \N} \|e_{k}\| < +\infty.
\end{equation}
\item[Inexact subgradients]
A point $z \in \R^{n}$ is said to \textit{evaluate $P_{\alpha}g(x)$ with $\veps$-precision}, denoted by
\begin{subequations}\label{eq:P-alpha-eps-OC}
\begin{align}\label{eq:P-alpha-eps-OC-a}
z \approxeq_{\veps} P_{\alpha}g(x),\quad \veps > 0
\intertext{if}\label{eq:P-alpha-eps-OC-b}
\frac{1}{\alpha}(x-z) \in \partial_{\frac{\veps^{2}}{ 2\alpha}} g(z),
\end{align}
\end{subequations}
where the right-hand side is defined by the $\veps$-subdifferential at $z$ \cite[Section 23]{Rockafellar:1970ab}
\begin{subequations}\label{eq:def-eps-subdiff}
\begin{align}
\partial_{\veps} g(z) 
&= \{p \in \R^{n} \mid g(y) \geq g(z) + \la p,y-z\ra - \veps\},\qquad \forall y \in \R^{n}
\label{eq:def-eps-subdiff-a} \\ 
\label{eq:def-eps-subdiff-b}
&= \{p \in \R^{n} \mid  g(z) + g^{\ast}(p)-\la z, p\ra\leq \veps\}.
\end{align}
\end{subequations}
Now consider the sequence of inexact evaluations
\begin{equation}\label{eq:z-subgr-inexact}
x_{k+1} = \ol{z}_{k} \approxeq_{\veps_{k}} P_{\alpha_{k}}g\big(y_{k}-\alpha_{k}\nabla f(y_{k})\big)
\end{equation}
corresponding to line 8 of Algorithm \ref{alg:FBS}. 
Theorem 4.4 in \cite{Villa:2013aa} assures the $\mc{O}(1/k^{2})$ convergence rate of Algorithm \ref{alg:FBS} provided the error parameter sequence $(\veps_{k})_{k \in \N}$ corresponding to \eqref{eq:z-subgr-inexact} decays at rate
\begin{equation}\label{eq:ek-decay-rate}
\veps_{k} = \mc{O}\Big(\frac{1}{k^{3/2+\delta}}\Big),\quad \delta > 0.
\end{equation}
\end{description}

The second criterion \eqref{eq:P-alpha-eps-OC-a} requires to recognize when $z$ satisfies condition \eqref{eq:P-alpha-eps-OC-b}. 
Taking into account the specific form $g_{c} = g_{u}+\delta_{K}$ of $g$ as defined by \eqref{eq:splitting-u-SM} and \eqref{eq:splitting-c-SM}, the following lemma parametrizes any subgradient $p \in \partial_{\veps} g(z)$ by auxiliary points $z_{p}, w \in \R^{n}$ through conjugation, in order to obtain a more explicit form of relation \eqref{eq:P-alpha-eps-OC-a}.
\begin{lemma}\label{lem:p-zp-w}
Let $g = g_{u} + \delta_{K} \in \mc{F}_{c}$ with $g_{u} \in \mc{F}_{c}^{1}(L_{g_{0}})$ and $z \in \R^{n}$. Then finding a vector $p \in \partial_{\veps} g(z)$ is equivalent to finding a pair of vectors $z_{p}\in K$, $w \in N_{K}(z_{p})$ such that $p$ is given by
\begin{equation}\label{eq:p-zp-w}
p = \nabla g_{u}(z_{p}) + w
\qquad\text{with}\qquad
z_{p}\in K,\quad w \in N_{K}(z_{p})
\end{equation}
satisfies
\begin{equation}\label{eq:def-eps-subdiff-c}
g_{u}(z)-g_{u}(z_{p})-\la\nabla g_{u}(z_{p}),z-z_{p}\ra - \la w, z\ra \leq \veps.
\end{equation}
In the unconstrained case $K = \R^{n}$, we have $w=0$ and no constraint imposed on $z_{p}\in \R^{n}$.
\end{lemma}
\begin{proof}
Denoting by $g^\ast$ the conjugate function of $g$, we can write
\begin{equation}
g^{\ast}(p) = \sup_{z}\{\la p,z \ra - g_{u}(z)-\delta_{K}(z)\}
= -\inf_{z}\{g_{u}(z)-\la p,z \ra + \delta_{K}(z)\},
\end{equation}
and Fermat's optimality condition for the latter minimization problem reads $p \in \nabla g_{u}(z_{p}) + N_{K}(z_{p})$, which is equivalent to \eqref{eq:p-zp-w}. Substituting $p$ and taking into 
account that $w \perp z_{p}$  due to  $z_{p}\in K$ and $w\in N_{K}(z_{p})$
gives $g^{\ast}(p) = \la\nabla g_{u}(z_{p}),z_{p}\ra - g_{u}(z_{p})$ and substitution into \eqref{eq:def-eps-subdiff-b} yields
\begin{equation}
g(z)+g^{\ast}(p)-\la z,p\ra
= g_{u}(z) + \la\nabla g_{u}(z_{p}),z_{p}\ra - g_{u}(z_{p})
- \la z, \nabla g_{u}(z_{p}) + w\ra \leq \veps,
\end{equation}
which is relation \eqref{eq:def-eps-subdiff-c}.
\end{proof}

In the following two subsections, we examine criterion \eqref{eq:P-alpha-eps-OC}, for both the unconstrained and the constrained case using Lemma \ref{lem:p-zp-w}, and 
compare that to using instead the error vectors \eqref{eq:def-ek}.

\subsubsection{Recognizing $z \approx_{\veps}P_{\alpha}g_{u}(x)$: The Unconstrained Case}\label{sec:recognizing-inexactness-unconstrained}
Let $K = \R^{n}$. Applying Lemma \ref{lem:p-zp-w}  shows that condition \eqref{eq:P-alpha-eps-OC-b} for recognizing $z \approx_{\veps} P_{\alpha} g(x)$ can be expressed as finding another point $z_{p}$ such that
\begin{subequations}\label{eq:z-Palpha-unconstrained}
\begin{align}
\frac{1}{\alpha}(x-z) &= \nabla g_{u}(z_{p})\qquad\text{and}
\label{eq:z-Palpha-unconstrained-a} \\ \label{eq:z-Palpha-unconstrained-b}
\frac{\veps^{2}}{2\alpha} 
&\geq g_{u}(z)-g_{u}(z_{p})-\la\nabla g_{u}(z_{p}),z-z_{p}\ra.
\end{align}
\end{subequations}
Note that replacing the pair $(z,z_{p})$ by a \textit{single} point $z$ that satisfies \textit{both} conditions, implies that $z = \ol{y} = P_{\alpha}g(x)$ is the \textit{exact} proximal point: equality $z=z_{p}$ makes the inequality in \eqref{eq:z-Palpha-unconstrained-b} hold trivially, whereas condition \eqref{eq:z-Palpha-unconstrained-a} then reads $\nabla\phi_{0}(z)=0$ due to \eqref{eq:def-nabla-phi0}, which implies $z=P_{\alpha} g(x)$ by \eqref{eq:P-alpha-normal-cone-a} since $K=\R^{n}$.

If $z \neq z_{p}$, then \eqref{eq:z-Palpha-unconstrained-a} says that neither point is equal to $P_{\alpha}g(x)$, whereas the right-hand side of condition \eqref{eq:z-Palpha-unconstrained-b} is always nonnegative and measures the difference between $z$ and $z_{p}$ by a Bregman-like distance induced by $g_{u}$ (it is not a true Bregman distance \cite[Definition 3.1]{Bauschke:1997aa} since $g_{u}$ merely is convex). The decomposition \eqref{eq:z-Palpha-unconstrained-a} of the optimality condition determining $P_{\alpha}g(x)$ in terms of $(z,z_{p})$ enables to terminate any iterative algorithm that converges to $P_{\alpha}g(x)$, once the level of inexactness \eqref{eq:z-Palpha-unconstrained-b} is reached.

We conclude this subsection by comparing \eqref{eq:z-Palpha-unconstrained} with \eqref{eq:def-ek}. Let $\ol{y}=P_{\alpha} g(x)$ denote the exact proximal point as in Subsection \ref{sec:Proxmap-OC-DG} and let $z \approx_{\veps} P_{\alpha} g(x)$ satisfy \eqref{eq:z-Palpha-unconstrained}. Then the error vector \eqref{eq:def-ek} reads
\begin{equation}\label{eq:e-z-oly}
e = \ol{y}-z. 
\end{equation}
Since $\phi_{0}$ given by \eqref{eq:def-phi-0} is $\frac{1}{\alpha}$-strongly convex \cite[Definition 12.58]{Rockafellar:2009aa} and $\nabla\phi_{0}(\ol{y})=0$, we have the inequality
\begin{equation}\label{eq:norm-e-unconstrained}
\|e\| = \|z-\ol{y}\| 
\leq 2\alpha\|\nabla\phi_{0}(z)\|
\overset{\eqref{eq:def-nabla-phi0}}{=} 
2\|z-x-\alpha\nabla g_{u}(z)\|.
\end{equation}
Hence, this error vector evaluates violations of Equation \eqref{eq:z-Palpha-unconstrained-a} due to $z \neq \ol{y}$, for a \textit{single} point $z$ on both sides.

\subsubsection{Recognizing $z \approx_{\veps}P_{\alpha}g(x)$: The  Constrained Case}
\label{sec:recognizing-inexactness}

Let $K = \R_{+}^{n}$. In view of \eqref{eq:P-alpha-eps-OC-b}, we identify $p = \frac{1}{\alpha}(x-z)$ in \eqref{eq:p-zp-w}. Lemma \ref{lem:p-zp-w} then shows that 
$z \approx_{\veps} P_{\alpha} g$ holds if there is another point $z_{p}$ such that
\begin{subequations}\label{eq:z-Palpha-constrained} 
\begin{align}
\frac{1}{\alpha}(x-z)-\nabla g_{u}(z_{p})
&= w 
\qquad\text{with}\qquad
z_{p} \perp w,\quad z, z_{p}\in K,\quad w\in N_{K}(z_{p})
\label{eq:z-Palpha-constrained-a} \\
\label{eq:z-Palpha-constrained-b}
\frac{\veps^{2}}{2\alpha}
&\geq g_{u}(z)-g_{u}(z_{p})-\la\nabla g_{u}(z_{p}),z-z_{p}\ra - \la w, z \ra.
\end{align}
\end{subequations}
The discussion above \eqref{eq:z-Palpha-unconstrained} applies here analogously. In particular, if $z=z_{p}$, then \eqref{eq:z-Palpha-constrained-a} reads
\begin{equation}\label{eq:rec-Pageps-OC}
\frac{1}{\alpha}(x-z)-\nabla g_{u}(z)
\overset{\eqref{eq:def-nabla-phi0}}{=} 
c_{\alpha}-B_{\alpha} z \leq 0 = w \in N_{K}(z)
\end{equation}
which is the optimality condition \eqref{eq:P-alpha-normal-cone} that uniquely determines the proximal point $z=\ol{y}=P_{\alpha}g(x)$.

We conclude this subsection by comparing \eqref{eq:z-Palpha-constrained} with \eqref{eq:def-ek} and consider again the error vector \eqref{eq:e-z-oly}, taking additionally into account the nonnegativity constraint $z \in K$. 
Since $\phi(y)-\frac{1}{2\alpha}\|y-x\|^{2}=g_{c}(x)$ is convex, $\phi$ is $\frac{1}{\alpha}$-strongly convex. Exploiting weak duality 
\eqref{eq:duality-gap-phi}, in particular $\phi(\ol{y})\geq\psi(p),\,\forall p \in K^{\ast}$, we estimate
\begin{equation}\label{eq:norm-e-constrained}
\|e\|^{2}=\|z-\ol{y}\|^{2}
\leq 2\alpha\big(\phi(z)-\phi(\ol{y})\big)
\leq 2\alpha\big(\phi(z)-\psi\big(p(z)\big)\big)
= 2\alpha \dgp(z),\quad z \in K,
\end{equation}
with $p(z)$ given by \eqref{eq:def-p(z)}. The explicit expression \eqref{eq:dgp-z} shows how $\dgp(z)$ evaluates the non-optimality of the dual point $p(z)$. Comparing the left-hand sides of \eqref{eq:z-Palpha-constrained-a} (where $z=z_{p}$) and \eqref{eq:rec-Pageps-OC} shows that the non-optimality results from $z \neq z_{p}$.

\subsection{Inner Iterative Loop: Accelerated Primal-Dual Iteration}\label{sec:Inner-Loop}
We work out in this subsection the inner iterative loop of Algorithm \ref{alg:FBS}, lines 4--7. 
We apply the primal-dual optimization approach \cite{Chambolle:2011aa} to two different problem decompositions of the corresponding optimization problem \eqref{eq:Palpha-gc}, in order to generate a minimizing sequence $(z_{l})$ that is terminated using the inexactness criteria of Subsection \ref{sec:basic-FB}. The first decomposition involves the inversion of the matrix $B_{\alpha}$ defined by \eqref{eq:def-B-c-alpha}. The second alternative decomposition avoids this inversion. There is a tradeoff regarding the overall computational efficiency: more powerful iterative steps converge faster but are more expensive computationally.

Since the \textit{inner} iterative loop will be only considered in this subsection, we drop the fixed arbitrary \textit{outer} loop iteration index $k$. The argument $x$ of \eqref{eq:Palpha-gc} either denotes the argument $x = x_{k} - \alpha_{k}\nabla f(x_{k})$ of the FBS algorithm \eqref{eq:intro-FBS-a} or the argument $x = y_{k} - \alpha_{k}\nabla f(y_{k})$ of the accelerated FBS iteration in Algorithm \ref{alg:FBS}. The vector $c_{\alpha}$  is then defined accordingly by \eqref{eq:def-B-c-alpha}. In either case, the exact proximal point is denoted by $\ol{y}=P_{\alpha}g(x)$ and satisfies the relations of Subsection \ref{sec:Proxmap-OC-DG}. As discussed below, the inner iterative loop generates a sequence $(z_{l})$ that is terminated once $z_{l} \approx \ol{y}$ is sufficiently close to the exact proximal point according to the inexactness criteria of Subsection \ref{sec:basic-FB}.

The approach \cite{Chambolle:2011aa} applies to convex optimization problems that can be written as saddle-point problems of the form
\begin{equation}\label{eq:Lagrangian-inner-problem}
\min_{x \in X}\max_{y \in Y}\big\{
\la M x, y\ra + G(x)-F^{\ast}(y)\big\},
\end{equation}
where $X, Y$ are finite-dimensional real vector spaces, $M : X \to Y$ is linear and $F, G \in \mc{F}_{c}$. The algorithm involves the proximal maps $P_{\alpha}F^{\ast}$ and $P_{\alpha} G$ as algorithmic operations. In Sections \ref{sec:Opt-Basic-Decomposition} and \ref{sec:Opt-Avoiding-Inversion}, we apply this algorithm to problem \eqref{eq:Palpha-gc} in two different ways, based on two different decompositions of problem \eqref{eq:Palpha-gc} conforming to \eqref{eq:Lagrangian-inner-problem}. Subsequently, in Section \ref{sec:Inner-Termination-Unconstrained} and \ref{sec:Inner-Termination-Constrained}, we examine the application of the inexactness criteria of Section \ref{sec:Inexactness-Criteria} for terminating the primal-dual iteration.

\subsubsection{Basic Decomposition}\label{sec:Opt-Basic-Decomposition} We consider problem \eqref{eq:Palpha-gc} 
rewritten as \eqref{eq:prox-map-FB-basic}. Using a dual vector $p$, we write it in the saddle-point form \eqref{eq:Lagrangian-inner-problem}
\begin{equation}
\min_{y \in \R^{n}}\max_{p \in \R^{n}}\Big\{\la y, p\ra + \phi_{0}(y) - \delta_{K^{\ast}}(p)\Big\}.
\end{equation}
Note that $\phi_{0} \in \mc{F}_{c}^{1}(L_{\phi_{0}},\mu)$ is strongly convex with constant $\mu = \frac{1}{\alpha}$. Applying Algorithm 2 of \cite{Chambolle:2011aa} yields Algorithm \ref{alg:PD-Basic} listed below.
\begin{algorithm}
\DontPrintSemicolon
\textbf{initialization:} 
Choose $\tau_{0},\sigma_{0}>0$ with $\tau_{0}\sigma_{0}\leq 1$, $z_{0},p_{0}\in \R^{n}$ arbitrary. Set $\ol{z}_{0}=z_{0}$ and $l=0$. \;
\Repeat{$z_{l+1}$ satisfies an inexactness criterion.}{
$p_{l+1} = (p_{l}+\sigma_{l}\ol{z}_{l})_{-}$ \;
$z_{l+1} = \big(I_{n}+\tau_{l}B_{\alpha})^{-1}\big(z_{l}-\tau_{l}(p_{l+1}-c_{\alpha})\big)$ \; \label{eq:decomp-1-z-update}
$\theta_{l} = \Big(1+2\frac{\tau_{l}}{\alpha}\Big)^{-1/2},\quad
\tau_{l+1} = \theta_{l}\tau_{l},\quad
\sigma_{l+1} = \frac{\sigma_{l}}{\theta_{l}}$ \;
$\ol{z}_{l+1} = z_{l+1} + \theta_{l}(z_{l+1}-z_{l})$ \;
}
\caption{Proximal Map Decomposition, Primal-Dual Iteration \label{alg:PD-Basic}}
\end{algorithm}

Step \ref{eq:decomp-1-z-update} of Algorithm \ref{alg:PD-Basic} involves the inversion of the $n \times n$ matrix
\begin{equation}
I_{n}+\tau_{l} B_{\alpha} = \Big(1 + \frac{\tau_{l}}{\alpha}\Big) I_{n} + \tau_{l} A^{\T} A.
\end{equation}
Due to assumption \eqref{eq:ass-A}, the computational cost can be reduced by using the Sherman-Morrison-Woodbury formula \cite{Higham:2008aa}, quoted here with arbitrary invertible matrix $B$ and matrices $U, V$ with compatible dimensions,
\begin{equation}
(B+U V^{\T})^{-1} = B^{-1}-B^{-1} U(I+V^{\T} B^{-1} U)^{-1} V^{\T} B^{-1}.
\end{equation}
Applying this formula yields
\begin{equation}\label{eq:Balpha-Woodbury}
(I_{n}+\tau_{l} B_{\alpha})^{-1}
= \frac{\alpha}{\alpha+\tau_{l}}\bigg(
I_{n}-\frac{\alpha}{\alpha+\tau_{l}} A^{\T}\Big(\frac{1}{\tau_{l}}I_{m}+\frac{\alpha}{\alpha+\tau_{l}}A A^{\T}\Big)^{-1} A\bigg),
\end{equation}
which only involves the inversion of an $m \times m$ matrix. 

For very large problem sizes, however, evaluating this formula may be too costly as well. Therefore, in the next subsection, we consider an alternative problem decomposition that avoids matrix inversion. 

\subsubsection{Avoiding Matrix Inversion}\label{sec:Opt-Avoiding-Inversion}
In order to avoid matrix inversion in step \eqref{eq:decomp-1-z-update} of Algorithm \ref{alg:PD-Basic}, we rewrite the objective function $\phi$ of \eqref{eq:prox-map-FB-basic} in the form
\begin{equation}
\phi(y) = \frac{1}{2}\|A y\|^{2} + \frac{1}{2\alpha}\|y\|^{2}-\la c_{\alpha}, y\ra + \delta_{K}(y),
\end{equation}
where $K$ either is $\R^n$ or $\R^{n}_{+}$. Using a dual vector $q$, we dualize the term comprising $A$, to obtain the saddle-point problem
\begin{equation}
\min_{y}\max_{q}\Big\{\la A y,q\ra + \frac{1}{2\alpha}\|y\|^{2}-\la c_{\alpha}, y\ra + \delta_{K}(y) - \frac{1}{2}\|q\|^{2} \Big\}.
\end{equation}
Applying Algorithm 2 of \cite{Chambolle:2011aa} yields Algorithm \ref{alg:PD-No-Inversion} listed below.
\begin{algorithm}
\DontPrintSemicolon
\textbf{initialization:} 
Choose $\tau_{0},\sigma_{0}>0$ with $\tau_{0}\sigma_{0}\leq \frac{1}{\|A\|_{2}^{2}}$, $z_{0} \in \R^{n},q_{0}\in \R^{m}$ arbitrary. Set $\ol{z}_{0}=z_{0}$, $l=0$. \;
\Repeat{$z_{l+1}$ satisfies an inexactness criterion.}{
$q_{l+1} = \frac{1}{1+\sigma_{l}}(q_{l}+\sigma_{l} A \ol{z}_{l})$ \; \label{eq:alg-inner-2-q-update}
$z_{l+1} = \Pi_{K}\Big(\frac{\alpha}{\alpha+\tau_{l}}\big(
z_{l}-\tau_{l}(A^{\T} q_{l+1} - c_{\alpha})\big)\Big)$ \; \label{eq:alg-inner-2-z-update}
$\theta_{l} = \Big(1+2\frac{\tau_{l}}{\alpha}\Big)^{-1/2},\quad
\tau_{l+1} = \theta_{l}\tau_{l},\quad
\sigma_{l+1} = \frac{\sigma_{l}}{\theta_{l}}$ \; 
$\ol{z}_{l+1} = z_{l+1} + \theta_{l}(z_{l+1}-z_{l})$ \;\label{eq:alg-inner-2-olz-update}
}
\caption{Proximal Map Decomposition, Primal-Dual Iteration without Matrix Inversion\label{alg:PD-No-Inversion}}
\end{algorithm}

Note that this algorithm only requires a single matrix-vector multiplication with respect to $A$ and $A^{\T}$ in each iteration. While this is much cheaper than the matrix inversion of Algorithm \ref{alg:PD-Basic}, considerably more iterations are required to satisfy the inexactness criterion.

\subsubsection{Termination: The Unconstrained Case}\label{sec:Inner-Termination-Unconstrained}
Let $K = \R^{n}$. Then $\ol{y}=P_{\alpha}g(x)$ uniquely minimizes $\phi=\phi_{0}$, and is determined by the optimality condition
\begin{equation}
\nabla\phi_{0}(\ol{y}) = B_{\alpha}\ol{y}-c_{\alpha} = 0
\quad\gdw\quad
\ol{y} = B_{\alpha}^{-1} c_{\alpha} = P_{\alpha}g_{u}(x).
\end{equation}
Now suppose that applying a direct method for computing $B_{\alpha}^{-1} c_{\alpha} = (I_{n}+\alpha A^{\T} A)^{-1} c_{\alpha}$ is numerically intractable, due to the problem size, even when a formula analogous to \eqref{eq:Balpha-Woodbury} is used to reduce the problem size for matrix inversion from $n$ to $m < n$. Then one chooses an iterative numerical method that generates a sequence $(z_{l})$ converging to $\ol{y}$ until the inexactness condition $z_{l} \approxeq_{\veps}P_{\alpha} g(x)$ holds. As discussed in Subsection \ref{sec:recognizing-inexactness-unconstrained}, in order to make this criterion explicit, a pair of points $(z, z_{p})$ has to be chosen such that the two relations \eqref{eq:z-Palpha-unconstrained} are satisfied.

Let $z_{l+1}$ be a point generated by the iterative algorithm that approximately satisfies the optimality condition 
\begin{equation}
\nabla\phi_{0}(z_{l+1}) = \nabla g_{u}(z_{l+1}) - \frac{1}{\alpha}(x-z_{l+1}) \approx 0.
\end{equation}
Setting
\begin{equation}\label{eq:zp-z-unconstrained-0}
z_{p} = z_{l+1},\qquad
z = x-\alpha\nabla g_{u}(z_{p})
\end{equation}
makes condition \eqref{eq:z-Palpha-unconstrained-a} hold and we can evaluate the right-hand side of \eqref{eq:z-Palpha-unconstrained-b} in order to check if $z \approxeq_{\veps} P_{\alpha} g(x)$. It remains to be checked if the choice \eqref{eq:zp-z-unconstrained-0} is compatible with the \textit{specific} Algorithm \ref{alg:PD-No-Inversion}. Step \eqref{eq:alg-inner-2-z-update} reads
\begin{subequations}\label{eq:derivation-unconstrained}
\begin{align}
0 &= \frac{\alpha}{\alpha+\tau_{l}}\big(
z_{l}-\tau_{l}(A^{\T} q_{l+1} - c_{\alpha})\big) - z_{l+1} 
\\ \label{eq:choice-zp-z-unc-1}
&= \frac{1}{\alpha}\Big(x-z_{l+1}-\frac{\alpha}{\tau_{l}}(z_{l+1}-z_{l}) - \alpha A^{\T}(q_{l+1}-A z_{l+1})\Big) - A^{\T}(A z_{l+1}-b),
\end{align}
\end{subequations}
and the choice
\begin{equation}\label{eq:zp-z-unconstrained-1}
z_{p}=z_{l+1},\qquad
z := z_{p}-\frac{\alpha}{\tau_{l}}(z_{p}-z_{l}) - \alpha A^{\T}(q_{l+1}-A z_{p}),
\end{equation}
makes Equation \eqref{eq:choice-zp-z-unc-1} equal to the second equation of \eqref{eq:zp-z-unconstrained-0}.

A minor disadvantage of \eqref{eq:zp-z-unconstrained-0} is the additional evaluation of the gradient $\nabla g_{u}(z_{p})$. In the present unconstrained case, this can be avoided as follows. Since
the sequence of dual vectors $(q_{l})$ converges to $q_{l} \to A \ol{y}$ as line \ref{eq:alg-inner-2-q-update} of Algorithm \ref{alg:PD-No-Inversion} reveals, the expression $A^{\T}(q_{l+1}-b) \approx \nabla g_{u}(\ol{y})$ approximates the gradient. Hence, rearranging Equation 
\eqref{eq:derivation-unconstrained} into the form
\begin{equation}
0 = \frac{1}{\alpha}\Big(x-z_{l+1}-\frac{\alpha}{\tau_{l}}(z_{l+1}-z_{l})\Big) - A^{\T}(q_{l+1}-b),
\end{equation}
suggests the choice
\begin{equation}\label{eq:zp-z-unconstrained-2}
A z_{p} = q_{l+1},\qquad
z = z_{l+1}+\frac{\alpha}{\tau_{l}}(z_{l+1}-z_{l}).
\end{equation}
Then $A^{\T}(q_{l+1}-b)=\nabla g_{u}(z_{p})$ and condition \eqref{eq:z-Palpha-unconstrained-a} is satisfied. Condition \eqref{eq:z-Palpha-unconstrained-b} can be efficiently checked as well, even though \eqref{eq:zp-z-unconstrained-2} does not determine $z_{p}$ explicitly: 
\begin{subequations}\label{eq:inexactness-unconstrained-final}
\begin{align}
\frac{\veps^{2}}{2\alpha}
&\geq \frac{1}{2}\|A z-b\|^{2}-\frac{1}{2}\|A z_{p}-b\|^{2}-\la A z_{p}-b,A z-A z_{p}\ra
\\ \label{eq:inexactness-unconstrained-final-term}
&= \frac{1}{2}\big\|(A z-b)-(A z_{p}-b)\big\|^{2}
= \frac{1}{2}\|A z-q_{l+1}\|^{2}.
\end{align}
\end{subequations}
Thus, checking these conditions during the primal-dual iteration only requires negligible additional computation. Definition \eqref{eq:zp-z-unconstrained-2} of $z$ shows that the single additional vector-matrix multiplication $A z$ is implicitly done by evaluating the primal-dual steps \eqref{eq:alg-inner-2-q-update} and \eqref{eq:alg-inner-2-olz-update} of Algorithm \ref{alg:PD-No-Inversion}.

We conclude by comparing the alternative error measure \eqref{eq:norm-e-unconstrained}, which for the choice \eqref{eq:zp-z-unconstrained-0} translates to
\begin{equation}
\|e\| = \|z_{p}-z\|,
\end{equation}
with condition \eqref{eq:ek-summability} for summing up these inner-loop errors at the outer iteration level. Conversely, inserting the alternative choice \eqref{eq:zp-z-unconstrained-2} into \eqref{eq:inexactness-unconstrained-final-term} yields the error
\begin{equation}\label{eq:veps-unconstrained}
\veps \geq \alpha^{1/2} \|A(z-z_{p})\|
\end{equation}
in terms of vectors of smaller dimension $m$ (cf.~assumption \eqref{eq:ass-A}) and with a factor $\alpha^{1/2} \ll 1$ that typically is much smaller than $1$ due to the upper bound $\alpha < \frac{2}{L_{f}}$ of feasible proximal parameters. The sum of the errors \eqref{eq:veps-unconstrained} for all outer steps $k \in \N$ of the iteration not only have to be summable but should decay with rate \eqref{eq:ek-decay-rate} to achieve fast convergence.

\subsubsection{Termination: The nonnegativity  Constrained Case}\label{sec:Inner-Termination-Constrained}
Let $K=\R_{+}^{n}$. 
We examine criteria based on the conditions \eqref{eq:z-Palpha-constrained} for terminating the inner-loop iteration that generates a sequence $(z_{l})$ converging to the proximal point $\ol{y}=P_{\alpha} g(x)$. Algorithm 
\ref{alg:PD-Basic} does not comprise a projection onto $K$ and hence does not conform to condition \eqref{eq:z-Palpha-constrained-a}. Therefore, we only consider early termination of Algorithm  \ref{alg:PD-No-Inversion}.

Based on the projection theorem, \cite[Thm. 3.14]{Bauschke:2010aa} we rewrite step \eqref{eq:alg-inner-2-z-update} as
\begin{subequations}
\begin{align}
N_{K}(z_{l+1})
&\ni \frac{\alpha}{\alpha+\tau_{l}}\big(
z_{l}-\tau_{l}(A^{\T} q_{l+1} - c_{\alpha})\big) - z_{l+1}
\\
&\overset{\eqref{eq:def-B-c-alpha}}{=} 
\frac{1}{\alpha}\Big(x-z_{l+1}-\frac{\alpha}{\tau_{l}}(z_{l+1}-z_{l}) - \alpha A^{\T}(q_{l+1}-A z_{l+1})\Big) - A^{\T}(A z_{l+1}-b).
\end{align}
\end{subequations}
Setting
\begin{subequations}\label{eq:zp-z-Alg2-constrained}
\begin{align}
z_{p}=z_{l+1},\qquad
z &= z_{l+1}+\frac{\alpha}{\tau_{l}}(z_{l+1}-z_{l}) + \alpha A^{\T}(q_{l+1}-A z_{l+1}),
\label{eq:zp-z-Alg2-constrained-a} \\
w &= A^{\T}(A z_{l+1}-b)-\frac{1}{\alpha}(x-z)
\end{align}
\end{subequations}
satisfies condition \eqref{eq:z-Palpha-constrained} \textit{except} for the condition $z \in K$. Clearly, $z$ given by \eqref{eq:zp-z-Alg2-constrained} becomes nonnegative as $l \to \infty$, and once this happens for sufficiently large $l$ we can check $z \approxeq_{\veps}P_{\alpha}g(x)$ by evaluating condition \eqref{eq:z-Palpha-constrained-b}. Similarly to \eqref{eq:inexactness-unconstrained-final}, we get
\begin{subequations}\label{eq:inexactness-constrained-final}
\begin{align}
\frac{\veps^{2}}{2\alpha}
&\geq \frac{1}{2}\|A z-b\|^{2}-\frac{1}{2}\|A z_{p}-b\|^{2}-\la A z_{p}-b,A z-A z_{p}\ra + \la w,z \ra
\\
&= \frac{1}{2}\|A (z-z_{p})\|^{2} + \la w, z\ra,
\end{align}
\end{subequations}
which, in view of \eqref{eq:zp-z-Alg2-constrained}, evaluates inexactness of $z$ through a distance and a complementarity term with respect to $z_{p}$.

A weakness of criterion \eqref{eq:inexactness-constrained-final} is that it cannot be applied if $z \not\in K$, with $z$ given by \eqref{eq:zp-z-Alg2-constrained-a}. As an alternative, we evaluate the error measure \eqref{eq:norm-e-constrained} in order to terminate the inner iterative loop. The duality gap $\dgp(z_{l})$ \eqref{eq:dgp-z} requires to compute 
\begin{equation}\label{eq:calpha-Balpha-z}
c_{\alpha}-B_{\alpha} z_{l} = c_{\alpha}-A^{\T} A z_{l} - \frac{1}{\alpha} z_{l} 
\end{equation}
where $A^{\T} A z_{l}$ has to be computed anyway in order to execute Algorithm \ref{alg:PD-No-Inversion}. To avoid the expensive evaluation of the first term on the right-hand side of \eqref{eq:dgp-z}, we substitute \eqref{eq:def-dual-p} and estimate
\begin{equation}
\dgp(z_{l}) \leq \frac{1}{2}\|B_{\alpha}^{-1}\|_{2} \|(c_{\alpha}-B_{\alpha}z_{l})_{+}\|^{2}-\la (c_{\alpha}-B_{\alpha} z_{l})_{-},z_{l}\ra
\end{equation}
and, hence, obtain with $\|B_{\alpha}^{-1}\|_{2}\leq \alpha$ and \eqref{eq:norm-e-constrained}
\begin{equation}\label{eq:norm-e-constrained-estimate}
\|e_{l}\| \leq \alpha \Big(\|(c_{\alpha}-B_{\alpha}z_{l})_{+}\|^{2} - \frac{2}{\alpha}\la (c_{\alpha}-B_{\alpha} z_{l})_{-},z_{l}\ra\Big)^{1/2}.
\end{equation}
After fixing errors $\|e_{k}\|$ at each outer iteration $k$ so as to satisfy \eqref{eq:ek-summability}, each correspoding inner loop can be terminated once the right-hand side of \eqref{eq:norm-e-constrained-estimate} drops below $\|e_{k}\|$.

\subsection{Reverse Problem Splitting, Proximal Map}\label{sec:reverse-FB}
As motivated and discussed in Subsection \ref{sec:SM-CO-Preliminary}, we also evaluate the reverse splitting \eqref{eq:splitting-c-reversed}, including the corresponding unconstrained splitting \eqref{eq:splitting-u-reversed} as a special case.

The FBS iteration \eqref{eq:intro-FBS} now reads
\begin{equation}
x_{k+1} 
= P_{\alpha_{k}} g\big(x_{k}-\alpha_{k} \nabla f_{u}(x_{k})\big)
= P_{\alpha_{k}} g\big(x_{k}-\alpha_{k} A^\top(A x_{k}-b)\big),
\end{equation}
where either $g=g_{c}$ or $g=g_{u}$. We confine ourselves to using \textit{exact} evaluations of the proximal map $P_{\alpha_{k}}g$ using the highly accurate box-constrained L-BFGS method 
from \cite{LBFGS-B}. These iterations replace lines 4--7 of Algorithm \ref{alg:FBS}, that otherwise is applied without change.

A major consequence of using this reversed splitting scheme is that now the target function $R_{\tau}$ defines the proximal map, whereas the least-squares problem defines the task
whose algorithm will be perturbed by the target function reductions. As a consequence, more iterations are spent to lower the target function as inner loop iterations of the proximal mapping evaluation, compared to the number of gradient iterations devoted to solving the unconstrained least-squares problem. Our empirical findings are reported and discussed, 
based on our experimental numerical work, in Section \ref{sec:NumericalResults}.

\section{Numerical Results}\label{sec:NumericalResults}

Subsection \ref{sec:Data} details the experimental set-up based that we use to compare superiorization with accelerated inexact optimization in Subsections \ref{sec:Error-Values} and \ref{sec:Complexity}.

\subsection{Data, Implementation}\label{sec:Data}
The linear system 
\begin{equation}
A x = b,\qquad A \in \R^{m \times n}
\end{equation}
considered for our experiments uses  linear tomographic measurements $b$ of the Shepp-Logan MATLAB $128\times 128$ phantom shown in Figure \ref{fig:data} in a highly 
underdetermined scenario, i.e., $m/n = 0.15625$. The vector representing this test phantom is denoted by $x_\ast$ in the sequel. The subsampling ratio $m/n$
is chosen according to  \cite{Kuske2019} so that recovery error $\|x-x_\ast\|$ with $x$ obtained as minimizer of \eqref{eq:hu-hc}  is small as illustrated in Figure \ref{fig:minimizers-unconstrained}.
We use the MATLAB routine  \texttt{paralleltomo.m} from the AIR Tools II package \cite{AirToolsII} that implements 
such a tomographic matrix for a given vector of projection angles. In particular, we used $20$ angles (projections) uniformly spaced between $1$ and $180$ degrees, 
and $120$ parallel rays per angle of projection, resulting in a full rank
$2560\times16384$ matrix $A$. 
We consider both exact and noisy measurements. In the latter case, i.i.d.~normally distributed noise was added with variance $\sigma^{2}$ determined by
\begin{equation}\label{eq:def-noisy-b}
b_{j} \quad\leftarrow\quad b_{j} + \xi_{j},\qquad
\xi_{j} \sim \mc{N}(0,\sigma^{2}),\qquad
j \in [m],\qquad
\sigma = \frac{0.02}{m} \sum_{j \in [m]} b_{j}.
\end{equation}

Figure \ref{fig:data} depicts in addition to the original image $x_\ast$ and its histogram the least-squares solution of minimal norm $x_{\rm{LS}}$
for noisy measurements along with its histogram.

Figure \ref{fig:minimizers-unconstrained} depicts minimizers $x$ of both the unconstrained and the nonnegativity constrained problem \eqref{eq:hu-hc}  for exact and for noisy data, respectively, and illustrates that the models in \eqref{eq:hu-hc} are appropriate and lead to a relatively small recovery error $\|x-x_\ast\|$.
The histograms show, in particular, that adding nonnegativity constraints increases accuracy.

\vspace{2mm}
\noindent
\textit{Smoothing Parameter.} The smoothing parameter of the target function $R_\tau$ \eqref{eq:R-tau} and of the regularization term in \eqref{eq:hu-hc} was set to $\tau = 0.01$  in all experiments. 

\vspace{2mm}
\noindent
\textit{Regularization Parameter.} 
The regularization parameter $\lambda$ in \eqref{eq:hu-hc} was set to
\begin{equation}\label{eq:lambda-values}
\lambda := \begin{cases}
0.01, &\text{for exact data} \\
1.6529, &\text{for noisy data}
\end{cases}
\qquad\implies\qquad
\left\{
\begin{aligned}
R_{\tau}(x) &\approx R_{\tau}(x_{\ast}) \\
\|A x-b\|^{2} &\approx \|A x_{\ast}-b\|^{2}
\end{aligned}
\right.
\end{equation}
which implies,
\begin{itemize}
\item
in the case of exact data $b$, that minimizers $x$ have almost the same total variation as the ``unknown'' true solution $x_{\ast}$,
\item
in the case of inexact data $b$, that minimizers $x$ reach the noise level of $x_{\ast}$ induced by \eqref{eq:def-noisy-b}.
\end{itemize}

\vspace{2mm}
\noindent
\textit{Least-Squares Regularization Parameter.} 
The parameter $\mu$ in the regularized least-squares objective \eqref{eq:def-gu-mu} was set to $\mu=1/y_\ast$
according to Lemma \ref{lem:choice_of_mu}.
The dual variable $y_\ast$, see the proof of Lemma \ref{lem:choice_of_mu}, is computed via CVX\footnote{
{CVX}: Matlab Software for Disciplined Convex Programming, \url{http://cvxr.com/cvx}.}.
This parameter choice  guarantees that the CG iteration converges to a solution of \eqref{eq:def-gu}.

\vspace{2mm}
\noindent
\textit{Starting vectors.} 
The zero vector $0\in\R^n$ was chosen as starting vectors $x_0,y_0\in\R^n$ in all variants of  Algorithm \ref{alg:SA}
and Algorithm  \ref{alg:FBS}.

\vspace{2mm}
\noindent
\textit{Termination Criteria.} 
For illustration purposes all algorithms where run for a maximum number of (outer) iterations equal to 2000
\emph{beyond} the used stopping rules of Algorithm \ref{alg:SA} and Algorithm \ref{alg:FBS}.
Regarding termination of the superiorized versions of the basic algorithms, recall the proximity sets in \eqref{def:Gamma-sets}, we used
\begin{equation}\label{eq:termination-unconstrained-sup}
g_u(x)\leq \veps
\end{equation}
to terminate the iterative process that concerns the unconstrained least-squares problem, while using
\begin{equation}\label{eq:termination-constrained_sup}
g_u(x) \leq \veps  \quad \text{and} \quad \min_{i \in [n]} x_i > - 10^{-8}
\end{equation}
in the case of the constrained least-squares problem. The parameter $\veps$ was set to the noise level $\veps\approx 0.047 m$ in the noisy case and 
$\veps = 0.001$ in the noise free case.

Regarding the optimization problem \eqref{eq:hu-hc}, 
in the  the unconstrained case  $x \in \R^{n}$, the criterion
\begin{equation}\label{eq:termination-unconstrained}
\|\nabla h_u(x)\|_{\infty} \leq 0.001
\end{equation}
was used to terminate the iterative optimization  and to accept $x$ as approximate minimizer of $h_u$ \eqref{eq:hu-hc}. In the nonnegativity constrained case $x \in \R_{+}^{n}$, the complementarity condition
\begin{equation}\label{eq:termination-constrained}
\big|\min(x,\nabla h_u(x))\big|\leq 0.001 \footnotemark 
\end{equation}
\footnotetext{The optimality condition for the nonnegativity constrained objective $h_c$ \eqref{eq:hu-hc} can be compactly written as $\min(x,\nabla h_u(x))=0 \Leftrightarrow x\ge 0, \nabla h_u(x)\ge 0, x^\top \nabla h_u(x)=0$, where $h_u$ denotes the function without constraints.  }
was used  to accept $x$ as approximate minimizer of $h_c$ \eqref{eq:hu-hc}.

\begin{figure}
   \centering
\begin{tabular}{ccc}
\includegraphics[height=0.2\textwidth]{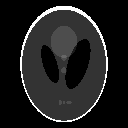}& \includegraphics[height=0.2\textwidth]{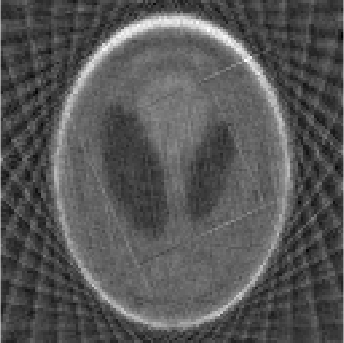} & \includegraphics[height=0.2\textwidth]{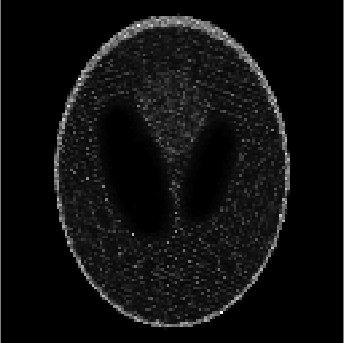}\\
\includegraphics[height=0.2\textwidth]{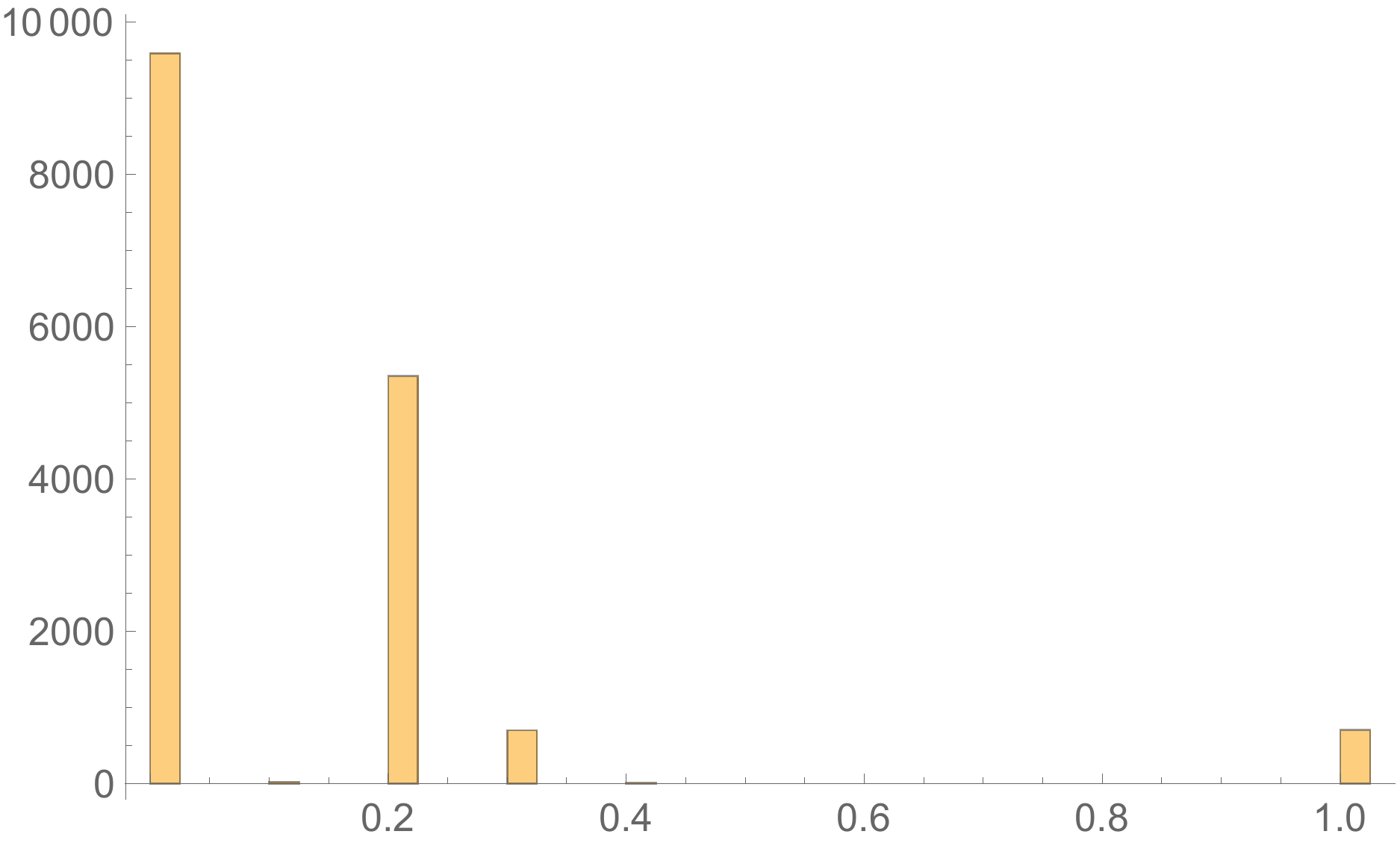} &  \includegraphics[height=0.2\textwidth]{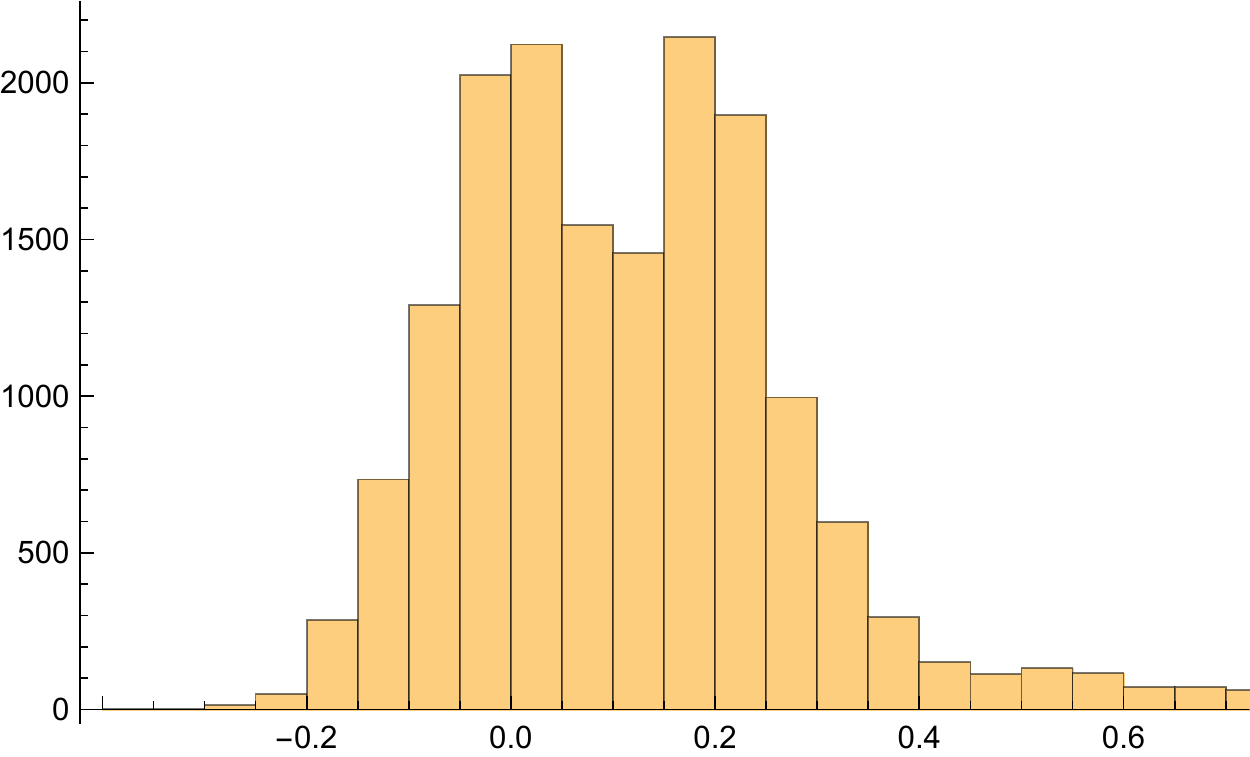} &\includegraphics[height=0.2\textwidth]{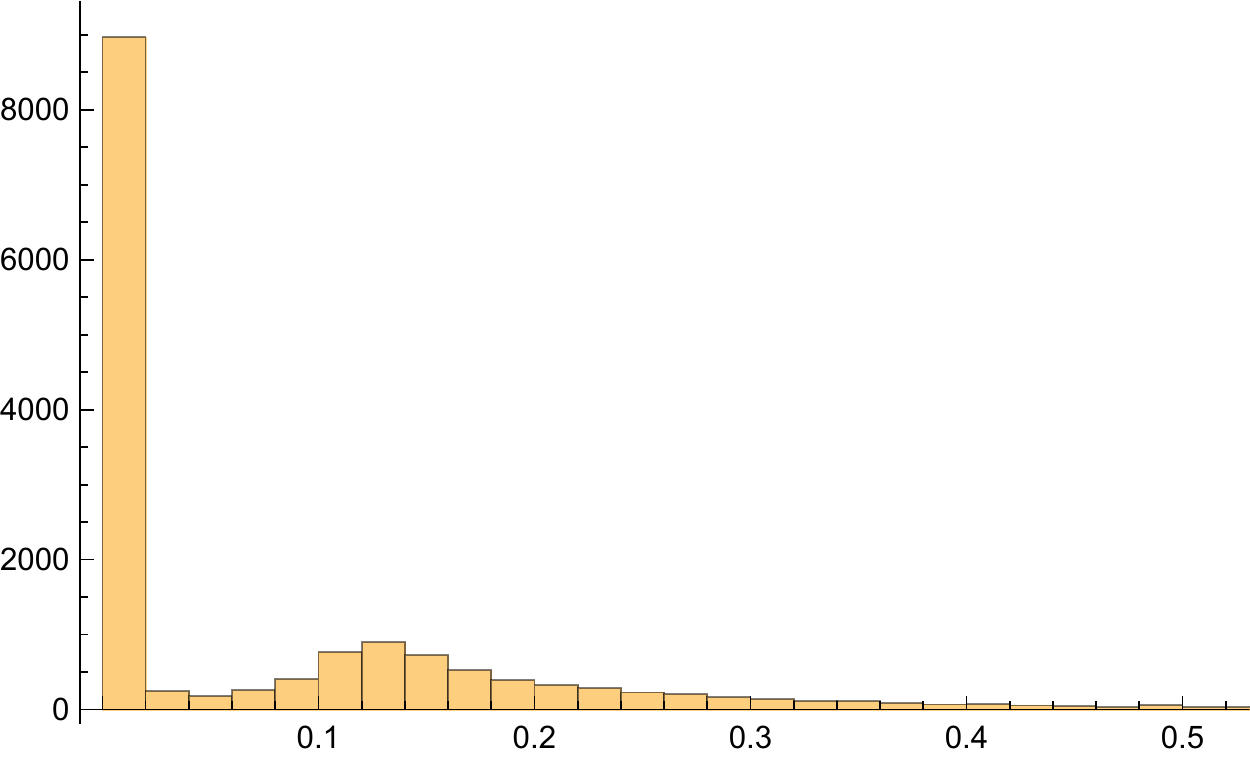}
\end{tabular}
\caption{\textsc{Top left:} The original image $x_{\ast}$ of size $n = 128^2$ to be recovered from $m = 0.15625 n$ linear measurements. \textsc{Bottom left:}  The histogram of $x_{\ast}$
depicts the number of pixels for each intensity value. All components of $x_{\ast,i},\,i\in[n]$ are confined to the interval $[0,1]$.  \textsc{Top middle and right:} The least-squares solution of minimal Euclidean norm $x_{\rm{LS}}$
(middle) and a nonnegative least-squares solution (right)
both include severe artefacts. \textsc{Bottom middle and right:} The histogram of  $x_{\rm{LS}}$ and the histogram of the nonnegative least-squares solution deviate significantly from that of $x_\ast$. This holds for both exact and noisy measurements. The reconstruction problem requires further regularization.
}\label{fig:data}
\end{figure}
\begin{figure}
\centering
\begin{tabular}{ccc}
\includegraphics[height=0.2\textwidth]{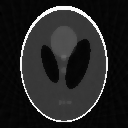}
&
\includegraphics[height=0.2\textwidth]{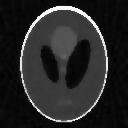}
&
\includegraphics[height=0.2\textwidth]{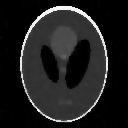}
\\
\includegraphics[width=0.3\textwidth]{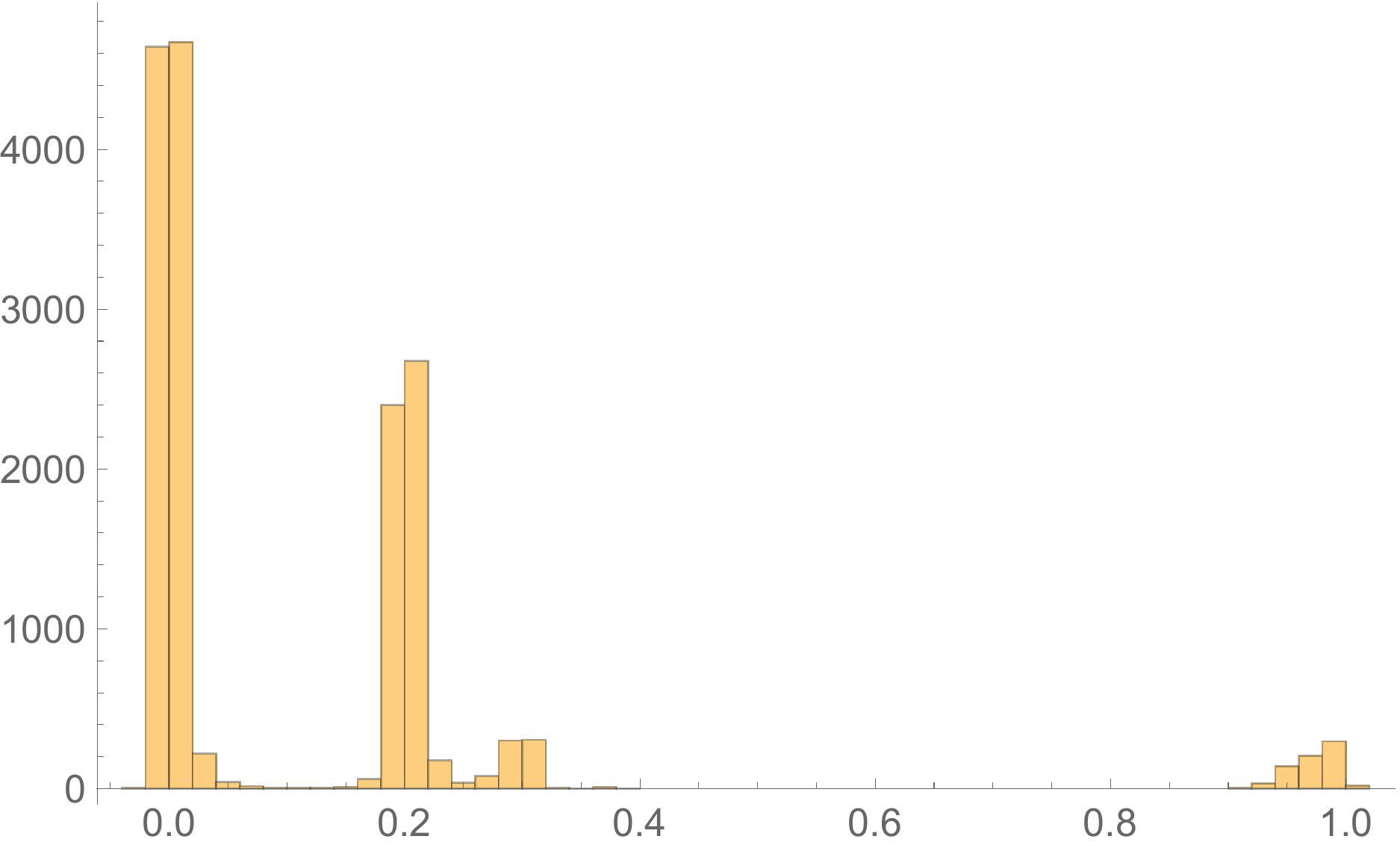}
&
\includegraphics[width=0.3\textwidth]{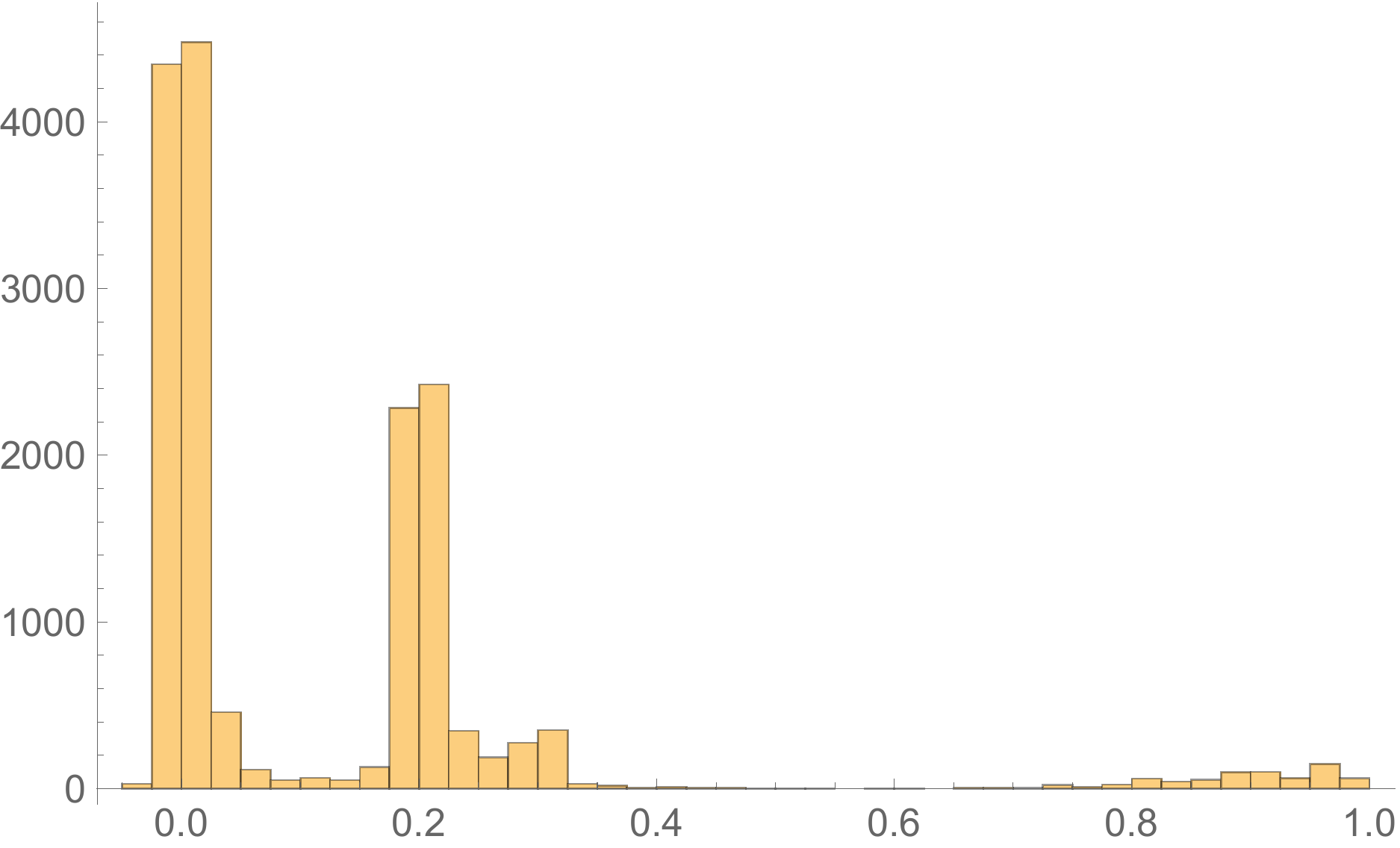}
&
\includegraphics[width=0.3\textwidth]{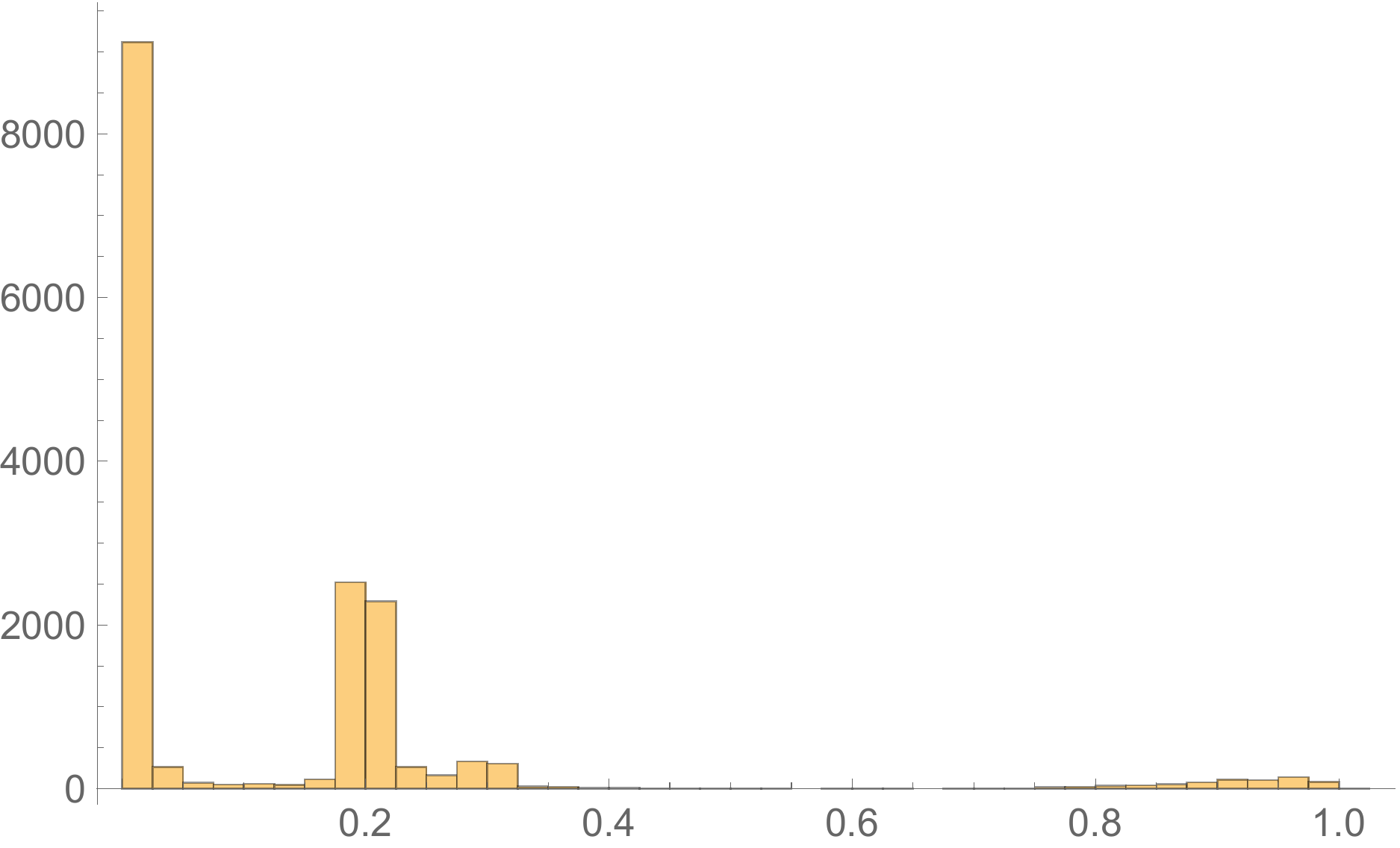}
\\
\includegraphics[width=0.3\textwidth]{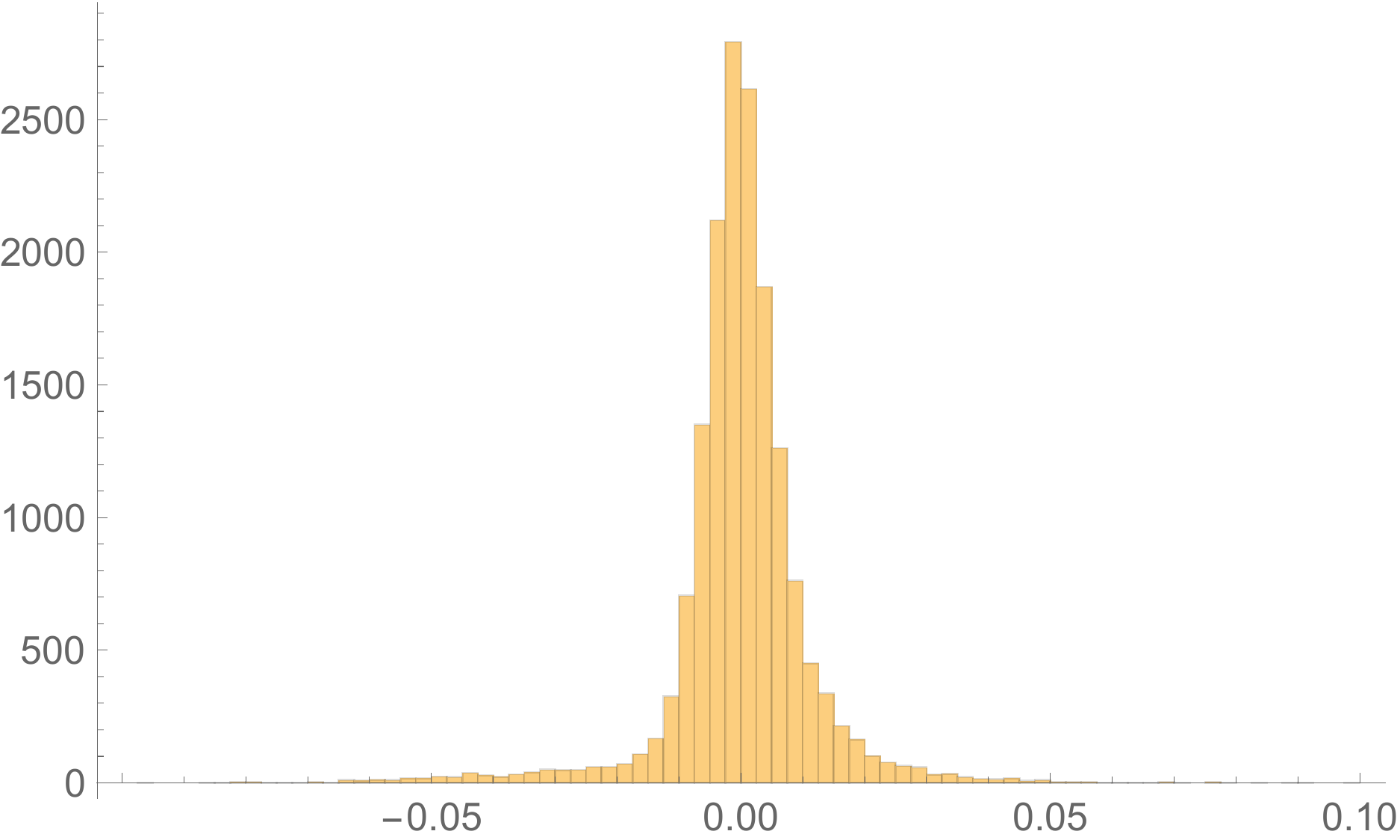}
&
\includegraphics[width=0.3\textwidth]{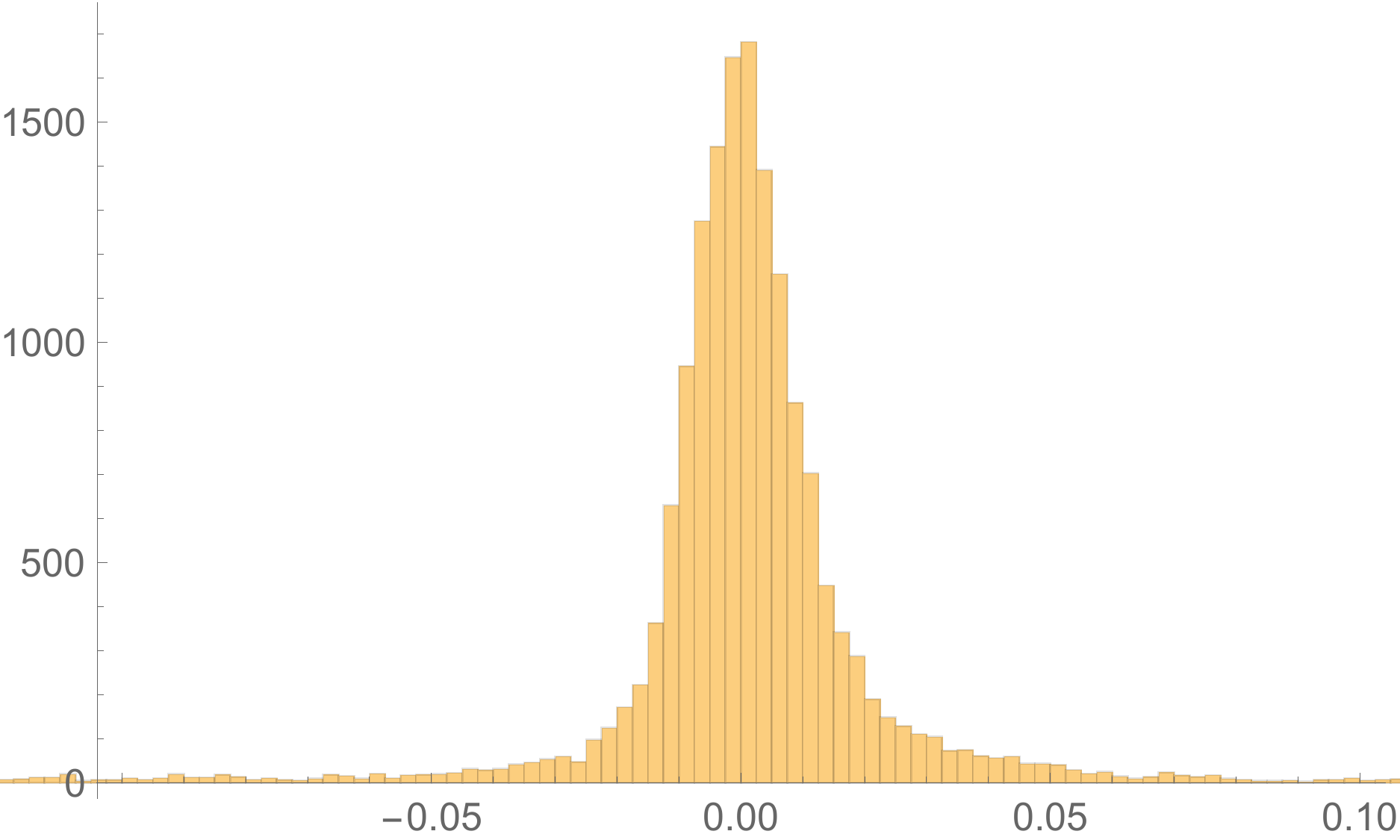}
&
\includegraphics[width=0.3\textwidth]{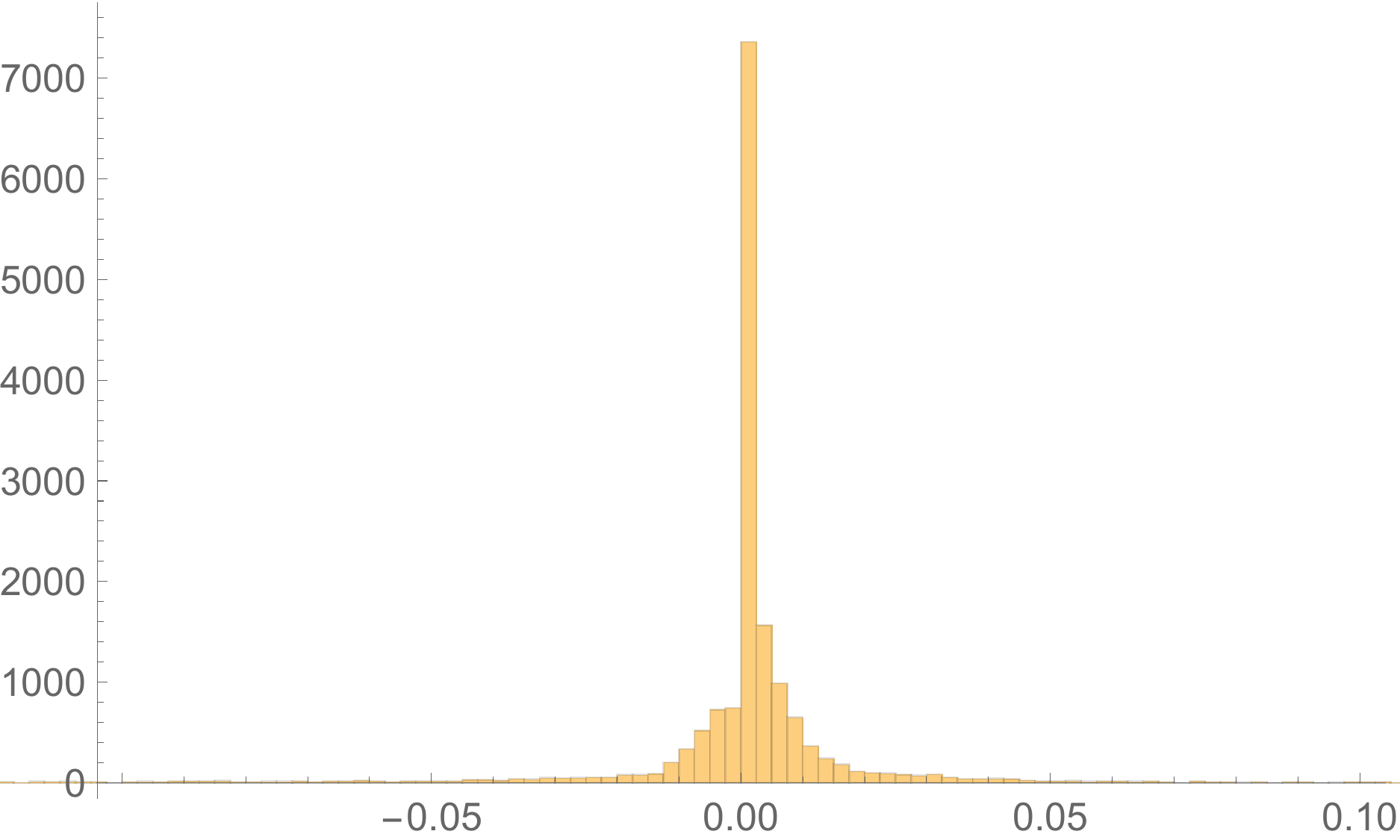}
\end{tabular}
\caption{
\textsc{Top row, from left to right:} Minimizers $x$ of the unconstrained objective function $h_u$ \eqref{eq:hu-hc} for exact data, for noisy data, and minimizer for the nonegatively constrained objective function $h_c$ and noisy data. \textsc{Center row:} The histograms corresponding to the top row. In the unconstrained cases, a number of components $x_{i}$ are outside the range $[0,1]$ of $x_{\ast,i},\,i \in [n]$. \textsc{Bottom row:} The difference histograms of $x-x_{\ast}$ corresponding to the top row. The left histogram (unconstrained case, exact data) peaks more sharply around $0$ than the histogram in the center (unconstrained case, noisy data), 
while the right histogram has the sharpest peak (nonnegativity constrained case, noisy data). 
}
\label{fig:minimizers-unconstrained}
\end{figure}

\subsection{Superiorization vs.~Optimization: Reconstruction Error  Values}\label{sec:Error-Values}

In this subsection we compare the algorithmic performance in terms of reconstruction quality
based on three reconstruction errors:
\begin{enumerate}
\item The \emph{scaled squared residual} $\|Ay_k-b\|^2/(2m)$ generated by variants of Algorithm \ref{alg:SA} and $\|Ax_k-b\|^2/(2m)$ generated by variants of Algorithm \ref{alg:FBS}  that should  approach $0$ in the noiseless case
and ideally approach the noise level $\approx 0.0473$ in the noisy case;
\item The \emph{scaled target function} (regularization  term)  $R_\tau(y_k)/n$ generated by  the superiorized versions of the basic algorithms and $R_\tau(x_k)/n$ generated by each optimization algorithm that should approach the smoothed TV value of the original image $x_\ast$;
\item The \emph{scaled error term} $\|y_k-x_\ast\|^2/n$ computed by Algorithms of type \ref{alg:SA} and $\|x_k-x_\ast\|^2/n$ computed by Algorithms of type \ref{alg:FBS}, that should be reduced as much as possible.
\end{enumerate}
Each optimization algorithm is guaranteed to reduce the combination of (1) and (2) as in \eqref{eq:hu-hc}, and also (3), as our model
 \eqref{eq:hu-hc} is appropriate, see Figure \ref{fig:minimizers-unconstrained}, while the superiorized versions of the basic algorithms
  are only guaranteed to reduce the error in (1). \\
\textbf{Superiorization.} 
We first address  the unconstrained least-squares problem and consider the superiorized CG algorithm and the superiorized Landweber algorithm, by gradient- or proximal point based target function reduction procedures $S_\nabla$ and $S_{\rm{prox}}$ respectively. To handle the underdetermined system by CG we consider
the regularized least-squares problem \eqref{eq:reg-least-squares} and an appropriate choice of $\mu$ as previously described.

The nonnegative least-squares problem is addressed by superiorizing the projected Landweber algorithm. To additionally steer the iterates of a basic algorithm
for the  least-squares problem towards the nonnegative orthant we employ perturbations by constrained proximal points with $S_{\rm{prox+}}$ .

The algorithms, listed in Table \ref{tab:list-of-algorithms}, have parameters that need to be tuned.
To select these parameters, we evaluated the quality of the reconstructed image   $y_{k}$ using the relative error
$
\|y_k-x_{\ast}\|^2/n
$
and then choose, for each algorithm separately,
the parameters that provided the best reconstructed image within
the first 100 iterations.
The superiorization strategies share three parameters ($\kappa$, $a$ and 
$\gamma_0$), see, e.g., Algorithm \ref{eq:superiorized-grad-CG} and Algorithm \ref{eq:superiorized-prox-CG}. 
These were chosen from the following sets:
$\kappa \in\{5,10,20\}$, $a\in\{0.5,1-10^{-2}, 1-10^{-4}, 1-10^{-6}\}$,
$\gamma_0\in\{0.01,0.001,0.0025,\beta_1\}$, with $\beta_1 =1.9\lambda/ \|A\|^2$ and $\lambda$ from \eqref{eq:lambda-values}.
The combination of these parameters, that provide the best results are listed in Table \ref{tab:optimal-param}, while
the results are shown in Figure \ref{fig:values-superiorization-NC} and Figure \ref{fig:values-superiorization-WC}.
For comparison we also plot the results of the FBS algorithm \eqref{eq:intro-FBS} for the reversed splitting \eqref{eq:splitting-c-reversed}, in view of its 
intimate connection with the proximal-point based superiorized version of the Landweber algorithm, Algorithm \ref{eq:superiorized-prox-plus-LW}, discussed in
Proposition \ref{prop:rel-LW-PG}. Note that the best found values for  parameter $a$ are very close to $1$ as Proposition \ref{prop:rel-LW-PG} suggests.
Note that \texttt{ProxSupCG}
has the best error values among all superiorized versions of the basic algorithms and that it approaches optimal values  in the least amount of
time when it is adapted to incorporate nonnegativity constraints, by \texttt{ProxCSupCG}.
\begin{table}[h!]
\begin{center}
 \begin{tabular}{||c | c ||} 
 \hline
Algorithm & Optimal parameters \\ [0.5ex] 
 \hline\hline
 \texttt{GradSupCG} & $(a,\gamma_0,\kappa)= (1-10^{-4},0.001,20)$ \\ 
 \hline
\texttt{ProxSupCG}, \texttt{ProxSubLW} &  $(\gamma_0,a) = (0.001,1-10^{-6})$ \\
 \hline
\texttt{ProxCSupCG}, \texttt{ProxCSubLW}, \texttt{ProxSupProjLW} & $(\gamma_0,a) = \frac{1.9\lambda}{\|A\|^2},1-10^{-6})$ \\ 
\hline
 \texttt{GradSubLW}, \texttt{GradSupProjLW}  & $(a,\gamma_0,\kappa)= (1-10^{-4},0.0025,20)$ \\[1ex] 
 \hline
\end{tabular}
\end{center}
\caption{Optimal parameters for the superiorized versions of the basic algorithms.}
\label{tab:optimal-param}
\end{table}

\begin{figure}
\centerline{
\includegraphics[width=0.45\textwidth]{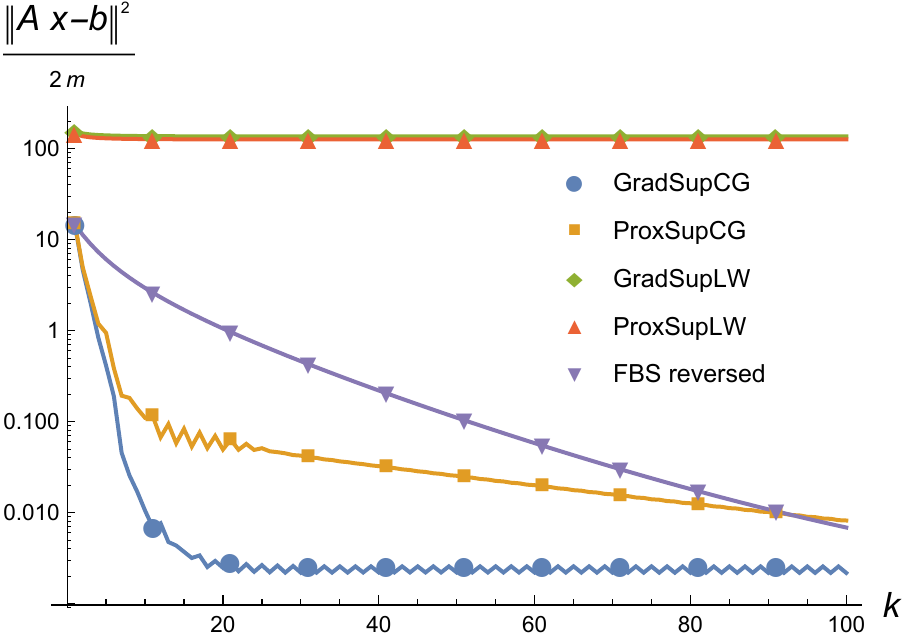}
\hfill
\includegraphics[width=0.45\textwidth]{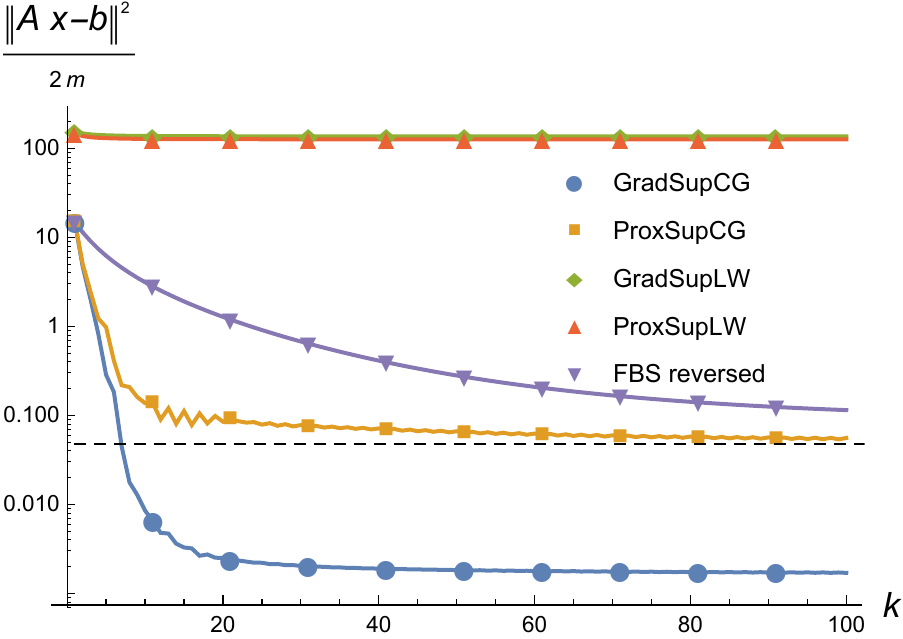}
}
\centerline{
\includegraphics[width=0.45\textwidth]{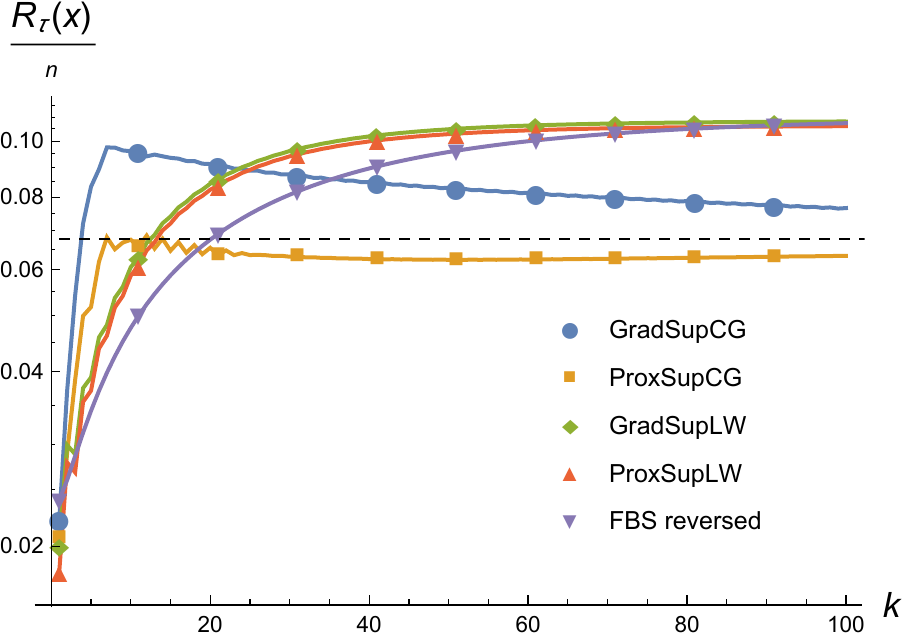}
\hfill
\includegraphics[width=0.45\textwidth]{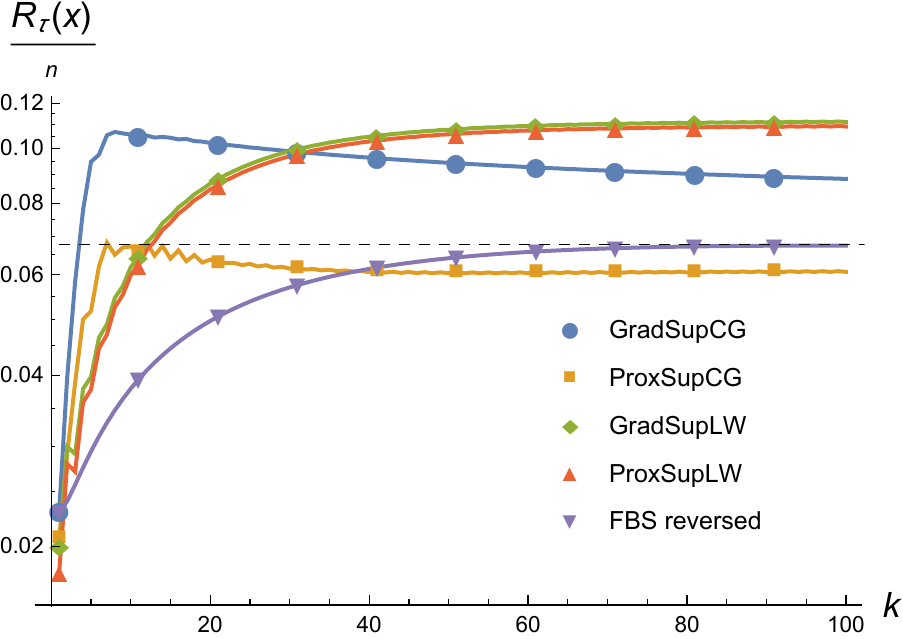}
}
\centerline{
\includegraphics[width=0.45\textwidth]{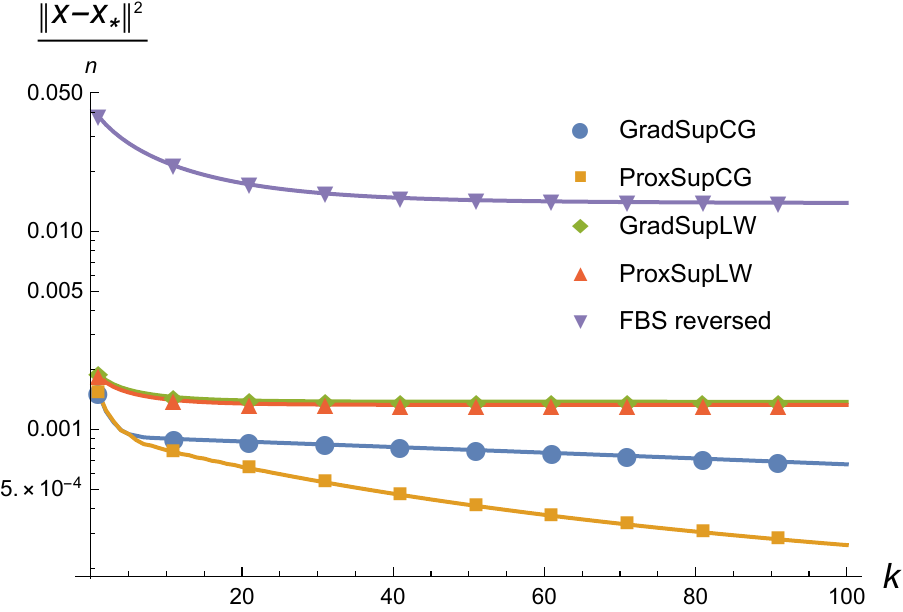}
\hfill
\includegraphics[width=0.45\textwidth]{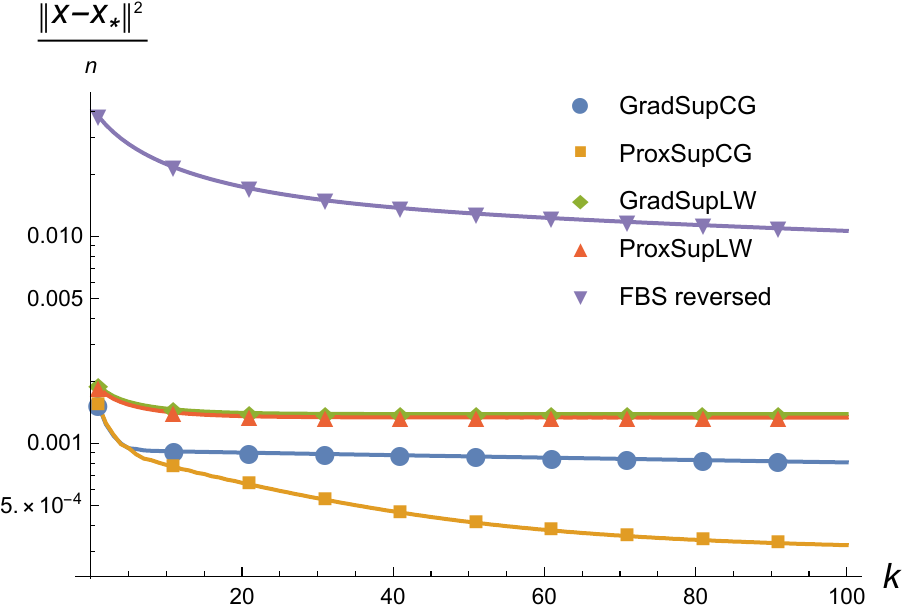}
}
\caption{\emph{Superiorization: Unconstrained Least-Squares.}
Scaled values of the data term (top row) and the regularizer (middle row) of the objective \eqref{eq:hu-hc} and the squared error norm (bottom row), as a function of  the iteration index $k$ of the superiorized versions of the  CG algorithm using perturbations based on gradients, see Algorithm  \ref{alg:sup-classic} and on proximal points, see \eqref{alg:sup-prox}, the negative gradients and proximal-point based superiorized versions of the Landweber algorithm and the FBS algorithm \eqref{eq:intro-FBS} for the reversed splitting \eqref{eq:splitting-c-reversed}.
The horizontal dashed lines indicate the corresponding values for $x_{\ast}$. The left and right columns correspond to exact and noisy data, respectively. Neither the basic algorithms nor their superiorized versions  incorporate nonnegativity constraints. After appropriate parameter tuning the superiorized CG iteration approaches the solution of \eqref{eq:hu-hc} significantly faster than the FBS algorithm. The superiorized version of the Landweber algorithm  makes slow progress due to very small step-sizes in view of the large value of $\|A\|^2$.
}
\label{fig:values-superiorization-NC}
\end{figure}


\begin{figure}
\centerline{
\includegraphics[width=0.45\textwidth]{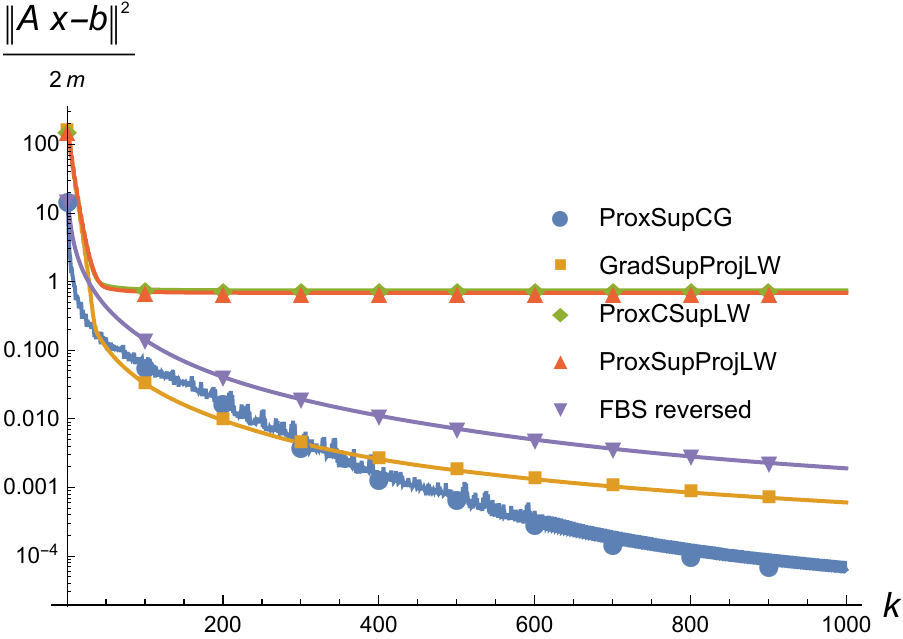}
\hfill
\includegraphics[width=0.45\textwidth]{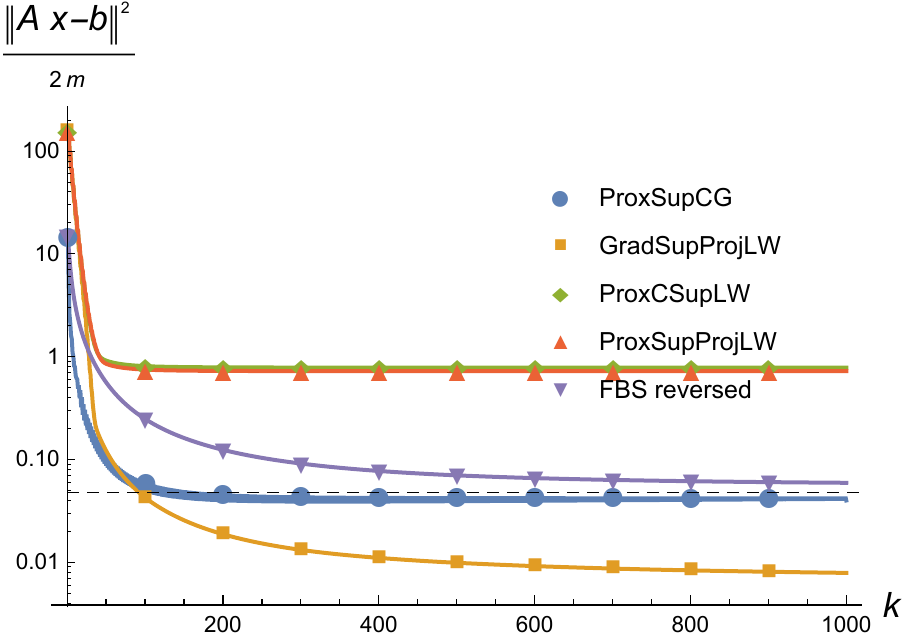}
}
\centerline{
\includegraphics[width=0.45\textwidth]{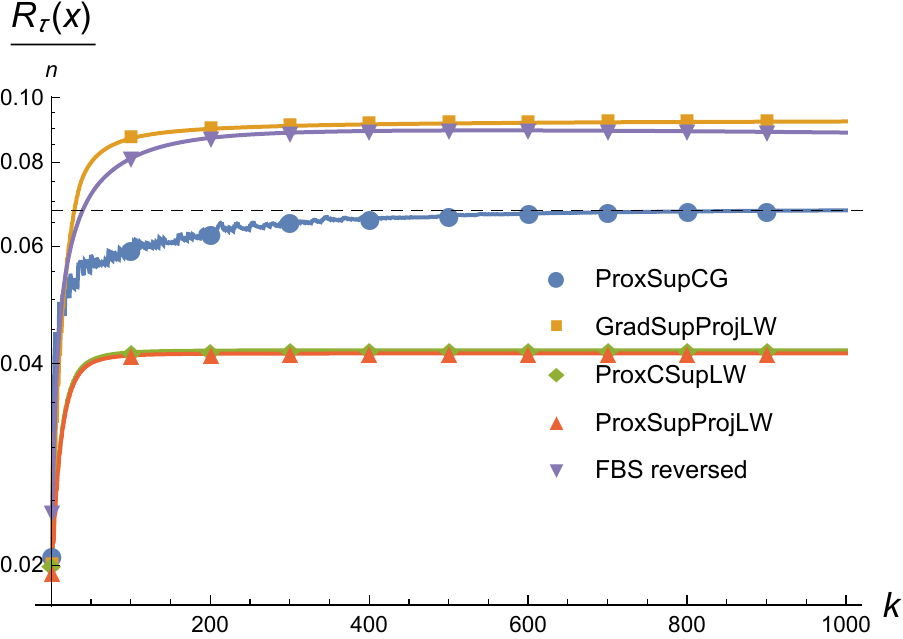}
\hfill
\includegraphics[width=0.45\textwidth]{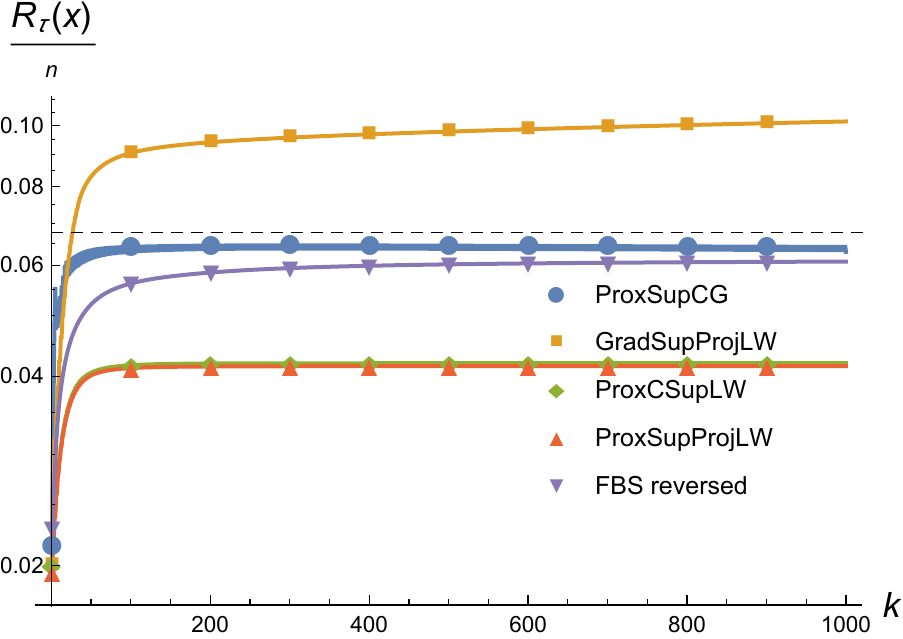}
}
\centerline{
\includegraphics[width=0.45\textwidth]{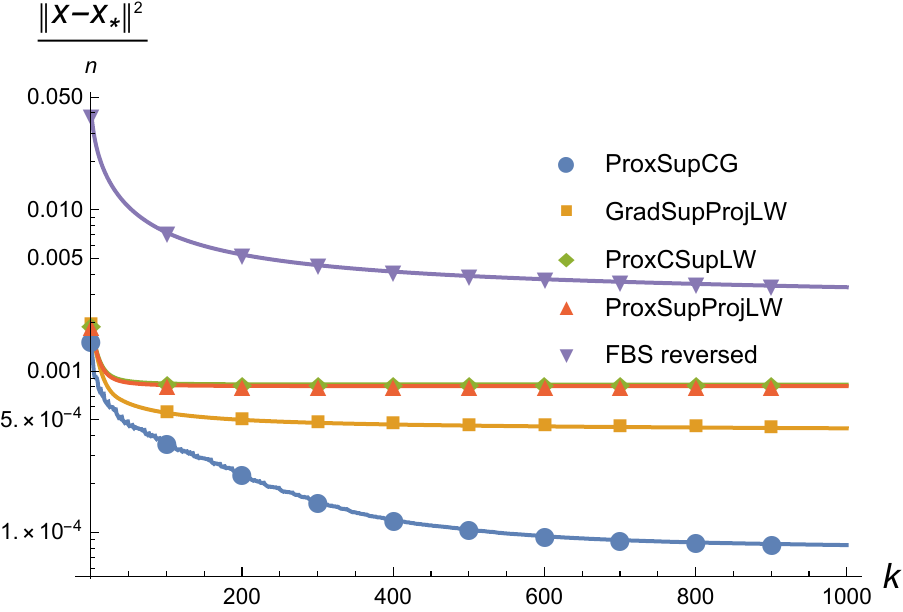}
\hfill
\includegraphics[width=0.45\textwidth]{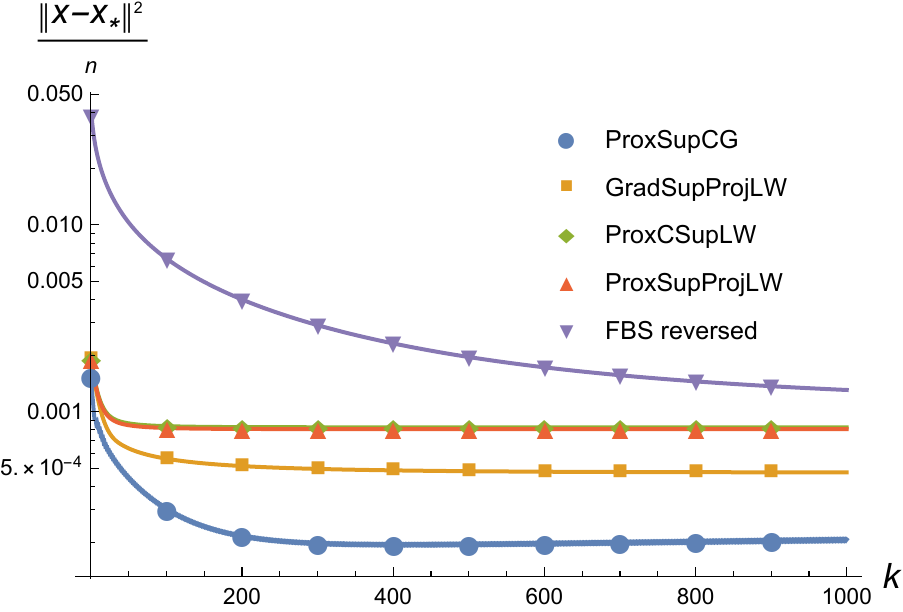}
}
\caption{\emph{Superiorization: Adding Nonnegative Constraints.}
Scaled values of the data term (top row) and the regularizer (middle row) of the objective \eqref{eq:hu-hc} and the squared error norm (bottom row), as a function of  the 
iteration index $k$ of the superiorized versions of the CG algorithm \ref{eq:superiorized-prox-CG} and the Landweber algorithm \ref{eq:superiorized-prox-plus-LW} using  
proximal points, 
see \eqref{alg:sup-prox-plus}, the negative gradient and proximal point based superiorized version of the  projected Landweber algorithm and the FBS algorithm 
\eqref{eq:intro-FBS} for the reversed splitting \eqref{eq:splitting-c-reversed}. The horizontal  dashed lines  indicate the corresponding values for $x_{\ast}$. The left and right columns correspond to exact and noisy data, respectively. Constraints are incorporated either in the basic algorithm  (\texttt{GradSupProjLW, ProxSupProjLW}) or via the superiorization strategy  (\texttt{ProxCSupCG, ProxCSubLW}) by constraining proximal points.   Superiorized versions of the Landweber algorithm perform similarly making slow progress towards $x_\ast$ due to the ill-conditioned matrix $A$, with a surprising better performance of the gradient based superiorized versions of the projected Landweber iteration \texttt{GradSupProjLW}. Again, the superiorized version of the 
CG algorithm approaches the solution of \eqref{eq:hu-hc} significantly faster than the FBS algorithm for the reversed splitting \eqref{eq:splitting-c-reversed} after appropriate parameter tuning.
}
\label{fig:values-superiorization-WC}
\end{figure}


\textbf{Optimization.}
Figure \ref{fig:values-optimization} shows the results of comparing the  FBS algorithm \eqref{eq:intro-FBS} 
with the accelerated FBS iteration  in Algorithm \ref{alg:FBS}
for solving \eqref{eq:hu-hc} without nonnegativity constraints. To obtain a fair comparison, the proximal maps were evaluated \textit{exactly} using formula \eqref{eq:Balpha-Woodbury}.  Depending on the value of $\lambda$ that affects the Lipschitz constant, 
i.e., a small value of $\lambda$ for exact data and a larger value  of $\lambda$ for noisy data, more iterations are required in the latter case. And acceleration pays: significantly fewer iterations are then required. 

Since the values of $\lambda$ were chosen such that minimizers match the properties \eqref{eq:lambda-values} of $x_{\ast}$, it seems fair to claim that -- from the viewpoint of optimization -- our implementation  structurally resembles the problem decomposition of superiorization, i.e.,~least-squares minimization adjusted by gradients of the target function $R_{\tau}$. Inspecting the plots of Figure  \ref{fig:values-optimization}, we notice that the final values are almost reached after merely $\approx 75$ outer iterations. 

Similarly, the accelerated FBS algorithm \eqref{eq:intro-FBS} for the reversed splitting \eqref{eq:splitting-c-reversed} reaches error term values close to final values within the
first $100$ iterations, at significantly lower cost as discussed next. Results for the reversed splitting are shown in Figure~\ref{fig:values-optimization-SPG-NC} and Figure~\ref{fig:values-optimization-SPG-WC} for the unconstrained and  nonnegativity constrained case, respectively.

\begin{figure}
\centerline{
\includegraphics[width=0.32\textwidth]{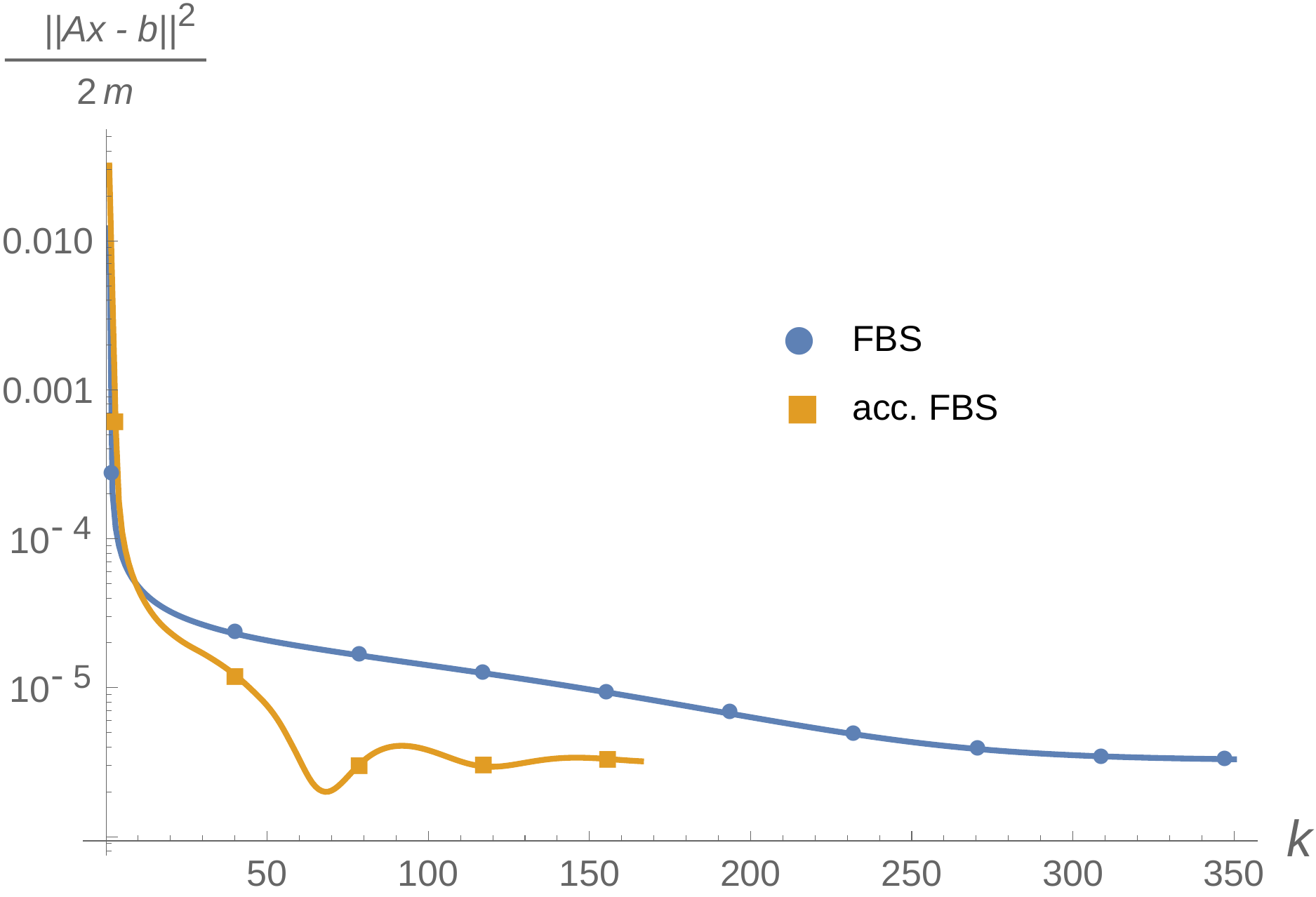}
\hfill
\includegraphics[width=0.32\textwidth]{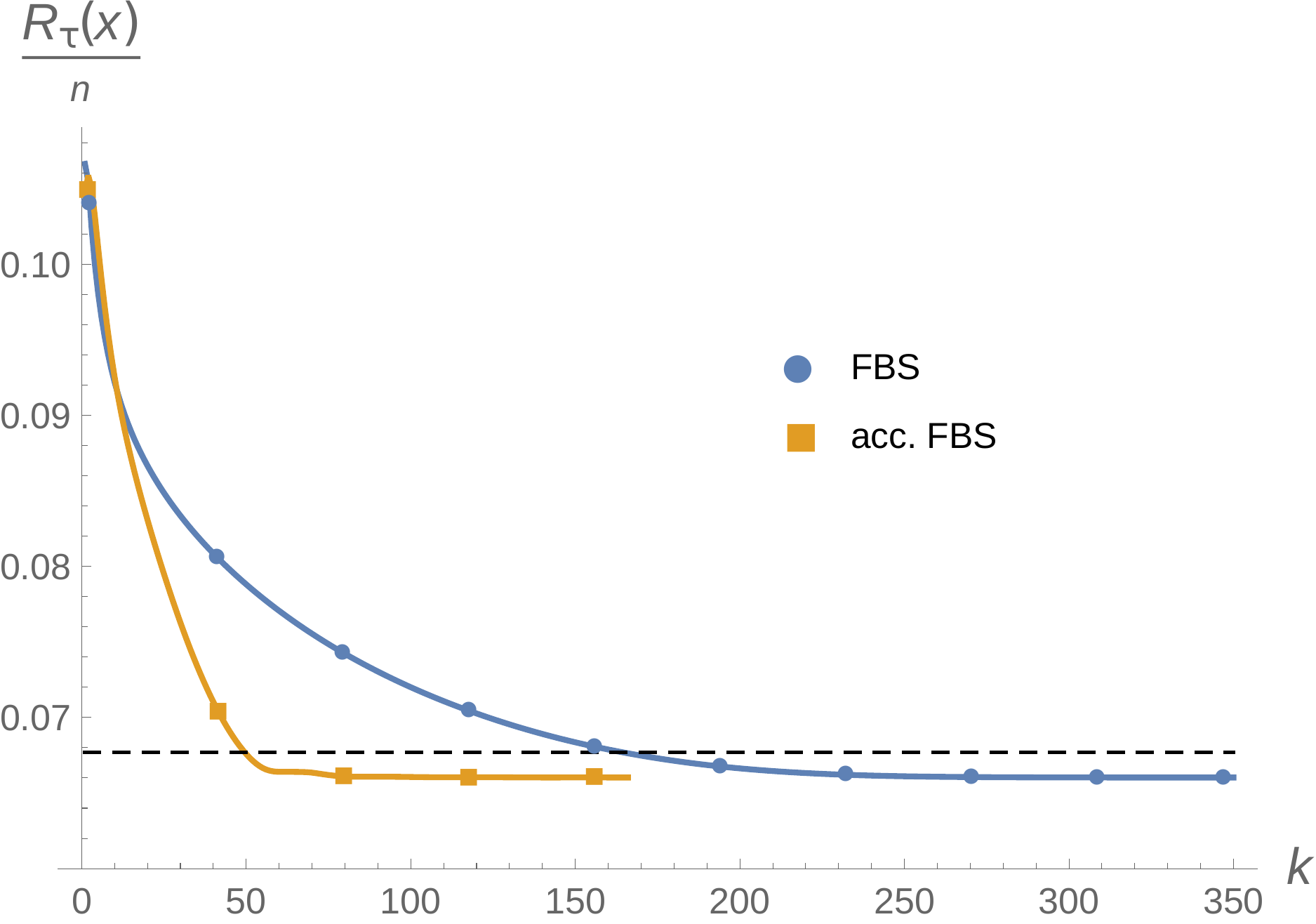}
\hfill
\includegraphics[width=0.32\textwidth]{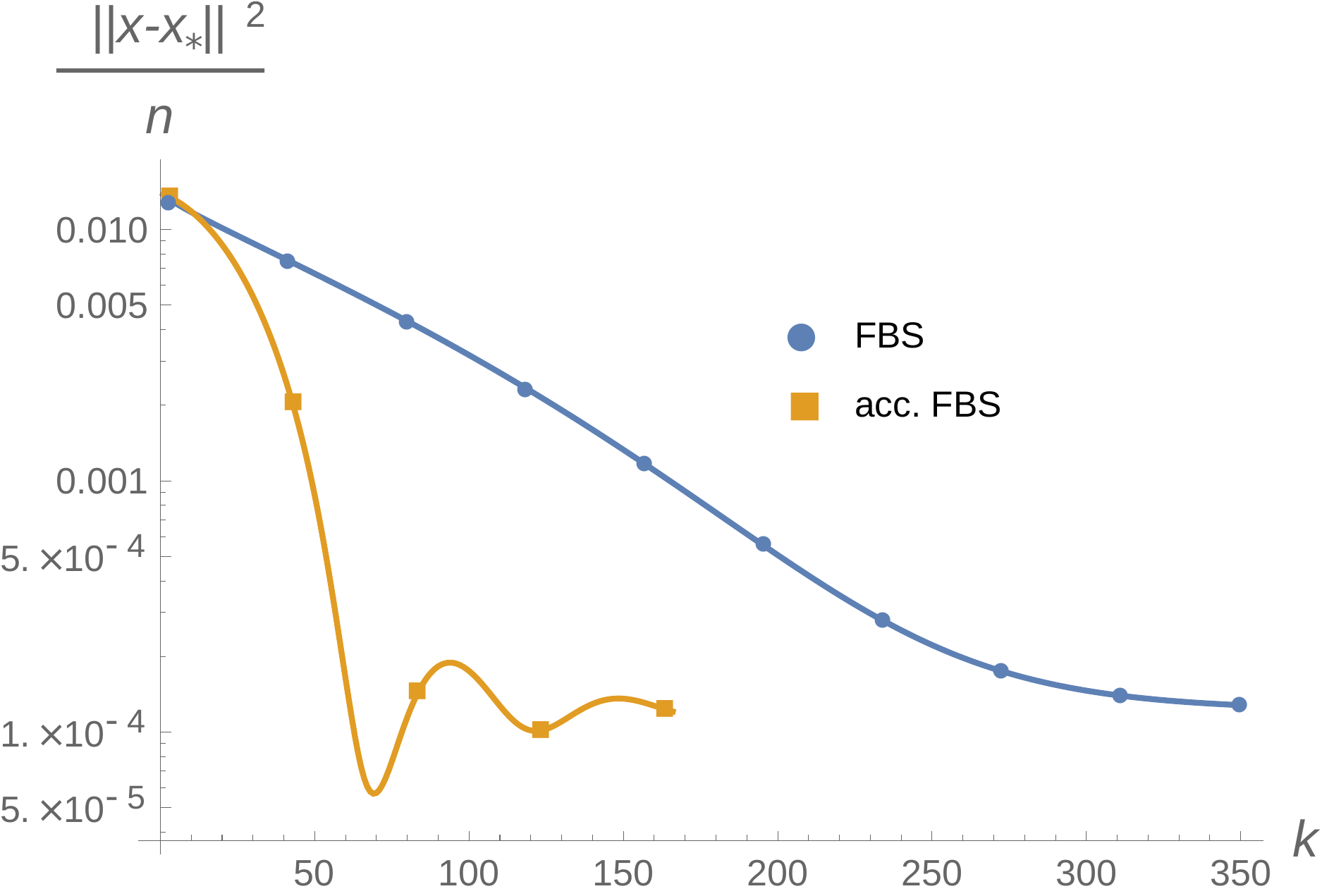}
}
\centerline{
\includegraphics[width=0.32\textwidth]{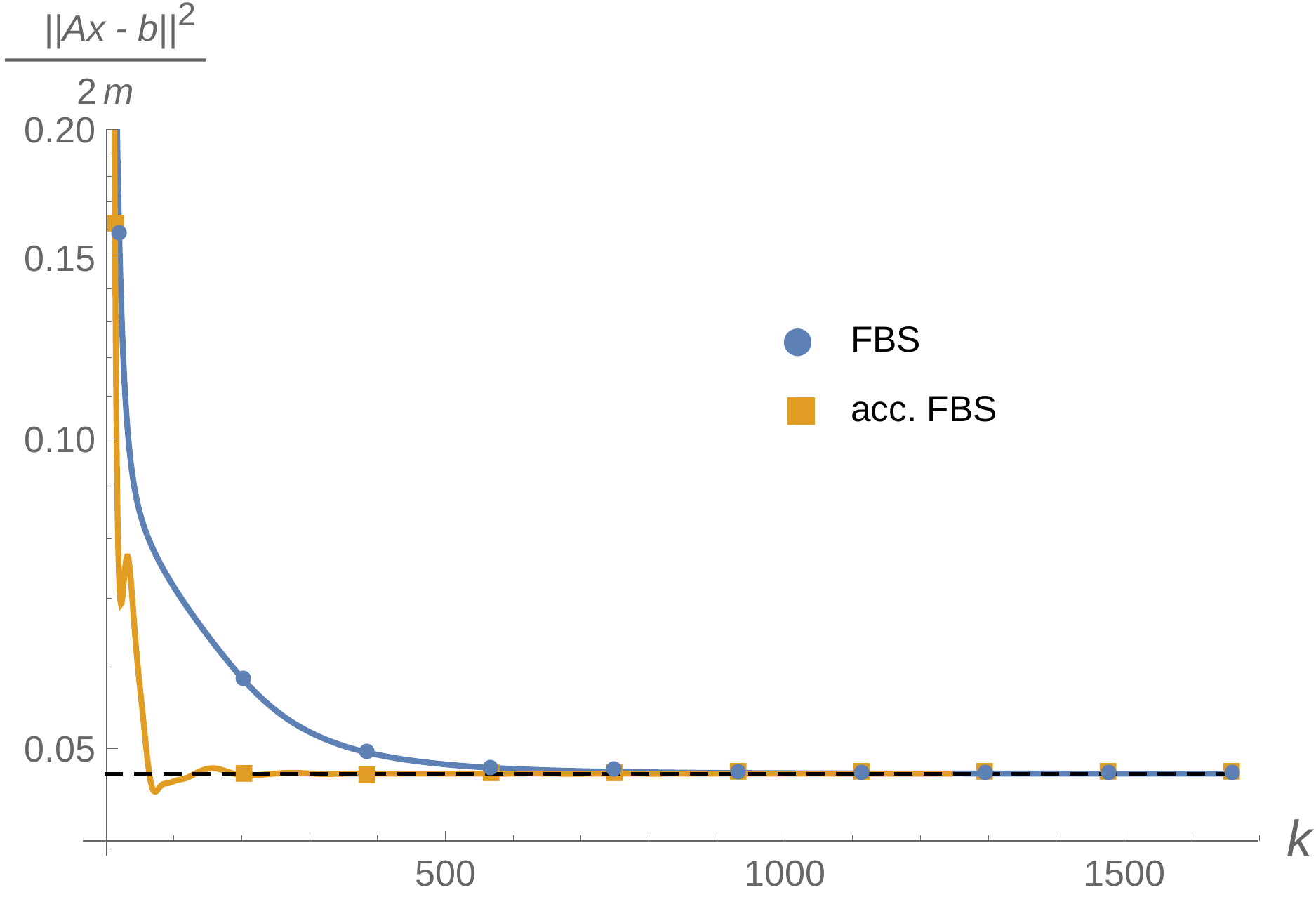}
\hfill
\includegraphics[width=0.32\textwidth]{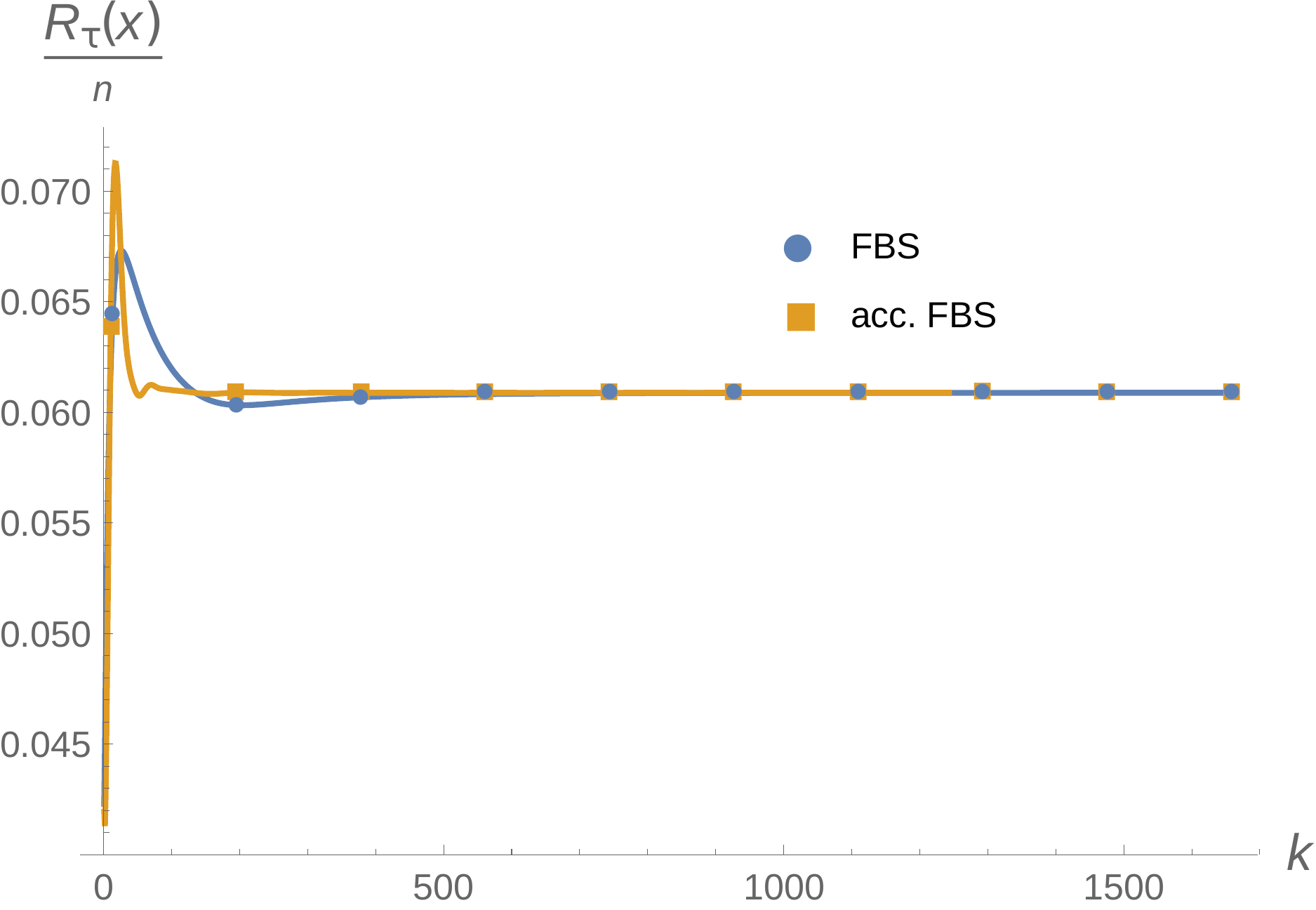}
\hfill
\includegraphics[width=0.32\textwidth]{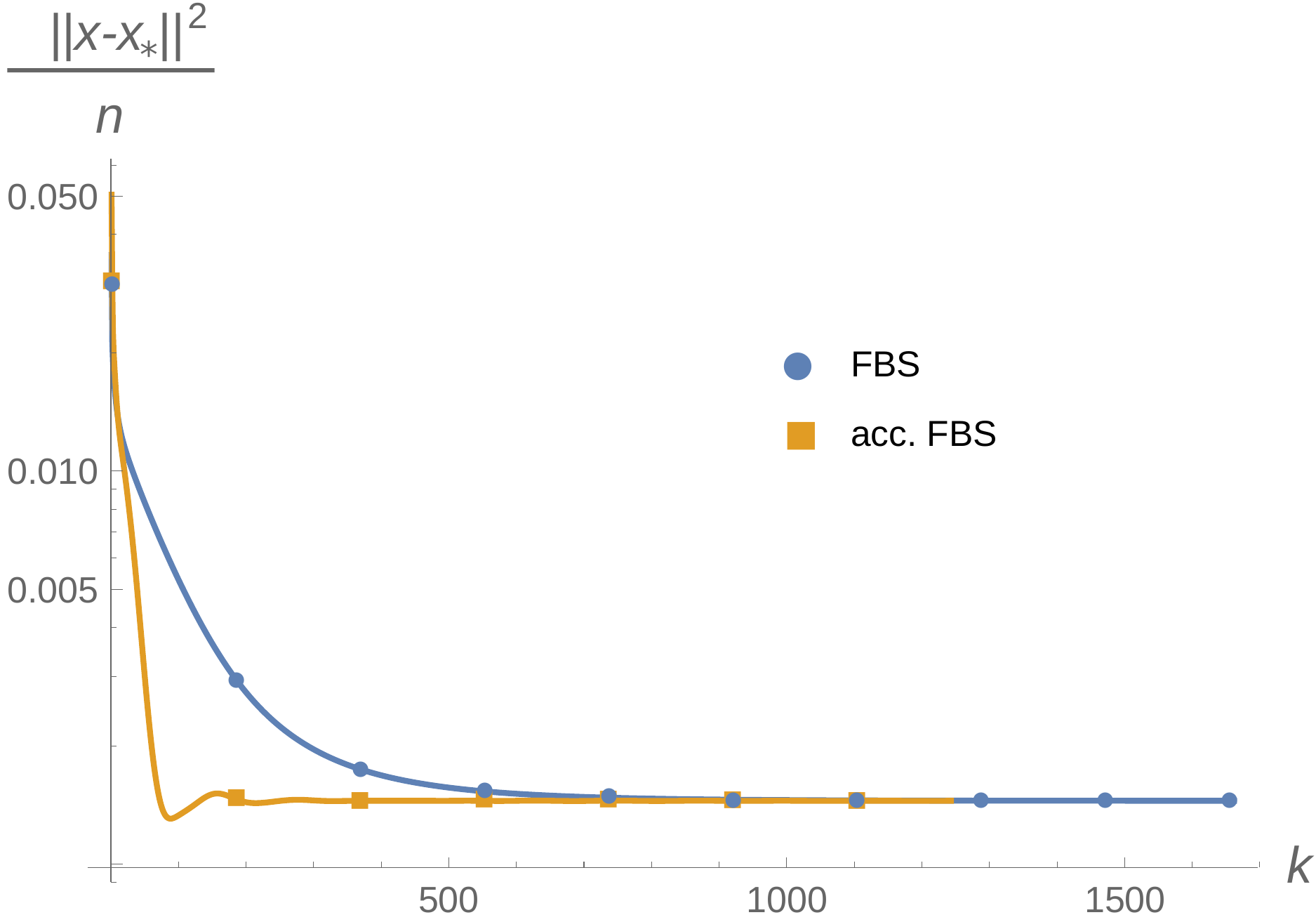}
}
\caption{\emph{Optimization:  FBS and accelerated FBS algorithms -- No Constraints.}
Scaled values of the data term (left column) and the regularizer (center column) of the objective \eqref{eq:hu-hc} and the squared error norm (right column), as a function of  the iteration index $k$ of the  FBS algorithm \eqref{eq:intro-FBS} (blue curves) and the accelerated FBS iteration in Algorithm \ref{alg:FBS}. The horizontal dashed lines indicate the corresponding values for $x_{\ast}$. The top and bottom rows correspond to exact and noisy data, respectively: The accelerated FBS algorithm needs about $50\%$ and $25\%$ fewer outer iterations, respectively. 
}
\label{fig:values-optimization}
\end{figure}

\begin{figure}
\centerline{
\includegraphics[width=0.32\textwidth]{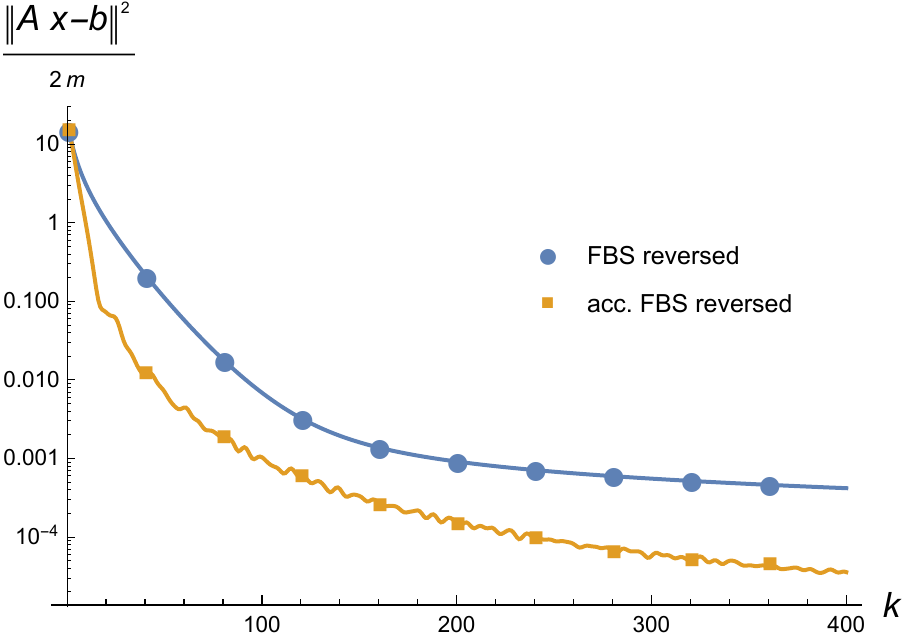}
\hfill
\includegraphics[width=0.32\textwidth]{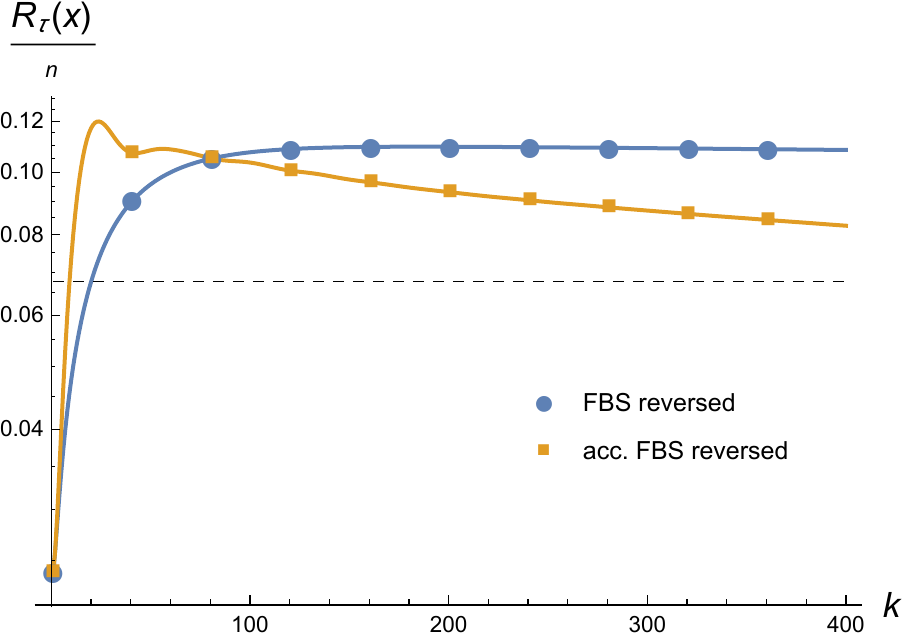}
\hfill
\includegraphics[width=0.32\textwidth]{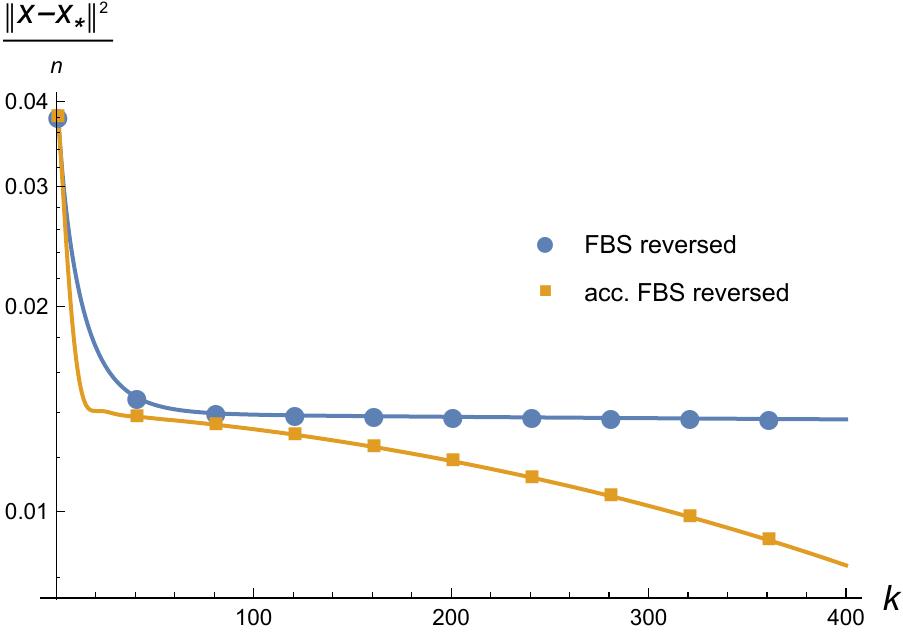}
}
\centerline{
\includegraphics[width=0.32\textwidth]{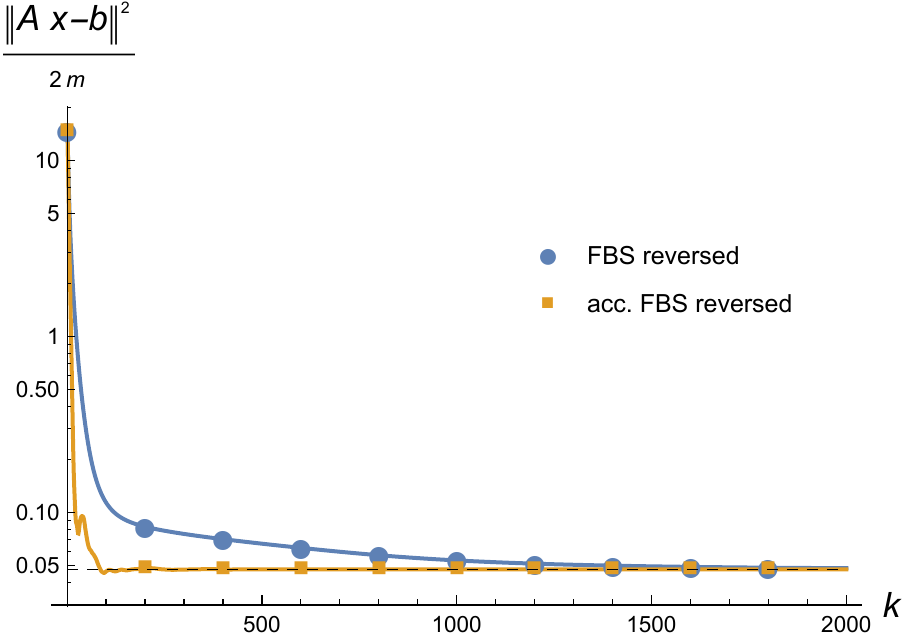}
\hfill
\includegraphics[width=0.32\textwidth]{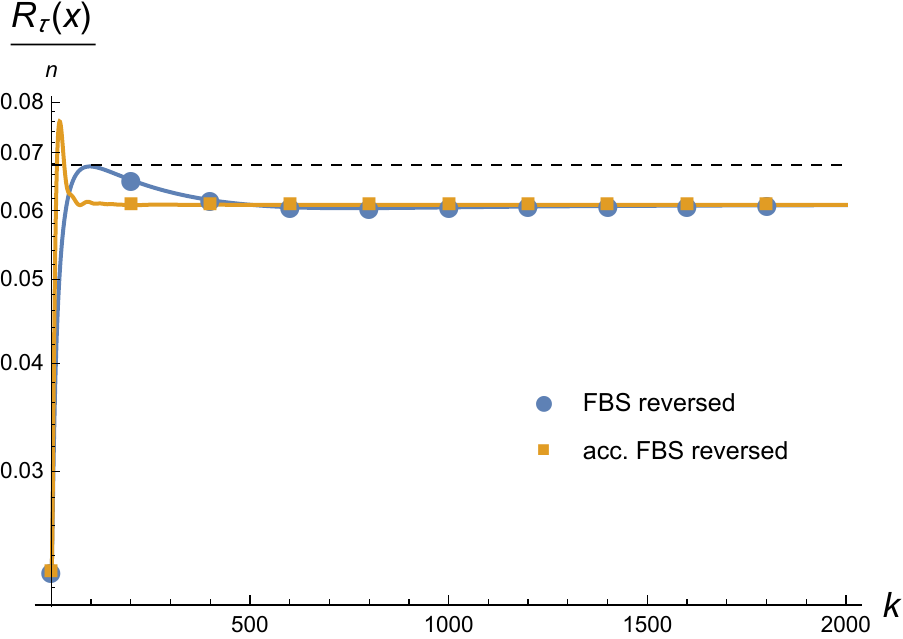}
\hfill
\includegraphics[width=0.32\textwidth]{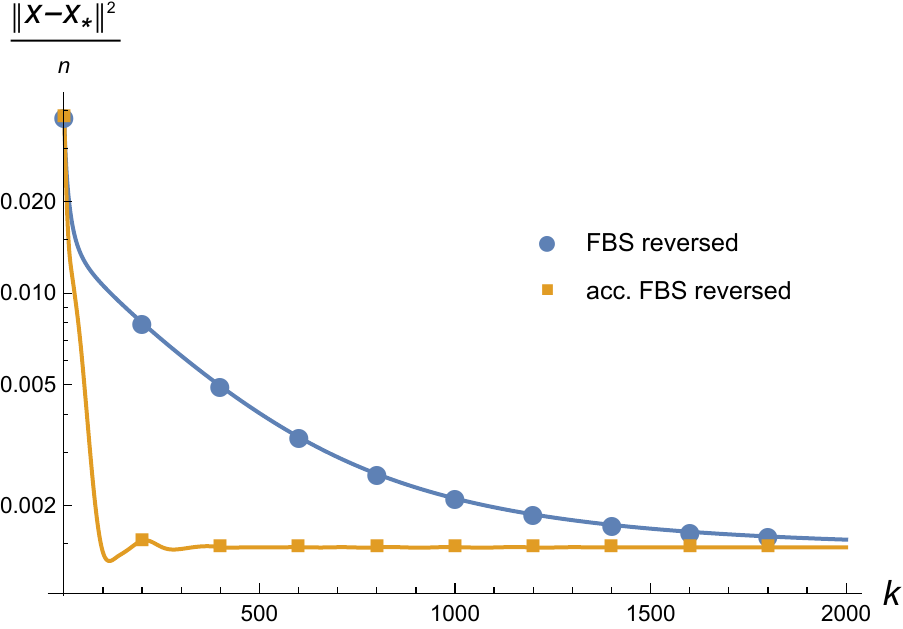}
}
\caption{\emph{Optimization:  FBS and accelerated FBS algorithm for the reversed splitting -- No Constraints.} 
Scaled values of the data term (left column) and the regularizer (center column) of the \emph{unconstrained} objective $h_u$ \eqref{eq:hu-hc} and the squared error norm (right column), as a function of the iteration indes $k$ of the FBS algorithm \eqref{eq:intro-FBS} (blue curves) and the accelerated FBS iteration in Algorithm \ref{alg:FBS} for the reversed splitting \eqref{eq:splitting-c-reversed}.  The horizontal dashed lines indicate the corresponding values for $x_{\ast}$. The top and bottom rows correspond to exact and noisy data, respectively.
}
\label{fig:values-optimization-SPG-NC}
\end{figure}

\begin{figure}
\centerline{
\includegraphics[width=0.32\textwidth]{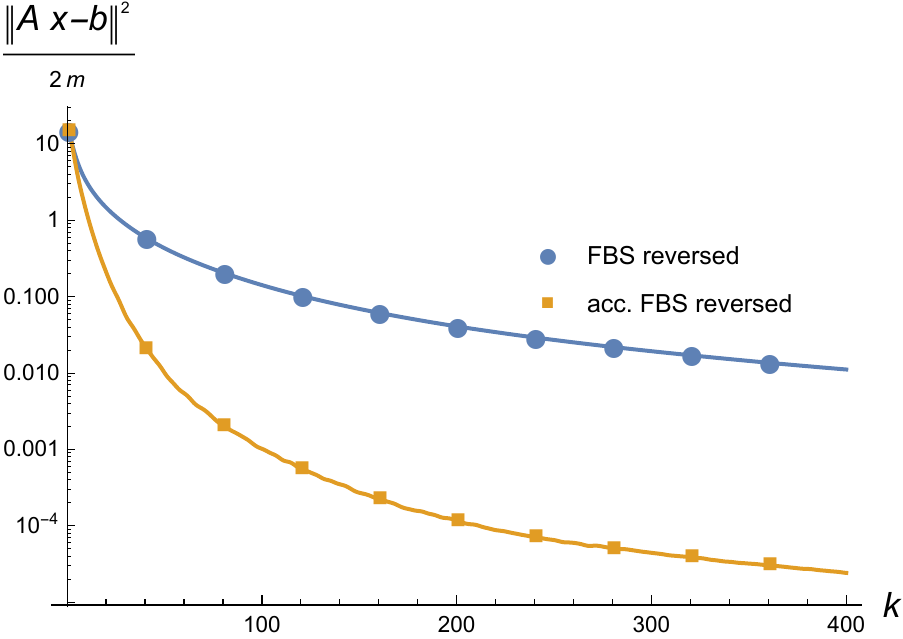}
\hfill
\includegraphics[width=0.32\textwidth]{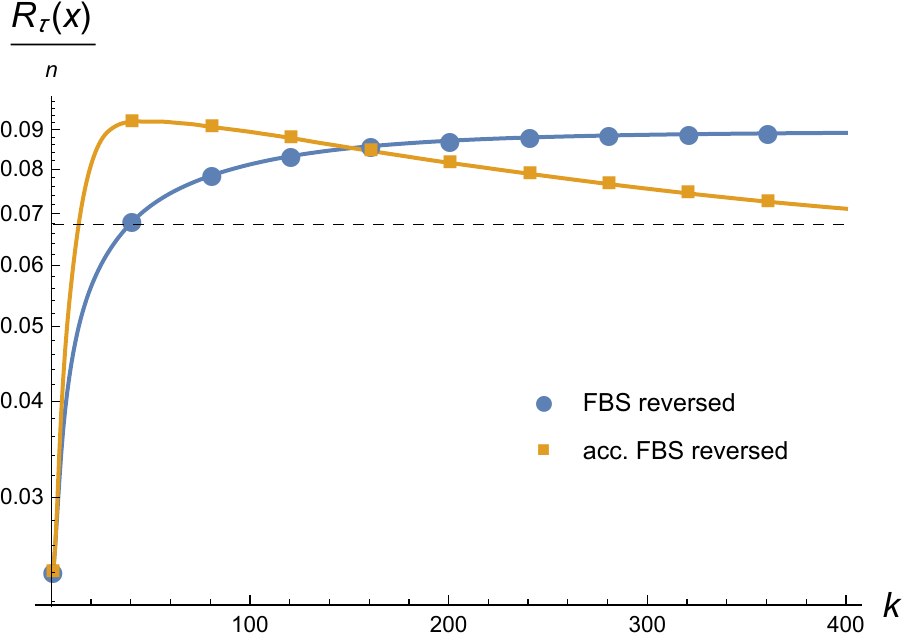}
\hfill
\includegraphics[width=0.32\textwidth]{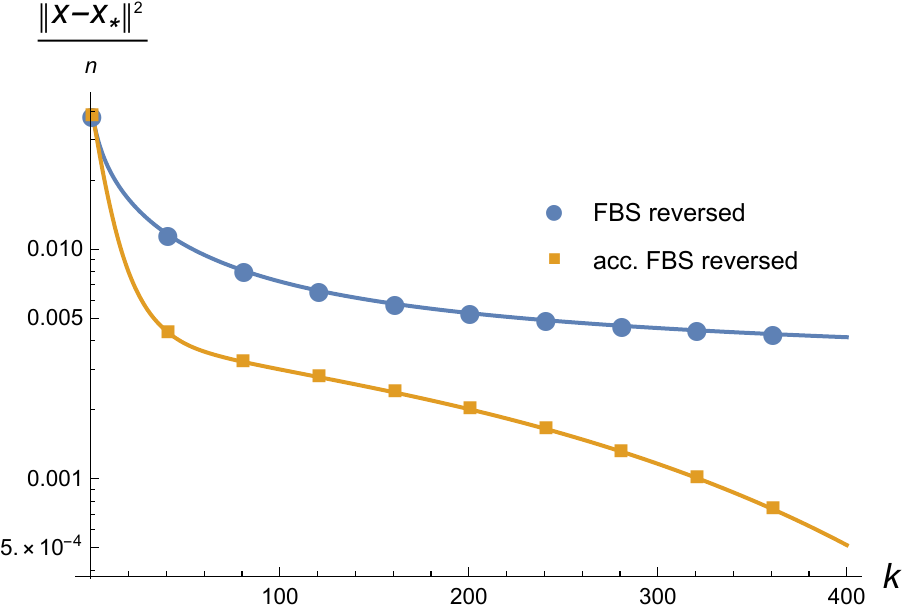}
}
\centerline{
\includegraphics[width=0.32\textwidth]{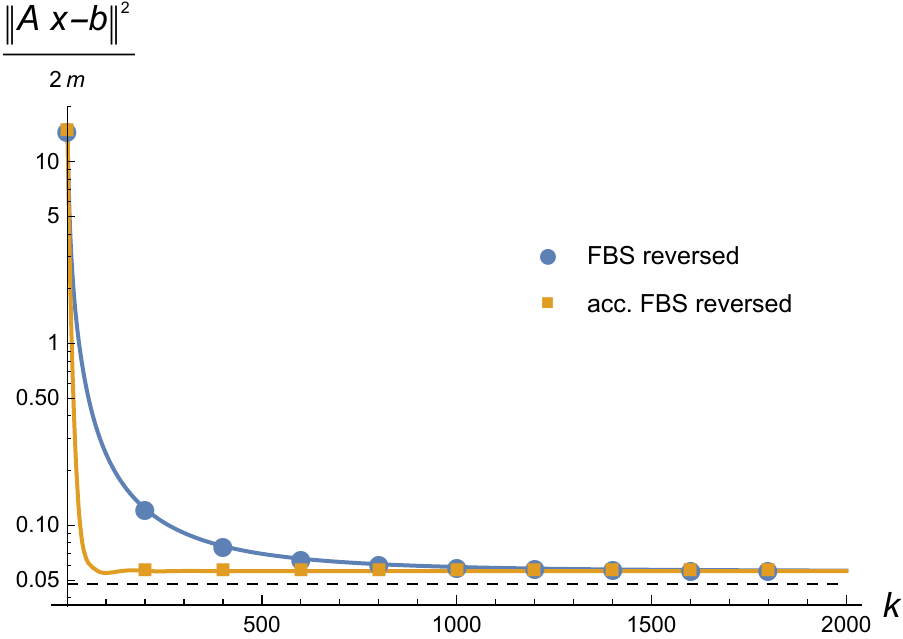}
\hfill
\includegraphics[width=0.32\textwidth]{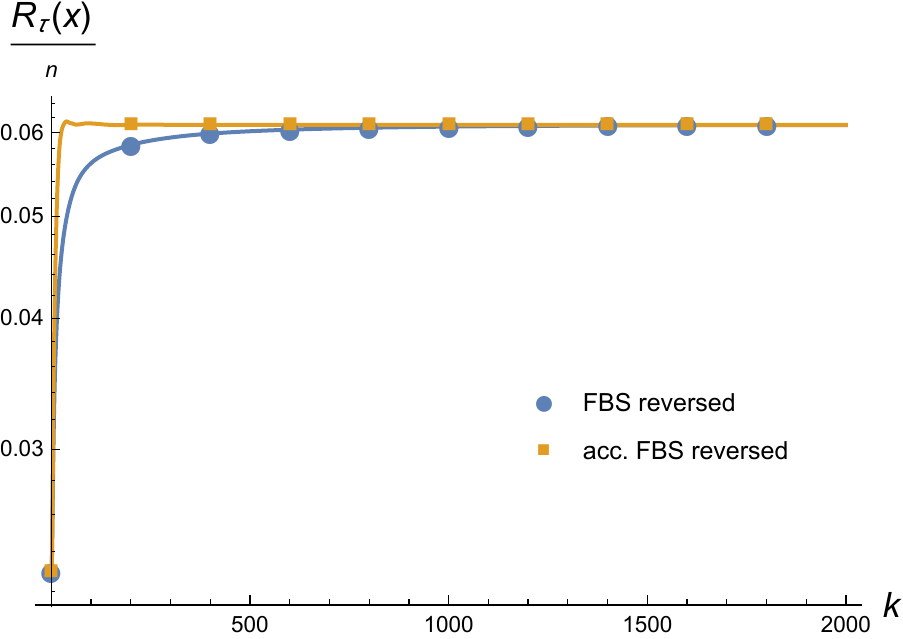}
\hfill
\includegraphics[width=0.32\textwidth]{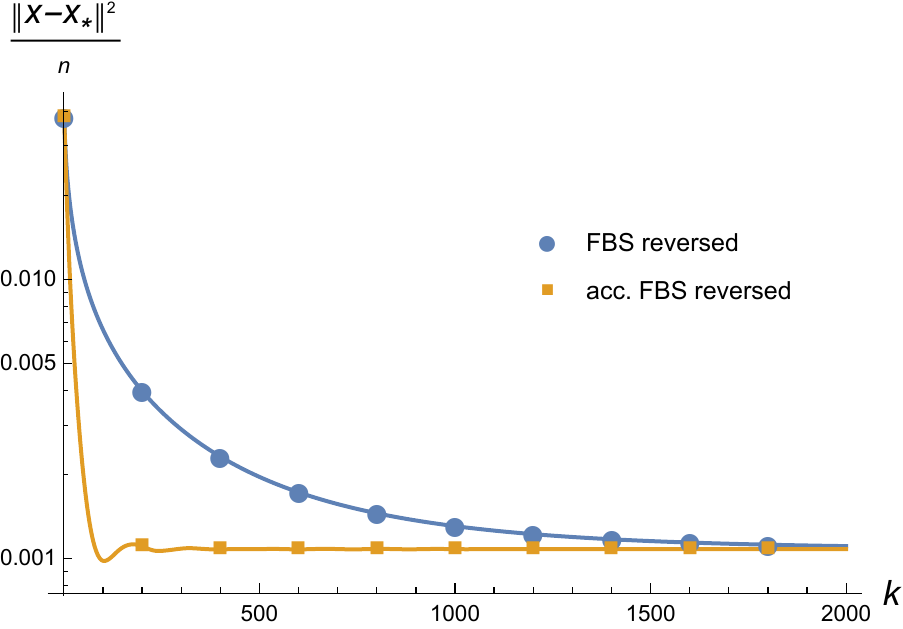}
}
\caption{\emph{Optimization:  FBS and accelerated FBS iteration for the reversed splitting -- Nonnegativity Constraints.} 
Scaled values of the data term (left column) and the regularizer (center column) of the \emph{nonnegativity constrained} objective $h_c$ \eqref{eq:hu-hc} and the squared error norm (right column), as a function of  the iteration index $k$ of the FBS algorthm (blue curves) and the  accelerated FBS algorithm for the reversed splitting \eqref{eq:splitting-c-reversed}. The horizontal dashed lines indicate the corresponding values for $x_{\ast}$. The top and bottom rows correspond to exact and noisy data, respectively. While the  FBS algorithm is damped by the inclusion of constraints, the nonnegativity constraints appear to be beneficial for the accelerated FBS algorithm on the other hand when comparing to the results in Figure~\ref{fig:values-optimization-SPG-NC}.
}
\label{fig:values-optimization-SPG-WC}
\end{figure}

\subsection{Superiorization vs.~Optimization: Computational Complexity}\label{sec:Complexity}
We look at the computational costs of the superiorization and optimization approaches.
\textbf{Superiorization.} To evaluate the cost of the superiorized versions of the basic algorithms, listed in Table \ref{tab:list-of-algorithms}, we call an outer iteration an ``iteration of the
basic algorithm" $\mc{A}(\cdot)$ and an inner iteration ``an execution of the target function reduction procedure'' $S(\cdot)$ by Algorithm \ref{alg:sup-classic} or by a proximal map via \eqref{alg:sup-prox} or \eqref{alg:sup-prox-plus}.
Within an outer iteration we have for the CG iteration in Algorithm \ref{alg:CG-pert-res} four costly matrix-vector operations in line \ref{CG_grad_update} and line 
\ref {CG_matrix_eval_2}, while for the Landweber Algorithm \ref{alg:Landweber} only two costly matrix-vector operations appear in line \ref{LW_matrix_eval_1}.
The cost of an inner iteration depends on the number of gradient and function evaluations of $R_\tau$. In the case of computing nonascent directions by scaled nonnegative gradients,
see, e.g., line \ref{eq:CG-nonsascent-step}  in Algorithm \ref{eq:superiorized-grad-CG}, the number of gradient and function evaluation is controlled by the parameter $\kappa$ and is upper bounded by $(\ell_k-\ell_{k-1})\kappa$.
Similarly, the cost of a proximal mapping evaluation via \eqref{alg:sup-prox} or \eqref{alg:sup-prox-plus} depends on the tolerance parameter of the
employed optimization algorithm. In our case, the box-constrained L-BFGS method from \cite{LBFGS-B} is very efficient in reaching a tolerance level of $10^{-6}$ within
few iterations due to the small values of the parameters $\beta_k$ used. In particular, we need 3--18 inner iterations while using at most 136 
function evaluations of $R_\tau$. This is \emph{comparable} to the cost of Algorithm \ref{alg:sup-classic}  for the considered choice of parameters.
We underline that the low cost of the proximal-point based Landweber algorithm is almost identical to the FBS algorithm \eqref{eq:intro-FBS}  and the accelerated FBS iteration in Algorithm
\ref{alg:FBS} for the reversed splitting \eqref{eq:splitting-c-reversed}.

\textbf{Optimization.}
We evaluated accelerated the FBS algorithm with \textit{inexact} evaluations of the proximal maps, for both the unconstrained and nonnegativity  constrained problem \eqref{eq:hu-hc}. In the former case, the inexactness criterion was evaluated during the primal-dual inner iteration using \eqref{eq:inexactness-unconstrained-final}. In the latter nonengatively constrained case, we noticed that the corresponding criterion \eqref{eq:inexactness-constrained-final} could not be used since $z \not\in K=\R^n_+$, with $z$ given by \eqref{eq:zp-z-Alg2-constrained-a}, happened frequently. We, therefore, resorted to the error estimate \eqref{eq:norm-e-constrained-estimate} and chose  $\veps_{k} = \mc{O}(k^{-q})$ so as to satisfy \eqref{eq:ek-summability}.

Figures \ref{fig:OuterInner-Unconstrained} and \ref{fig:OuterInner-Constrained} depict the corresponding results. Inspecting the panels on the left shows that both in the unconstrained and in the constrained case, choosing the proper error decay rate \eqref{eq:ek-decay-rate} significantly affects the accelerated FBS algorithm. Having chosen a proper exponent, an increasing number of inner iterations, ranging from few dozens to a couple of hundreds, are required for each outer iteration. Even though these inner iterations can be computed efficiently, their total number adds up to few hundreds of thousands for the entire optimization procedure.

\begin{figure}
\centerline{
\includegraphics[width=0.32\textwidth]{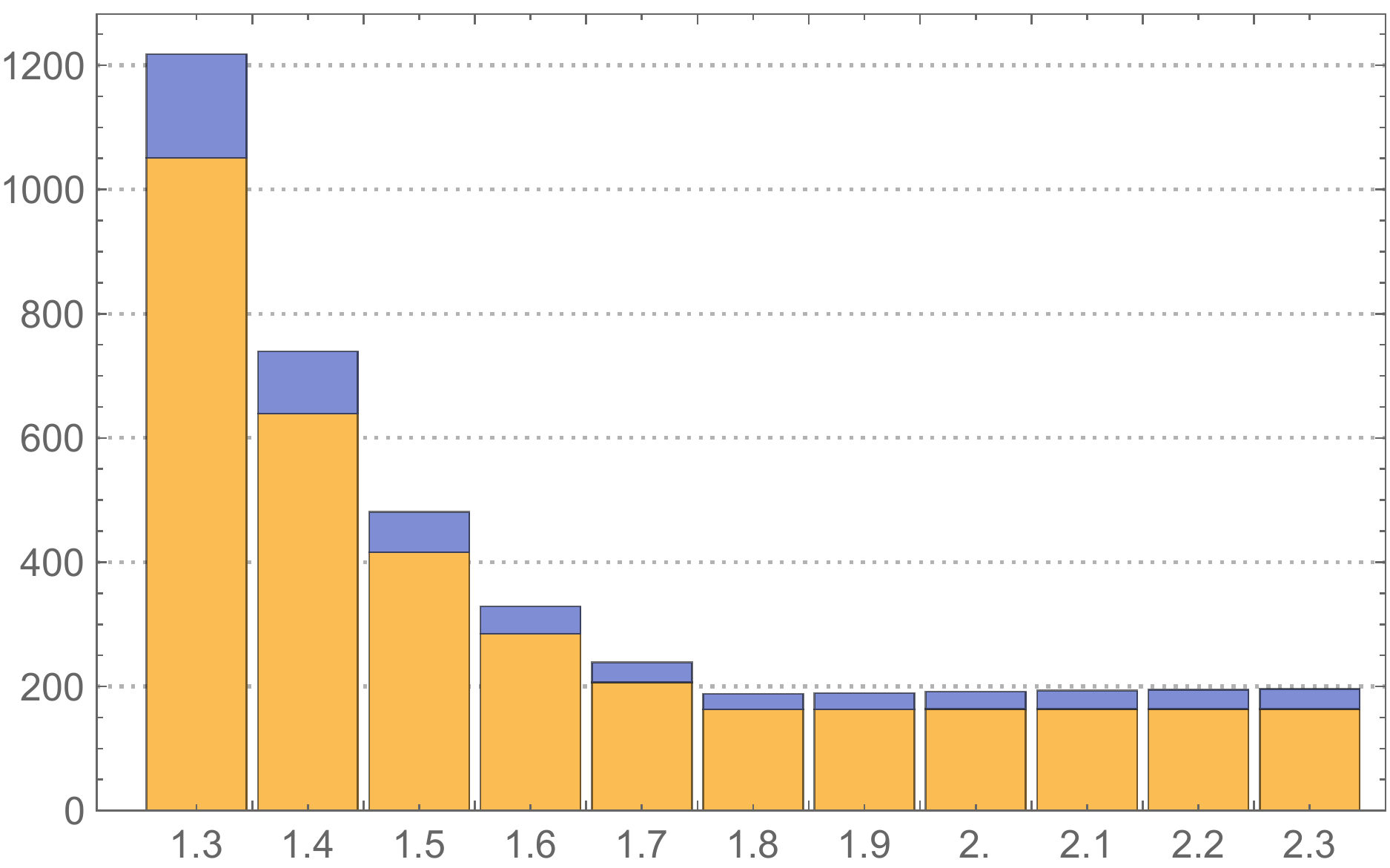}
\hspace{0.05\textwidth}
\includegraphics[width=0.32\textwidth]{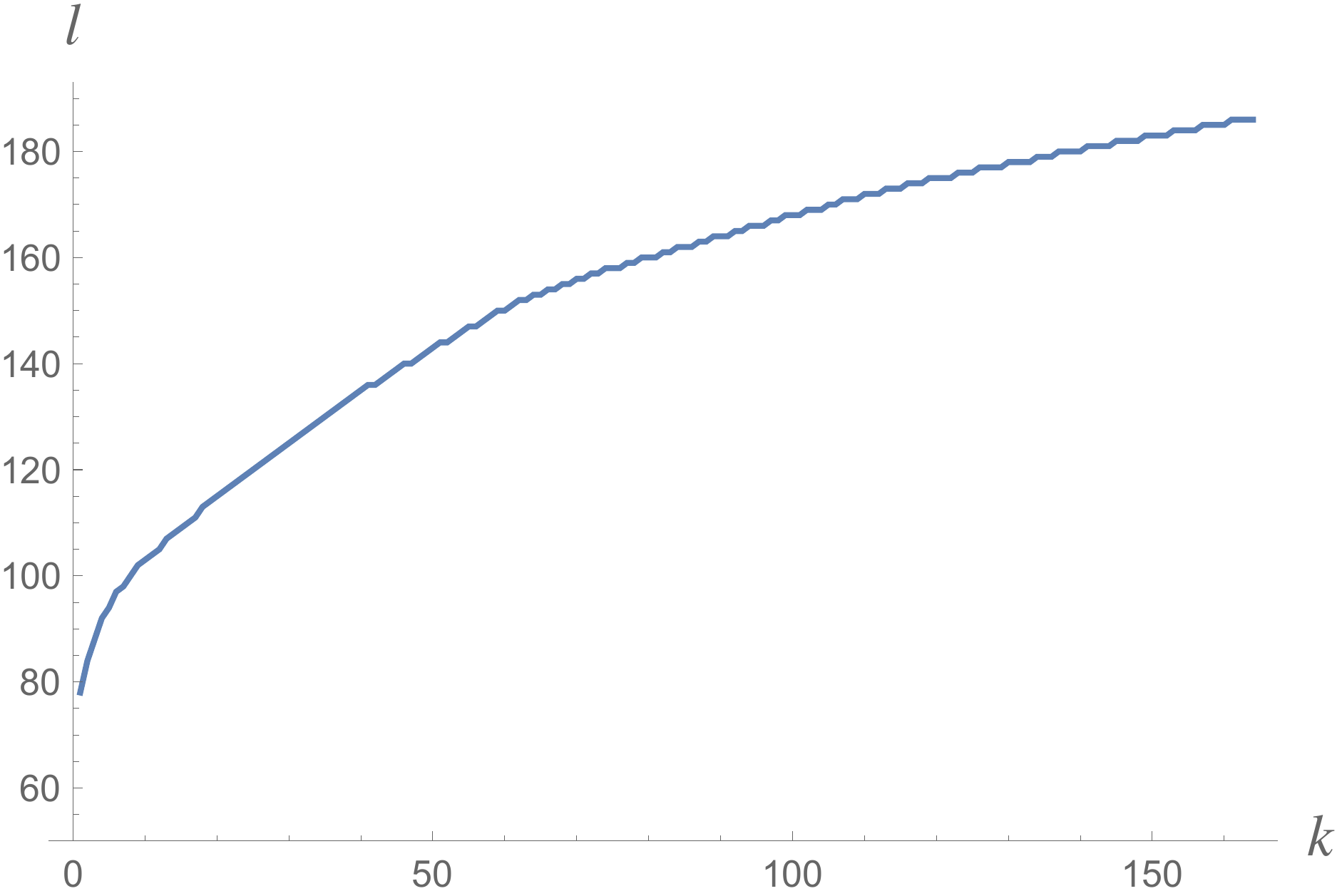}
}
\centerline{
\includegraphics[width=0.32\textwidth]{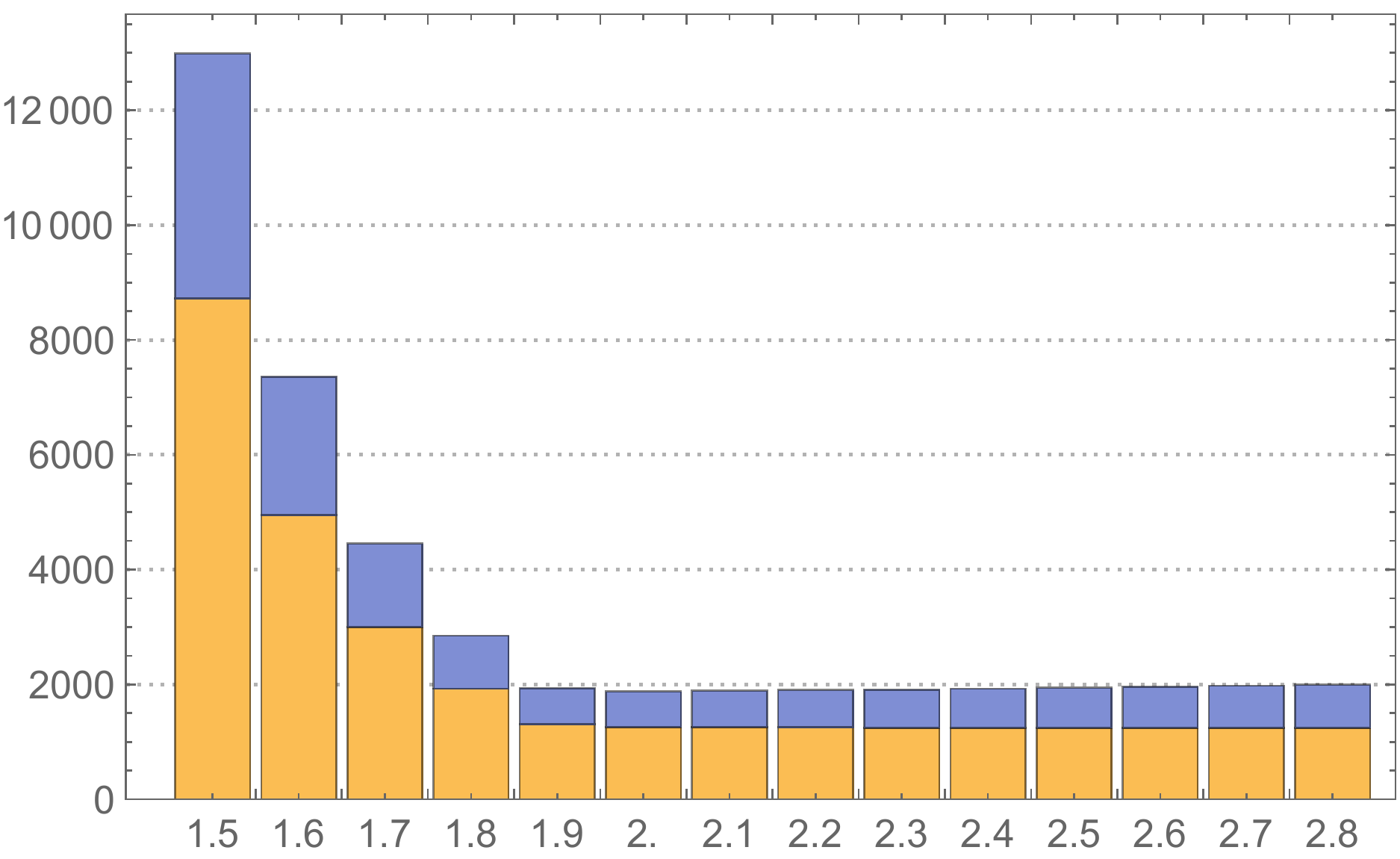}
\hspace{0.05\textwidth}
\includegraphics[width=0.32\textwidth]{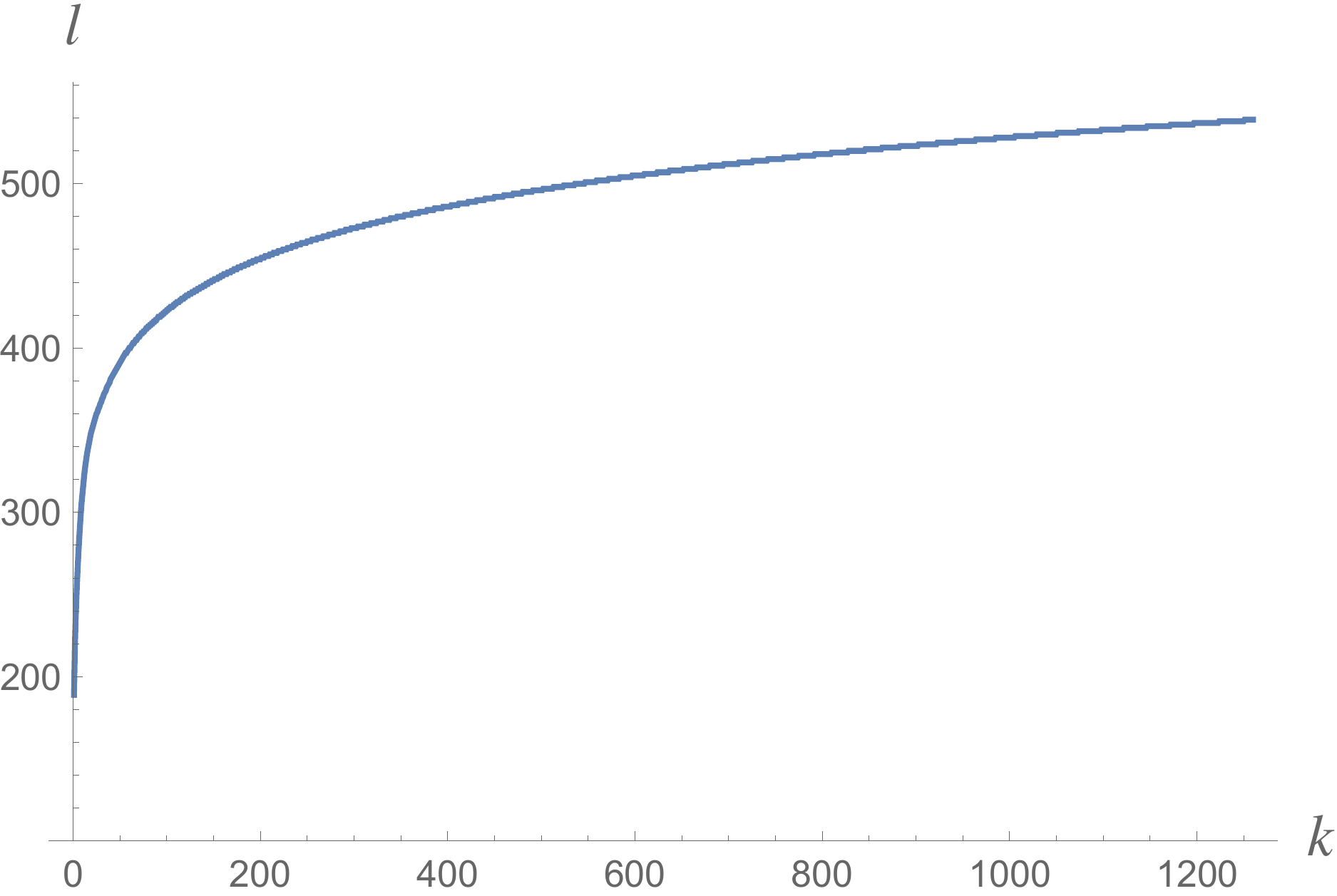}
}
\caption{
The inexact accelerated FBS algorithm for the \textit{unconstrained} problem \eqref{eq:hu-hc}. 
\textsc{Left:} Relation of the number of outer iterations (not scaled down) and the \textit{total} number of inner iterations (scaled down by $1/1000$) as a function of the exponent $q$ of the error $\veps_{k} = \mc{O}(k^{-q})$ \eqref{eq:ek-decay-rate}. \textsc{Right:} The number $l$ of inner iterations at each outer iteration $k$, for exponents $q=1.8$ (top) and $q=2.0$ (bottom) of \eqref{eq:ek-decay-rate}.
}
\label{fig:OuterInner-Unconstrained}
\end{figure}

\begin{figure}
\centerline{
\includegraphics[width=0.32\textwidth]{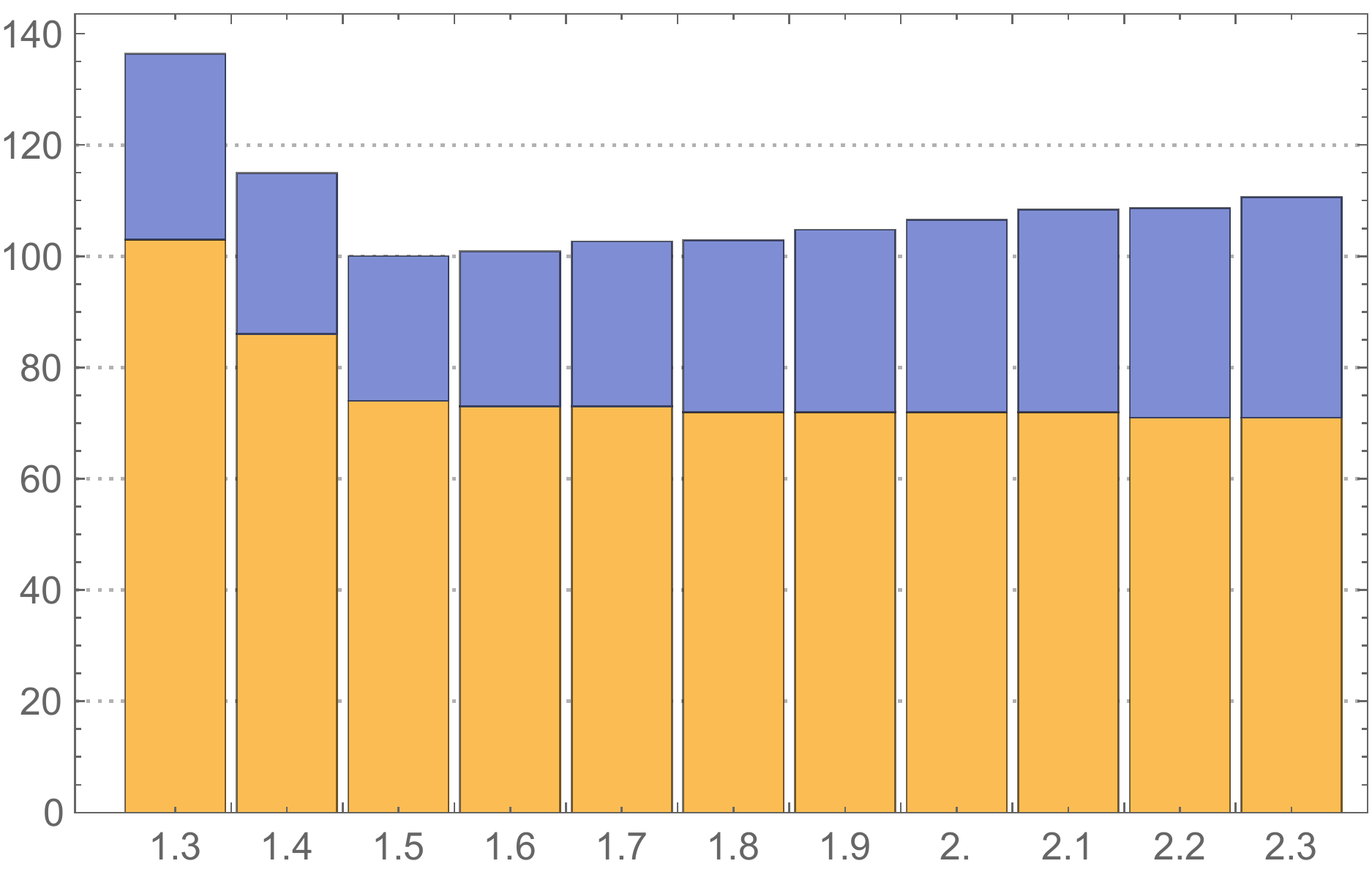}
\hspace{0.05\textwidth}
\includegraphics[width=0.32\textwidth]{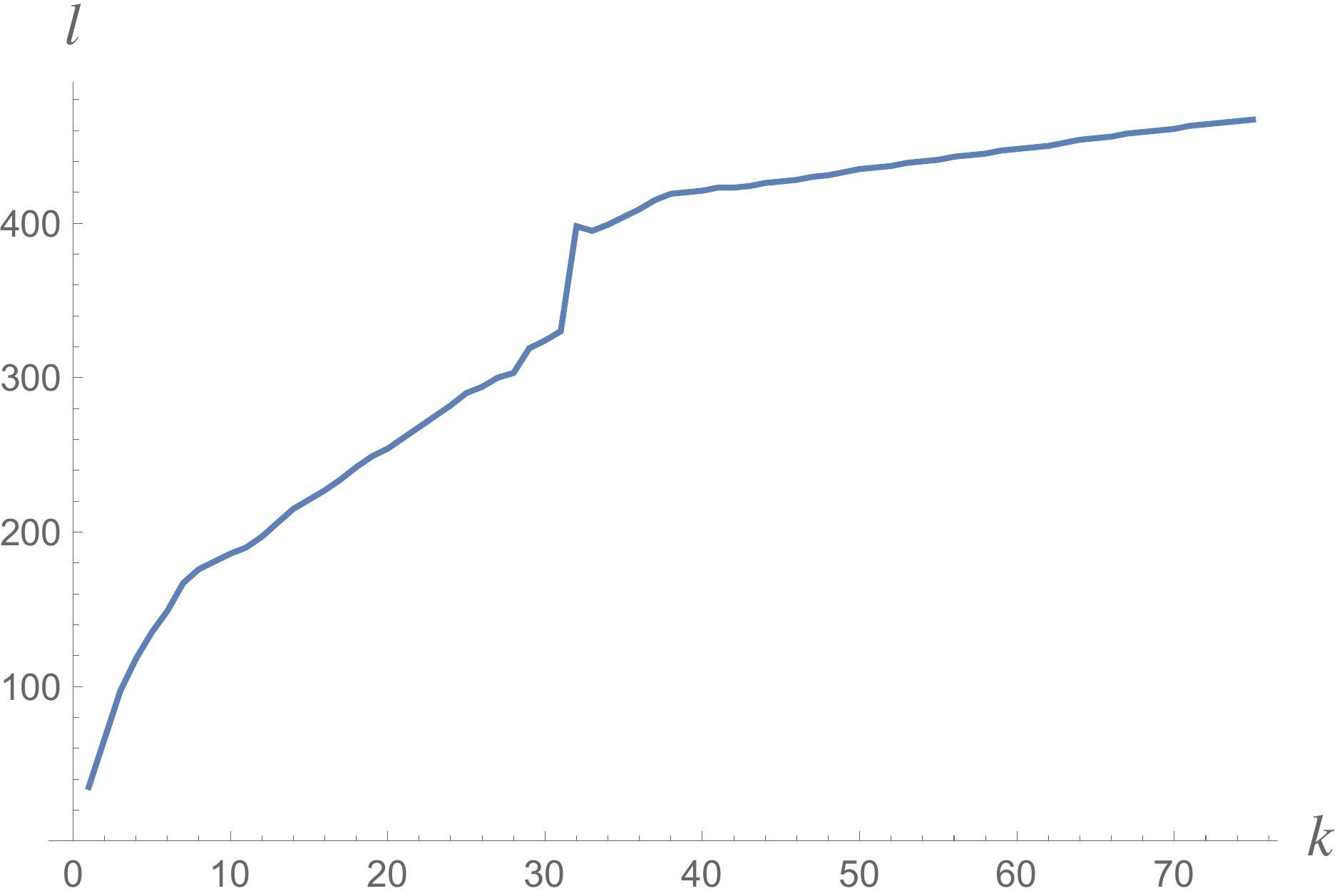}
}
\centerline{
\includegraphics[width=0.32\textwidth]{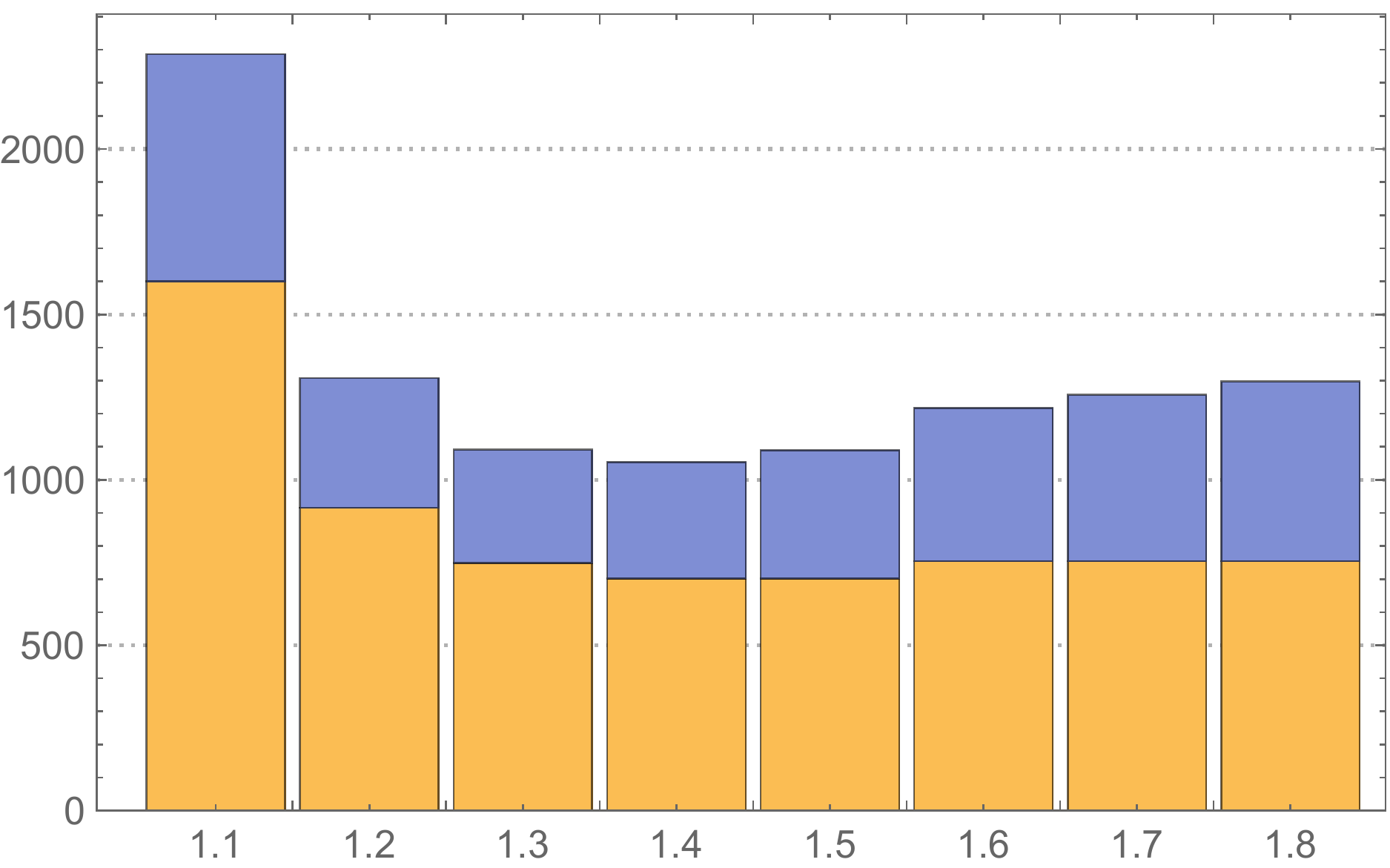}
\hspace{0.05\textwidth}
\includegraphics[width=0.32\textwidth]{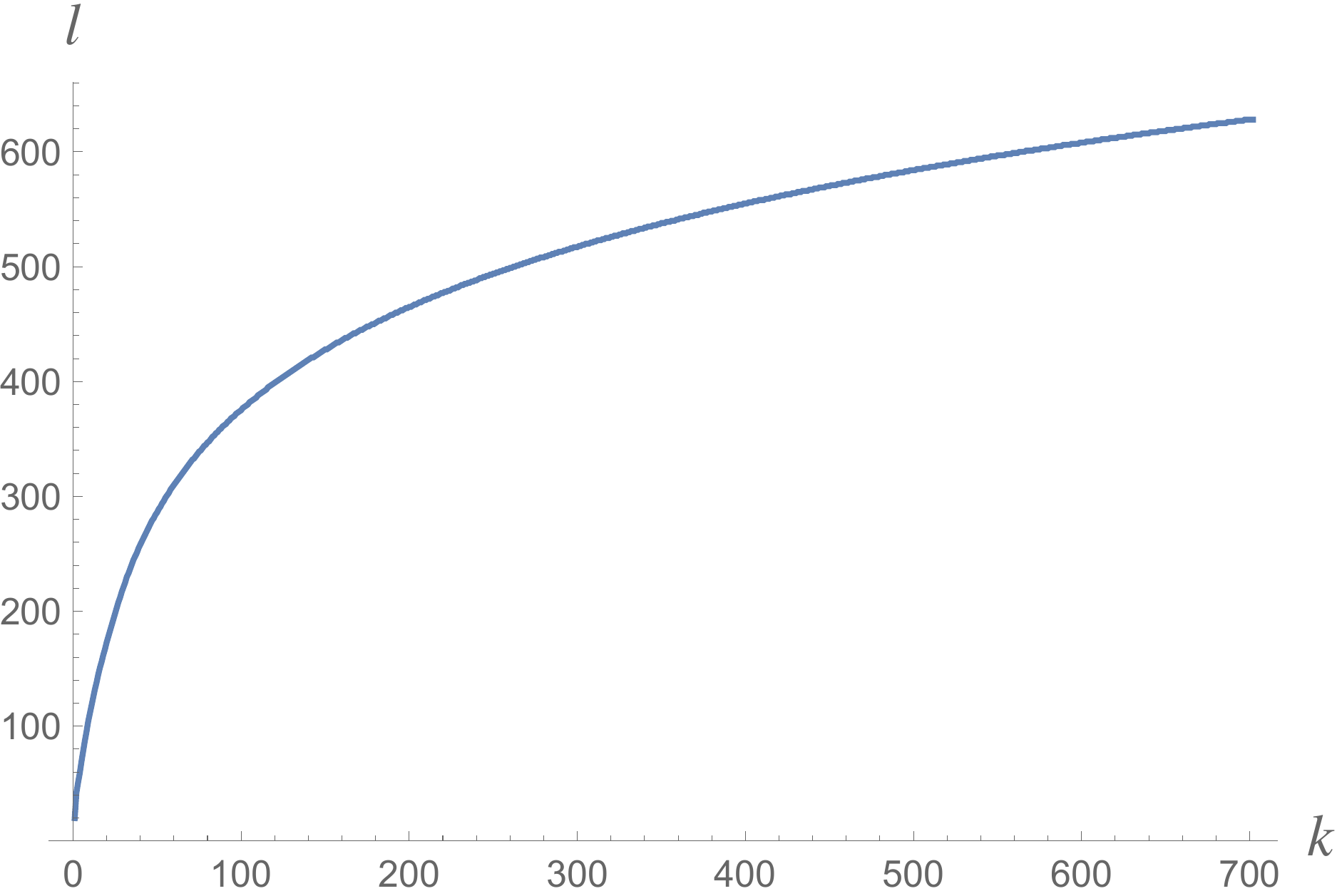}
}
\caption{The 
inexact accelerated FBS algorithm for the \textit{nonnegativity constrained} problem \eqref{eq:hu-hc}. 
\textsc{Left:} Relation of the number of outer iterations (not scaled down) and the \textit{total} number of inner iterations (scaled down by $1/1000$) as a function of  the exponent $q$ of the error $\veps_{k} = \mc{O}(k^{-q})$ \eqref{eq:ek-decay-rate}. \textsc{Right:} The number $l$ of inner iterations at each outer iteration $k$, for exponents $q=1.8$ (top) and $q=2.0$ (bottom) of \eqref{eq:ek-decay-rate}.
}
\label{fig:OuterInner-Constrained}
\end{figure}

\subsection{Discussion}\label{sec:Discussion}
We discuss our observations according to the following five aspects:
\begin{enumerate}[(1)]
\item Observations about superiorization;
\item Observations about optimization;
\item Superiorization vs.~convex optimization: empirical findings;
\item How superiorization could stimulate convex optimization;
\item How optimization could stimulate superiorization.
\end{enumerate}

\begin{enumerate}[(1)]
\item

\begin{enumerate}[(a)]
\item \textit{Superiorization, Figure~\ref{fig:values-superiorization-NC},
Figure~\ref{fig:values-superiorization-WC}: The choice of the basic
algorithm is essential.} The superiorized versions of the CG algorithm
outperforms in our experimental numerical work the superiorized versions
of the Landweber algorithm. The superiorized versions of the Landweber
algorithm makes slow progress towards $x_{\ast}$ due to the bad conditioning
of the matrix $A$ and the corresponding small step-size parameters.
\item \textit{Superiorization, Figure~\ref{fig:values-superiorization-NC},
Figure~\ref{fig:values-superiorization-WC}: Incorporating nonnegativity
constraints in the superiorization method slows down the iteration
process.} Compare the three reconstruction errors in Figure~\ref{fig:values-superiorization-NC}
to those values in Figure~\ref{fig:values-superiorization-WC}. Note
the disparity in the scale of the x-axis between Figures \ref{fig:values-superiorization-NC}
and \ref{fig:values-superiorization-WC}. Since the proximal-point
based superiorized version of the projected Landweber algorithm (\texttt{ProxSupProjLW})
behaves identically to the superiorized version of the Landweber algorithm,
which includes nonnegativity constraints via the proximal map (\texttt{ProxCSubLW}),
we conclude that this is not an artefact of the proposed method of
computing nonascent directions but rather a characteristic of projection-based
algorithms. We conjecture that a geometric basic algorithm that smoothly
evolves on a manifold defined by the constraints, in connection with
a proximal point that employs an appropriate Bregman distance, can
remedy this situation. 
\end{enumerate}
\item
\begin{enumerate}[(a)]
\item \textit{Optimization, Figure~\ref{fig:values-optimization-SPG-NC},
Figure~\ref{fig:values-optimization-SPG-WC}: Incorporating nonnegativity
constraints in the FBS algorithm for the reversed splitting \eqref{eq:splitting-c-reversed}
also slows down the algorithm. On the other hand, the accelerated
FBS iteration in Algorithm \ref{alg:FBS} does not suffer from this
shortcoming.}
\item \textit{Optimization, Figure~\ref{fig:values-optimization}: Acceleration
is very effective at little additional cost provided proximal maps
are computed exactly.} About 50 and 25 outer iterative steps are merely
needed to terminate in the noise-free and noisy case, respectively
(the latter termination criterion - just reaching the noise level
- is more loose). These small numbers of outer iterations result from
computing exact proximal maps.
\item 
\textit{Unconstrained optimization, Figure \ref{fig:OuterInner-Unconstrained}: Inexact inner loops considerably increase the number of iterations.} 
In comparison to the results shown by Figure \ref{fig:values-optimization}, where termination of the accelerated FBS algorithm using \textit{exact} proximal maps was reached after $\leq 75$ iterations, Figure \ref{fig:OuterInner-Unconstrained} shows that the number of outer iterations increases to about 150 and 1200 in the case of noise-free and noisy data, respectively, when \textit{inexact} evaluations of the proximal map are used. This contrast with the results shown by Figure \ref{fig:values-optimization} (exact proximal maps), where the difference between noise-free and noisy data is insignificant.

Each inexact evaluation of the proximal map requires on average 130 and 450 primal-dual iterations when using noise-free and noisy data, respectively, which are computationally inexpensive. The two panels on left-hand side of Figure \ref{fig:OuterInner-Unconstrained} reveal that the efficiency of \textit{acceleration} of the outer FBS iteration requires to choose the exponent $q$ of the error $\veps_{k} = \mc{O}(k^{-q})$ not smaller than a problem-dependent critical value (cf.~the caption of Figure \ref{fig:OuterInner-Unconstrained}).

\item \textit{nonnegativity constrained optimization, Figure~\ref{fig:OuterInner-Constrained}:
The number of outer iterations decreases relative to Figure \ref{fig:OuterInner-Unconstrained}
(so constraints help), but the number of inner iterations increases.}
Otherwise, the observations regarding Figure \ref{fig:OuterInner-Unconstrained}
hold.
\end{enumerate}
\item
\begin{enumerate}[(a)]
\item \textit{Superiorization vs. optimization: Convergence.} Each optimization
algorithm is guaranteed to reduce the combined recovery errors (1)
and (2), listed in the beginning of Subsection \ref{sec:Error-Values}.
The sub-sampling ratio $m/n$ is chosen according to \cite{Kuske2019}
so that recovery via \eqref{eq:hu-hc} is stable. As a consequence,
also the recovery error (3) in Subsection \ref{sec:Error-Values}
is reduced by the optimization algorithm. While the superiorized versions
of the basic algorithms are guaranteed to reduce only the, above mentioned,
recovery error in (1), we observe in Figure~\ref{fig:values-superiorization-NC}
and Figure~\ref{fig:values-superiorization-WC} that also recovery
errors (2) and (3) are decreased to levels comparable to those achieved
by optimization. This favorable performance of superiorization requires
some parameter tuning as described in Subsection \ref{sec:Error-Values}.
\item 
\textit{Superiorization vs. optimization: Best performance is achieved
by the superiorized version of the CG algorithm and by the accelerated
FBS algorithm for the reversed splitting.} The evaluation of the implemented
algorithms shows that, in the unconstrained case, both the proximal-point
based superiorized \textit{version of the CG algorithm} and the accelerated
FBS algorithm for the reversed splitting \eqref{eq:splitting-c-reversed}
reach almost optimal error term values within the fewest number of
iterations at low computational cost (4 vs. 2 matrix-vector evaluations
and 1 proximal-point evaluation of $R_{\tau}$). In the constrained
case, the accelerated FBS iteration, Algorithm \ref{alg:FBS}, applied
for the reversed splitting \eqref{eq:splitting-c-reversed} performs
best while keeping low the computational cost.
\item 
\textit{Superiorization vs. optimization: We addressed the balancing
question by the two alternative splittings.} By analyzing and evaluating
the FBS algorithm for the two alternative splittings, we also addressed
the so-called \emph{balancing question} in the SM: How to distribute
the efforts that a superiorized versions of a basic algorithm invests
in target function reduction steps versus the efforts invested in
the basic algorithm iterative steps? In the case of the splitting
\eqref{eq:splitting-c-SM} that results in the inexact (accelerated)
FBS algorithm, we performed \emph{many basic algorithm iteration steps},
e.g., via Algorithm \ref{alg:PD-No-Inversion}, and just \emph{one
target function reduction step} with respect to $R_{\tau}$. By contrast,
in the case of the reversed splitting \eqref{eq:splitting-c-reversed},
we performed \emph{one basic algorithm iteration} and \emph{many target
function reduction steps}, to compute the proximal point in \eqref{eq:intro-FBS}.
Our results show that the \emph{accelerated} reversed FBS
is computationally more efficient. They suggest that one basic iteration
should be followed by many target function reduction steps for our
considered scenario.
\end{enumerate}

\item \textit{There is a deeper relation between the balancing problem of
the SM and the splitting problem of optimization.} Due to the above
comment (3)(b), the number of steps of the basic algorithm relative
to the number of steps for target function reduction is crucial for
the performance of superiorization. For our considered problem, a
relatively larger number of target function reduction steps seems
to pay. This observation agrees with the better performance of the
\textit{reversed} splitting for optimization, since then inexact evaluations
of the proximal map entail a larger relative number of target function
reduction steps.
\item \textit{Superiorization vs. optimization: Future work.} Observation
(4) above motivates us to consider in future research an inexact Douglas-Rachford
(DR) splitting that includes two proximal maps: one corresponding
to the least-squares term of \eqref{eq:hu-hc} and one corresponding
to the regularizer $R_{\tau}$ that might include nonnegativity constraints
as well. Such an inexact DR splitting would not only be structurally
close to the superiorization of a basic algorithm for least-squares,
but also provide a foundation for a theoretical investigation of the
SM.
\end{enumerate}

\section{Conclusion}

\label{sec:Conclusion}

We considered the underdetermined nonnegative least-squares problem,
regularized by total variation, and compared, for its solution, superiorization
with a state-of-the-art convex optimization approach, namely, accelerated
forward-backward (FB) splitting with inexact evaluations of the corresponding
proximal mapping. The distinction of the basic algorithm and target
function reduction by the SM approach motivated us to contrast superiorization
with an FB splitting, such that the more expensive backward part corresponds
structurally to the basic algorithm, whereas the forward part takes
into account the target function. However, in view of the balancing
problem of the SM, we also considered the \textit{reverse} FB splitting,
since exchanging the forward and backward parts in connection with
inexact evaluations of the proximal mapping is structurally closer
to superiorization.

Our experiments showed that superiorization outperforms convex optimization
\textit{without} acceleration, after proper parameter tuning. Convex
optimization \textit{with} acceleration, however, is on a par with
superiorization or even more efficient when using the \textit{reverse}
splitting. This raises the natural question: How can \textit{accelerated}
superiorization be defined for basic algorithms and target function
reduction procedures?

Superiorization performs best when the number of iterative steps used
for the basic algorithm and for target function reduction, respectively,
are balanced properly. Our results indicate a deeper relationship
of the balancing problem of superiorization and the splitting problem
of convex optimization. We suggested a splitting approach (point (5)
in Subsection 5.4 above) that deems most promising to us for further
closing the theoretical gap between superiorization and optimization
of composite convex problems.

Finally, we point out that we restricted our study to an overall \textit{convex}
problem, which enabled us to apply an advanced optimization scheme
that combines acceleration with inexact evaluations. While superiorization
has proven to be useful for \textit{nonconvex} problems as well, more
research is required to relate both methodologies to each other in
the nonconvex case.

\section*{Acknowledgements} 
We are grateful to the three anonymous reviewers for their constructive comments that helped us improve the paper.
The work of Y. C. is supported by the ISF-NSFC joint research program grant
No. $2874/19$. 

\bibliographystyle{amsalpha}

\newcommand{\etalchar}[1]{$^{#1}$}
\providecommand{\bysame}{\leavevmode\hbox to3em{\hrulefill}\thinspace}
\providecommand{\MR}{\relax\ifhmode\unskip\space\fi MR }
\providecommand{\MRhref}[2]{%
  \href{http://www.ams.org/mathscinet-getitem?mr=#1}{#2}
}
\providecommand{\href}[2]{#2}

\end{document}